\documentclass[12pt]{amsart}
\usepackage[letterpaper,margin=1.2in]{geometry}
\usepackage[utf8]{inputenc}

    \usepackage{amsmath, amsthm, amsfonts, amssymb}
    \usepackage{mathrsfs}
    \usepackage{cancel}

    \usepackage[table]{xcolor}
    \usepackage{multirow}
    \usepackage{subfigure}
    \usepackage{tikz}
    \usepackage{changepage}
    
    \usepackage{enumerate}
    \usepackage{listings}
    
    \usepackage{color}

    \usepackage{hyperref}
    \usepackage{comment}
    
    \theoremstyle{plain}
    \newtheorem{theorem}{Theorem}[section]
    \newtheorem{lemma}[theorem]{Lemma}
    \newtheorem{proposition}[theorem]{Proposition}
    
    \newtheorem{conjecture}{Conjecture}
    
    \newtheorem{corollary}[theorem]{Corollary}
    \newtheorem{example}[theorem]{Example}

    \theoremstyle{definition}

    \newcommand{\ee}{\mathcal{E}}
    
    \newcommand{\RR}{\mathbb{R}}
    \newcommand{\ii}{\mathcal{I}}
    \newcommand{\nn}{\mathcal{N}}
    \newcommand{\pp}{\mathcal{P}}
    \newcommand{\qq}{\mathcal{Q}}
    \newcommand{\uu}{\mathcal{U}}
    \newcommand{\vv}{\mathcal{V}}
    \newcommand{\ww}{\mathcal{W}}
    
    \newcommand{\wwhat}{\hat{\mathcal{W}}}
    \newcommand{\elltwo}{\mathscr{L}^2[0,1]}

    \newcommand{\sset}{\mathcal{S}_{17}}

    \newcommand{\SPR}{\text{SPR}}
    \newcommand{\SPRS}{\SPR_S}

    \newcommand{\tf}{\tilde{f}}
    \newcommand{\tg}{\tilde{g}}
    \newcommand{\tilh}{\tilde{h}}
    
    \newcommand{\tW}{\tilde{W}}
    \newcommand{\hatW}{\hat{\mathcal{W}}}

    \newcommand\restr[2]{{
    \left.\kern-\nulldelimiterspace 
    #1 
    \vphantom{\big|} 
    \right|_{#2} 
    }}

    \newcommand{\av}{\text{av}}
    
    \newcommand{\diam}{\text{diam}}
    \newcommand{\spr}{\text{spr}}

    \newcommand{\black}{\cellcolor{black}}
    \newcommand{\redcell}{\cellcolor{red}}

    \newcommand{\x}{\mathbf{x}}
    \newcommand{\z}{\mathbf{z}}
    \newcommand{\cutdist}{\delta_\square}
    \newcommand\norm[1]{\left\lVert#1\right\rVert}

    \newcommand{\blue}[1]{\textcolor{blue}{#1}}

\title{Maximum spread of graphs and bipartite graphs}
\author{Jane Breen} 
\address{Ontario Tech University, Oshawa, ON, Canada.}
\email{jane.breen@ontariotechu.ca}

\author{Alex W.~N.~Riasanovksy} 
\address{Karlsruhe Institute of Technology, Karlsruhe, Germany.}
\email{alexander.riasanovsky@kit.edu}

\author{Michael Tait} 
\address{Department of Mathematics \& Statistics, Villanova University, Villanova, PA, USA.}
\email{michael.tait@villanova.edu}

\author{John Urschel} 
\address{Department of Mathematics, Massachusetts Institute of Technology, Cambridge, MA, USA. }
\email{urschel@mit.edu}
\address{School of Mathematics, Institute for Advanced Study, Princeton, NJ, USA.}
\email{jcurschel@ias.edu}

\date{September 2021}

\begin{document}

\maketitle

\begin{abstract}
    
    Given any graph $G$, the (adjacency) spread of $G$ is the maximum absolute difference between any two eigenvalues of the adjacency matrix of $G$. In this paper, we resolve a pair of 20-year-old conjectures of Gregory, Hershkowitz, and Kirkland regarding the spread of graphs. The first states that for all positive integers $n$, the $n$-vertex graph $G$ that maximizes spread is the join of a clique and an independent set, with $\lfloor 2n/3 \rfloor$ and $\lceil n/3 \rceil$ vertices, respectively.  
    Using techniques from the theory of graph limits and numerical analysis, we prove this claim for all $n$ sufficiently large.  
    As an intermediate step, we prove an analogous result for a family of operators in the Hilbert space over $\elltwo$.  
    The second conjecture claims that for any fixed $e\leq n^2/4$, if $G$ maximizes spread over all $n$-vertex graphs with $e$ edges, then $G$ is bipartite. We prove an asymptotic version of this conjecture. Furthermore, we exhibit an infinite family of counterexamples, which shows that our asymptotic solution is tight up to lower order error terms.

\end{abstract}

\section{Introduction}

The spread $s(M)$ of an arbitrary $n\times n$ complex matrix $M$ is the diameter of its spectrum; that is,
\[s(M):= \max_{i,j} |\lambda_i-\lambda_j|,\]
where the maximum is taken over all pairs of eigenvalues of $M$. This quantity has been well studied in general, see \cite{deutsch1978spread,johnson1985lower,mirsky1956spread,wu2012upper} for details and additional references. Most notably, Johnson, Kumar, and Wolkowitz produced the lower bound
$$ s(M) \ge \textstyle{ \big| \sum_{i \ne j} m_{i,j} \big|/(n-1)}$$
for normal matrices $M = (m_{i,j})$ \cite[Theorem 2.1]{johnson1985lower}, and Mirsky produced the upper bound
$$ s(M) \le \sqrt{\textstyle{2 \sum_{i,j} |m_{i,j}|^2 - (2/n)\big| \sum_{i} m_{i,i} \big|^2}}$$
for any $n$ by $n$ matrix $M$, which is tight for normal matrices with $n-2$ of its eigenvalues all equal and equal to the arithmetic mean of the other two \cite[Theorem 2]{mirsky1956spread}.

The spread of a matrix has also received interest in certain particular cases. Consider a simple undirected graph $G = (V(G),E(G))$ of order $n$. The adjacency matrix $A$ of a graph $G$ is the $n \times n$ matrix whose rows and columns are indexed by the vertices of $G$, with entries satisfying $A_{u,v} = 1$ if $\{u,v\} \in E(G)$ and $A_{u,v} = 0$ otherwise. This matrix is real and symmetric, and so its eigenvalues are real, and can be ordered $\lambda_1(G) \geq \lambda_2(G)\geq \cdots \geq \lambda_n(G)$. When considering the spread of the adjacency matrix $A$ of some graph $G$, the spread is simply the distance between $\lambda_1(G)$ and $\lambda_n(G)$, denoted by
$$s(G) := \lambda_1(G) - \lambda_n(G).$$
In this instance, $s(G)$ is referred to as the \emph{spread of the graph}.

% {\color{red} Let $\mathcal{G}_n$ be the set of all simple undirected graphs of order $n$; that is, graphs with a vertex set $V(G)$ of order $n$. The adjacency matrix $A$ of a graph $G\in \mathcal{G}_n$ is the matrix whose rows and columns are indexed by the vertices of $G$, with entries defined as 
% \[A_{u,v} = \left\{\begin{array}{cc}1 & \mbox{if }u\sim v;\\
% 0 & \mbox{if } u\nsim v.\end{array}\right.\]
% Since $A$ is a real symmetric matrix, its eigenvalues are real, and can be ordered $\lambda_1(G) \geq \lambda_2(G)\geq \cdots \geq \lambda_n(G)$.

% The spread $s(M)$ of an arbitrary $n\times n$ complex matrix $M$ is given by the diameter of its spectrum; that is,
% \[s(M):= \max_{i,j} |\lambda_i-\lambda_j|,\]
% where the maximum is taken over all pairs of eigenvalues of $M$. When the matrix is the adjacency matrix of some graph $G$, the spread is simply the distance between $\lambda_1(G)$ and $\lambda_n(G)$, denoted by
% $$s(G) := \lambda_1(G) - \lambda_n(G).$$
% In this instance, $s(G)$ is referred to as the \emph{spread of the graph}.}

In \cite{gregory2001spread}, the authors investigated a number of properties regarding the spread of a graph, determining upper and lower bounds on $s(G)$. Furthermore, they made two key conjectures. Let us denote the maximum spread over all $n$ vertex graphs by $s(n)$, the maximum spread over all $n$ vertex graphs of size $e$ by $s(n,e)$, and the maximum spread over all $n$ vertex bipartite graphs of size $e$ by $s_b(n,e)$. Let $K_k$ be the clique of order $k$ and $G(n,k) := K_k \vee \overline{K_{n-k}}$ be the join of the clique $K_k$ and the independent set $\overline{K_{n-k}}$. We say a graph is \emph{spread-extremal} if it has spread $s(n)$. The conjectures addressed in this article are as follows.

\begin{conjecture}[\cite{gregory2001spread}, Conjecture 1.3]\label{conj:spread}
For any positive integer $n$, the graph of order $n$ with maximum spread is $G(n,\lfloor 2n/3 \rfloor)$; that is, $s(n)$ is attained only by $G(n,\lfloor 2n/3 \rfloor)$.
\end{conjecture}

\begin{conjecture}[\cite{gregory2001spread}, Conjecture 1.4]\label{conj:bispread}
If $G$ is a graph with $n$ vertices and $e$ edges attaining the maximum spread $s(n,e)$, and if $e\leq \lfloor n^2/4\rfloor$, then $G$ must be bipartite. That is, $s_b(n,e) = s(n,e)$ for all $e \le \lfloor n^2/4\rfloor$.
\end{conjecture}

Conjecture \ref{conj:spread} is referred to as the Spread Conjecture, and Conjecture \ref{conj:bispread} is referred to as the Bipartite Spread Conjecture. Much of what is known about Conjecture \ref{conj:spread} is contained in \cite{gregory2001spread}, but the reader may also see \cite{StanicBook} for a description of the problem and references to other work on it. In this paper, we resolve both conjectures. We prove the Spread Conjecture for all $n$ sufficiently large, prove an asymptotic version of the Bipartite Spread Conjecture, and provide an infinite family of counterexamples to illustrate that our asymptotic version is as tight as possible, up to lower order error terms. These results are given by Theorems \ref{thm: spread maximum graphs} and \ref{thm: bipartite spread theorem}.

\begin{theorem}\label{thm: spread maximum graphs}
    There exists a constant $N$ so that the following holds:  
    Suppose $G$ is a graph on $n\geq N$ vertices with maximum spread;
    then $G$ is the join of a clique on $\lfloor 2n/3\rfloor$ vertices and an independent set on $\lceil n/3\rceil$ vertices.   
\end{theorem}

\begin{theorem}\label{thm: bipartite spread theorem}
$$s(n,e) - s_b(n,e) \le \frac{1+16 e^{-3/4}}{e^{3/4}} s(n,e)$$
for all $n,e \in \mathbb{N}$ satisfying $e \le \lfloor n^2/4\rfloor$. In addition, for any $\varepsilon>0$, there exists some $n_\varepsilon$ such that
$$s(n,e) - s_b(n,e) \ge  \frac{1-\varepsilon}{e^{3/4}} s(n,e)$$
for all $n\ge n_\varepsilon$ and some $e \le \lfloor n^2/4\rfloor$ depending on $n$.
\end{theorem}

The proof of Theorem \ref{thm: spread maximum graphs} is quite involved, and constitutes the main subject of this work. The general technique consists of showing that a spread-extremal graph has certain desirable properties, considering and solving an analogous problem for graph limits, and then using this result to say something about the Spread Conjecture for sufficiently large $n$. For the interested reader, we state the analogous graph limit result in the language of functional analysis.

\begin{theorem}\label{thm: functional analysis spread}
    Let $W:[0,1]^2\to [0,1]$ be a Lebesgue-measurable function such that $W(x,y) = W(y,x)$ for a.e. $(x,y)\in [0,1]^2$ and let $A = A_W$ be the kernel operator on $\elltwo$ associated to $W$. For all unit functions $f,g\in\elltwo$, 
    \begin{align*}
        \langle f, Af\rangle - 
        \langle g, Ag\rangle
        &\leq 
            \dfrac{2}{\sqrt{3}}.  
    \end{align*}
    Moreover, equality holds if and only if there exists a measure-preserving transformation $\sigma$ on $[0,1]$ such that for a.e. $(x,y)\in [0,1]^2$,
    \begin{align*}
        W(\sigma(x),\sigma(y))
        &=
        \left\{\begin{array}{rl}
            0, & (x,y)\in [2/3, 1]\times [2/3, 1]\\
            1, &\text{otherwise}
        \end{array}\right.
        .  
    \end{align*}
\end{theorem}

The proof of Theorem \ref{thm: spread maximum graphs} can be found in Sections 2-6, with certain technical details reserved for the Appendix. We provide an in-depth overview of the proof of Theorem \ref{thm: spread maximum graphs} in Subsection \ref{sub: outline}. In comparison, the proof of Theorem \ref{thm: bipartite spread theorem} is surprisingly short, making use of the theory of equitable decompositions and a well-chosen class of counter-examples. The proof of Theorem \ref{thm: bipartite spread theorem} can be found in Section \ref{sec:bispread}. Finally, in Section \ref{sec: conclusion}, we discuss further questions and possible future avenues of research.

\subsection{High-Level Outline of Spread Proof}\label{sub: outline}

Here, we provide a concise, high-level description of our asymptotic proof of the Spread Conjecture. The proof itself is quite involved, making use of interval arithmetic and a number of fairly complicated symbolic calculations, but conceptually, is quite intuitive. Our proof consists of four main steps. \\

\noindent{\bf Step 1: } Graph-Theoretic Results \\
\begin{adjustwidth}{1.5em}{0pt}
In Section \ref{sec:graphs}, we observe a number of important structural properties of any graph that maximizes the spread for a given order $n$. In particular, we show that\\
\begin{itemize}
    \item any graph that maximizes spread must be the join of two threshold graphs (Lemma \ref{lem: graph join}),
    \item both graphs in this join have order linear in $n$ (Lemma \ref{linear size parts}),
    \item the unit eigenvectors $\mathbf{x}$ and $\mathbf{z}$ corresponding to $\lambda_1(A)$ and $\lambda_n(A)$ have infinity norms of order $n^{-1/2}$ (Lemma \ref{upper bound on eigenvector entries}),
    \item the quantities $\lambda_1 \mathbf{x}_u^2 - \lambda_n \mathbf{z}_u^2$, $u \in V$, are all nearly equal, up to a term of order $n^{-1}$ (Lemma \ref{discrete ellipse equation}).\\
\end{itemize}
This last structural property serves as the backbone of our proof. In addition, we note that, by a tensor argument, an asymptotic upper bound for $s(n)$ implies a bound for all $n$. \\

% We note that any such graph must be the join of two threshold graphs (Lemma \ref{lem: graph join}), a class of graphs with a special edge structure. Next, we show that both graphs in this join have order linear in $n$ (Lemma \ref{linear size parts}), and that the unit eigenvectors corresponding to both $\lambda_1(A)$ and $\lambda_n(A)$ have infinity norms of order $n^{-1/2}$ (Lemma \ref{upper bound on eigenvector entries}). Using these results, we prove that the quantities $\lambda_1 x_u^2 - \lambda_n z_u^2$, $x$ are nearly equal (up to a $C/n$ term) for all vertices $u$ (Lemma \ref{discrete ellipse equation}). This structural result serves as the backbone of our proof. \\
\end{adjustwidth}

\noindent{\bf Step 2: } Graphons and a Finite-Dimensional Eigenvalue Problem \\

\begin{adjustwidth}{1.5em}{0pt}
In Sections \ref{sec: graphon background} and \ref{sec: graphon spread reduction}, we make use of graphons to understand how spread-extremal graphs behave as $n$ tends to infinity. Section \ref{sec: graphon background} consists of a basic introduction to graphons, and a translation of the graph results of Step 1 to the graphon setting. In particular, we prove the graphon analogue of the graph properties that \\
\begin{itemize}
\item vertices $u$ and $v$ are adjacent if and only if $\mathbf{x}_u \mathbf{x}_v - \mathbf{z}_u \mathbf{z}_v >0$ (Lemma \ref{lem: K = indicator function}),
\item the quantities $\lambda_1 \mathbf{x}_u^2 - \lambda_n \mathbf{z}_u^2$, $u \in V$, are all nearly equal (Lemma \ref{lem: local eigenfunction equation}). \\
\end{itemize}
Next, in Section \ref{sec: graphon spread reduction}, we show that the spread-extremal graphon for our problem takes the form of a particular stepgraphon with a finite number of blocks (Theorem \ref{thm: reduction to stepgraphon}). In particular, through an averaging argument, we note that the spread-extremal graphon takes the form of a stepgraphon with a fixed structure of symmetric seven by seven blocks, illustrated below.
    \begin{align*}
    	\begin{tabular}{||cccc||ccc||}\hline\hline
            \black & \black & \black & \black & \black & \black & \black \\
            \black & \black & \black &  & \black & \black & \black  \\
            \black & \black &  &  & \black & \black & \black  \\
            \black &  &  &  & \black & \black & \black  \\\hline\hline
            \black & \black & \black & \black & \black & \black &  \\
            \black & \black & \black & \black & \black &  &  \\
            \black & \black & \black & \black &  &  &  \\\hline\hline
\end{tabular}
    \end{align*}
The lengths $\alpha = (\alpha_1,...,\alpha_7)$, $\alpha^T {\bf 1} = 1$, of each row and column in the spread-extremal stepgraphon is unknown. For any choice of lengths $\alpha$, we can associate a $7\times7$ matrix whose spread is identical to that of the associated stepgraphon pictured above. Let $B$ be the $7\times7$ matrix with $B_{i,j}$ equal to the value of the above stepgraphon on block $i,j$, and $D = \text{diag}(\alpha_1,...,\alpha_7)$ be a diagonal matrix with $\alpha$ on the diagonal. Then the matrix $D^{1/2} B D^{1/2}$ has spread equal to the spread of the associated stepgraphon. \\ 
\end{adjustwidth}

\noindent{\bf Step 3: } Computer-Assisted Proof of a Finite-Dimensional Eigenvalue Problem \\

\begin{adjustwidth}{1.5em}{0pt}
In Section \ref{sec:spread_graphon}, we show that the optimizing choice of $\alpha$ is, without loss of generality, given by $\alpha_1 = 2/3$, $\alpha_6 =1/3$, and all other $\alpha_i =0$ (Theorem \ref{thm: spread maximum graphon}). This is exactly the limit of the conjectured spread-extremal graph as $n$ tends to infinity. The proof of this fact is extremely technical, and relies on a computer-assisted proof using both interval arithmetic and symbolic computations. This is the only portion of the proof that requires the use of interval arithmetic. Though not a proof, in Figure 1 we provide intuitive visual justification that this result is true. In this figure, we provide contour plots resulting from numerical computations of the spread of the above matrix for various values of $\alpha$. The numerical results suggest that the $2/3-1/3$ two by two block stepgraphon is indeed optimal. See Figure 1 and the associated caption for details. The actual proof of this fact consists of the following steps: \\
\begin{itemize}
\item we reduce the possible choices of non-zero $\alpha_i$ from $2^7$ to $17$ different cases (Lemma \ref{lem: 19 cases}),
\item using eigenvalue equations, the graphon version of $\lambda_1 \mathbf{x}_u^2 - \lambda_n \mathbf{z}_u^2$ all nearly equal, and interval arithmetic, we prove that, of the $17$ cases, only the cases
\begin{itemize}
    \item $\alpha_1,\alpha_7 \ne 0$
    \item $\alpha_4,\alpha_5,\alpha_7 \ne 0$
\end{itemize}
can produce a spread-extremal stepgraphon (Lemma \ref{lem: 2 feasible sets}),
\item prove that the three by three case cannot be spread-extremal, using basic results from the theory of cubic polynomials and computer-assisted symbolic calculations (Lemma \ref{lem: SPR457}). \\
\end{itemize}
% One of these cases corresponds to a bipartite graph (which can be removed immediately), another to our conjectured block two by two structure, and a third corresponds to a three by three block structure. Using the graphon version of the structural result [ADD REF] above, combined with the eigenvalue and eigenvector equations above, using interval arithmetic we prove that the remaining fourteen cases are not spread-extremal. The block three by three matrix corresponds to a matrix of order three, and is shown to not be spread extremal using some basic results from the theory of cubic polynomials and computer assisted symbolic calculations. 
This proves the the spread-extremal graphon is a two by two stepgraphon that, without loss of generality, takes value zero on the block $[2/3,1]^2$ and one elsewhere (Theorem \ref{thm: spread maximum graphon}). \\
\end{adjustwidth}

\begin{figure}
    \centering
    \subfigure[$\alpha_i \ne 0$ for all $i$]{\includegraphics[width=2.9in,height = 2.5in]{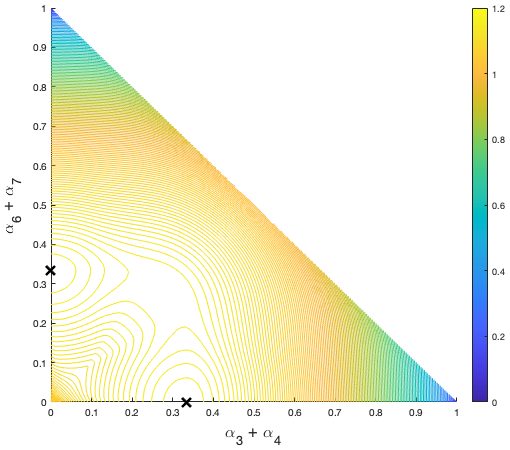}} \quad
    \subfigure[$\alpha_2=\alpha_3=\alpha_4 = 0$]{\includegraphics[width=2.9in,height = 2.5in]{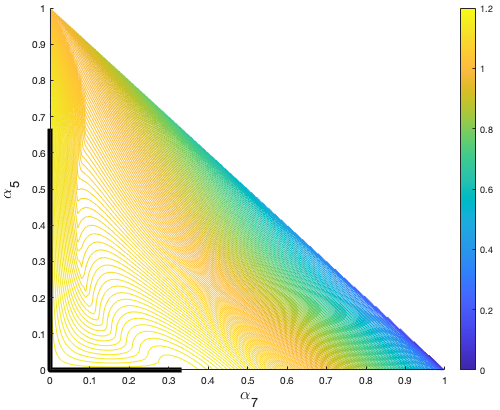}} 
    \caption{Contour plots of the spread for some choices of $\alpha$. Each point $(x,y)$ of Plot (a) illustrates the maximum spread over all choices of $\alpha$ satisfying $\alpha_3 + \alpha_4 = x$ and $\alpha_6 + \alpha_7 = y$ (and therefore, $\alpha_1 + \alpha_2 + \alpha_ 5 = 1 - x - y$) on a grid of step size $1/100$. Each point $(x,y)$ of Plot (b) illustrates the maximum spread over all choices of $\alpha$ satisfying $\alpha_2=\alpha_3=\alpha_4=0$, $\alpha_5 = y$, and $\alpha_7 = x$ on a grid of step size $1/100$. The maximum spread of Plot (a) is achieved at the black x, and implies that, without loss of generality, $\alpha_3 + \alpha_4=0$, and therefore $\alpha_2 = 0$ (indices $\alpha_1$ and $\alpha_2$ can be combined when $\alpha_3 + \alpha_4=0$). Plot (b) treats this case when $\alpha_2 = \alpha_3 = \alpha_4 = 0$, and the maximum spread is achieved on the black line. This implies that either $\alpha_5 =0$ or $\alpha_7 = 0$. In both cases, this reduces to the block two by two case $\alpha_1,\alpha_7 \ne 0$ (or, if $\alpha_7 = 0$, then $\alpha_1,\alpha_6 \ne 0$).}
     \label{fig:contour}
\end{figure}

\noindent{\bf Step 4: } From Graphons to an Asymptotic Proof of the Spread Conjecture \\

\begin{adjustwidth}{1.5em}{0pt}
Finally, in Section \ref{sub-sec: graphons to graphs}, we convert our result for the spread-extremal graphon to a statement for graphs. This process consists of two main parts:\\ 
\begin{itemize}
    \item using our graphon theorem, we show that any spread-extremal graph takes the form $(K_{n_1}\dot{\cup} \overline{K_{n_2}})\vee \overline{K_{n_3}}$ for $n_1 = (2/3+o(1))n$, $n_2 = o(n)$, and $n_3 = (1/3+o(1))n$ (Lemma \ref{lem: few exceptional vertices}), i.e. any spread-extremal graph is equal up to a set of $o(n)$ vertices to the conjectured optimal graph $K_{\lfloor 2n/3\rfloor} \vee \overline{K_{\lceil n/3 \rceil}}$,
    \item we show that, for $n$ sufficiently large, the spread of $(K_{n_1}\dot{\cup} \overline{K_{n_2}})\vee \overline{K_{n_3}}$, $n_1 + n_2 + n_3 = n$, is maximized when $n_2 = 0$ (Lemma \ref{lem: no exceptional vertices}).\\
\end{itemize}
Together, these two results complete our proof of the spread conjecture for sufficiently large $n$ (Theorem \ref{thm: spread maximum graphs}).
% using this graphon result, we prove that, for $n$ sufficiently large, any optimal graph is, up to o(n) vertices, the join of a clique of order $2/3$ and an independent set of order $1/3$. Using the original graph-theoretic result above, we show that the spread optimal graph takes the form $\big[K_{n_1} \cup \bar{K}_{n_2}\big] \vee \bar{K}_{n_3}$ for some $n_1 + n_2 + n_3 = n$. From here, using calculus we show that the spread is maximized when $n_2=0$.
\end{adjustwidth}

\section{Properties of spread-extremal graphs} \label{sec:graphs}

In this section, we review what has already been proven about spread-extremal graphs ($n$ vertex graphs with spread $s(n)$) in \cite{gregory2001spread}, where the original conjectures were made. We then prove a number of properties of spread-extremal graphs and properties of the eigenvectors associated with the maximum and minimum eigenvalues of a spread-extremal graph. 

Let $G$ be a graph, and let $A$ be the adjacency matrix of $G$, with eigenvalues $\lambda_1 \geq \cdots \geq \lambda_n$. For unit vectors $\mathbf{x}$, $\mathbf{y} \in \mathbb{R}^n$, we have 
\[\lambda_1 \geq \mathbf{x}^T A \mathbf{x} \quad \mbox{and} \quad \lambda_n \leq \mathbf{y}^TA\mathbf{y}.\]
Hence (as it is observed in \cite{gregory2001spread}), the spread of a graph can be expressed
\begin{equation}\label{eq: gregory min max}
s(G) = \max_{\mathbf{x}, \mathbf{z}}\sum_{u\sim v} (\mathbf{x}_u\mathbf{x}_v-\mathbf{z}_u\mathbf{z}_v)\end{equation}
where the maximum is taken over all unit vectors $\mathbf{x}, \mathbf{z}$. Furthermore, this maximum is attained only for $\mathbf{x}, \mathbf{z}$ orthonormal eigenvectors corresponding to the eigenvalues $\lambda_1, \lambda_n$, respectively. We refer to such a pair of vectors $\mathbf{x}, \mathbf{z}$ as \emph{extremal eigenvectors} of $G$. For any two vectors $\mathbf{x}$, $\mathbf{z}$ in $\mathbb{R}^n$, let $G(\mathbf{x}, \mathbf{z})$ denote the graph for which distinct vertices $u, v$ are adjacent if and only if $\mathbf{x}_u\mathbf{x}_v-\mathbf{z}_u\mathbf{z}_v\geq 0$. Then from the above, there is some graph $G(\mathbf{x}, \mathbf{z})$ which is a spread-extremal graph,  with $\mathbf{x}$, $\mathbf{z}$ orthonormal and $\mathbf{x}$ positive (\cite[Lemma 3.5]{gregory2001spread}). 

In addition, we enhance \cite[Lemmas 3.4 and 3.5]{gregory2001spread} using some helpful definitions and the language of threshold graphs.  
Whenever $G = G(\x, \z)$ is understood, let $P = P(\x, \z) := \{u \in V(G) : \z_u \geq 0\}$ and $N = N(\x, \z) := V(G)\setminus P$.  
%Under the assumption that $G$ is a spread-extremal graph, then $G=G_1 \vee G_2$, where $G_1 = G[P]$ and $G_2 = G[N]$ are the subgraphs induced on $P$ and $N$.  
%
%Our first result extends Lemma 3.4 in \cite{gregory2001spread} to show that if $G=G_1\vee G_2$ is a spread-extremal graph, then both $G_1$ and $G_2$ are in fact threshold graphs.
%
% \red{
% \begin{definition}
%     A {\em threshold graph} is a graph that can be constructed by starting from a one-vertex graph and repeatedly adding either a single isolated vertex or a single dominating vertex. \\
%     \\
% \end{definition}
% }
For our purposes, we say that $G$ is a \emph{threshold graph} if and only if there exists a function $\varphi:V(G)\to (-\infty,\infty]$ such that for all distinct $u,v\in V(G)$, $uv\in E(G)$ if and only if $\varphi(u)+\varphi(v) \geq 0$
    \footnote{
        Here, we take the usual convention that for all $x\in (-\infty, \infty]$, $\infty + x = x + \infty = \infty$}.  
Here, $\varphi$ is a {\it threshold function} for $G$ (with $0$ as its {\it threshold}). 
The following detailed lemma shows that any spread-extremal graph is the join of two threshold graphs with threshold functions which can be made explicit.  

\begin{lemma}\label{lem: graph join}
    Let $n> 2$ and suppose $G$ is a $n$-vertex graph such that $s(G) = s(n)$.  
    Denote by $\mathbf{x}$ and $\mathbf{z}$ the extremal unit eigenvectors for $G$.  
    Then
    \begin{enumerate}[(i)]
        \item\label{item: matrix 0 or 1} 
        For any two vertices $u,v$ of $G$, $u$ and $v$ are adjacent whenever $\mathbf{x}_u\mathbf{x}_v-\mathbf{z}_u\mathbf{z}_v>0$ and $u$ and $v$ are nonadjacent whenever $\mathbf{x}_u\mathbf{x}_v-\mathbf{z}_u\mathbf{z}_v<0$.
    
        \item\label{item: xx-zz nonzero} 
            For any distinct $u,v\in V(G)$,  $\mathbf{x}_u\mathbf{x}_v-\mathbf{z}_u\mathbf{z}_v\not=0$.
        \item\label{item: graph join}
            Let $P := P(\x, \z)$, $N := N(\x, \z)$ and let $G_1 := G[P]$ and $G_2 := G[N]$.  
            Then $G = G(\x, \z) = G_1\vee G_2$.  
        \item\label{item: G1 G2 thresholds}
            For each $i\in \{1,2\}$, $G_i$ is a threshold graph with threshold function defined on all $u\in V(G_i)$ by 
            \begin{align*}
                \varphi(u) 
                := 
                    \log\left|
                        \dfrac{\x_u}
                        {\z_u}
                    \right|
                .  
            \end{align*}
    \end{enumerate}
\end{lemma}

\begin{proof}

Suppose $G$ is a $n$-vertex graph such that $s(G) = s(n)$ and write $A = (a_{uv})_{u,v\in V(G)}$ for its adjacency matrix.  
Item \eqref{item: matrix 0 or 1} is equivalent to Lemma 3.4 from \cite{gregory2001spread}.  
For completeness, we include a proof.  
By Equation \eqref{eq: gregory min max} we have that 
\begin{align*}
    s(G) &= \max_{ x,z }
        \x^T A\x
        -\z^TA\z
        = 
        \sum_{u,v\in V(G)}
            a_{uv}\cdot 
            \left(\x_u\x_v - \z_u\z_v\right), 
\end{align*}
where the maximum is taken over all unit vectors of length $|V(G)|$.  
If $\x_u\x_v - \z_u\z_v > 0$ and $a_{uv} = 0$, then $s(G+uv) > s(G)$, a contradiction.  
And if $\x_u\x_v - \z_u\z_v < 0$ and $a_{uv} = 1$, then $s(G-uv) > s(G)$, a contradiction.  
So Item \eqref{item: matrix 0 or 1} holds.  
\\

For a proof of Item \eqref{item: xx-zz nonzero} suppose $\x_u\x_v - \z_u\z_v = 0$ and denote by $G'$ the graph formed by adding or deleting the edge $uv$ from $G$.  
With $A' = (a_{uv}')_{u,v\in V(G')}$ denoting the adjacency matrix of $G'$, note that 
\begin{align*}
    s(G')
    % &\geq 
    %     \max_{\x', \z'}
    %     \langle \x', A\x'\rangle
    %     -\langle \z', A\z'\rangle
    \geq 
        \x^T A'\x 
        -\z^TA'\z
    =
        \x^T A\x
        -\z^TA\z
    = 
        s(G)
    &\geq 
        s(G), 
\end{align*}
so each inequality is an equality.  
It follows that $\x, \z$ are eigenvectors for $A'$.  
Furthermore, without loss of generality, we may assume that $uv\in E(G)$.  
In particular, there exists some $\lambda'$ such that 
\begin{align*}
    A\x &= \lambda \x\\
    (A
        -{\bf e}_u{\bf e}_v^T
        -{\bf e}_v{\bf e}_u^T
    )\x
    &= 
        \lambda'\x 
    .  
\end{align*}
So $(
    {\bf e}_u{\bf e}_v^T
    +{\bf e}_v{\bf e}_u^T
)\x = (\lambda - \lambda')\x$.  
Let $w\in V(G)\setminus \{u,v\}$.  
By the above equation, $(\lambda-\lambda')\x_w = 0$ and either $\lambda' = \lambda$ or $\x_w = 0$.  
To find a contradiction, it is sufficient to note that $G$ is a connected graph with Perron-Frobenius eigenvector $\x$.  
Indeed, let $P := \{w\in V(G) : \z_w \geq 0\}$ and let $N := V(G)\setminus P$.  
Then for any $w\in P$ and any $w'\in N$, $\x_w\x_{w'} - \z_w\z_{w'} > 0$ and by Item \eqref{item: matrix 0 or 1}, $ww'\in E(G)$.  
So $G$ is connected and this completes the proof of Item \eqref{item: xx-zz nonzero}.  
\\

Now, we prove Item \eqref{item: graph join}.  
To see that $G = G(\x, \z)$, note by Items \eqref{item: matrix 0 or 1} and \eqref{item: xx-zz nonzero}, for all distinct $u,v\in V(G)$, $\x_u\x_v - \z_u\z_v > 0$ if and only if $uv\in 
E(G)$, and otherwise, $\x_u\x_v - \z_u\z_v < 0$ and $uv\notin E(G)$.  
To see that $G = G_1\vee G_2$, note that for any $u\in P$ and any $v\in N$, $0 \neq \x_u\x_v - \z_u\z_v \geq \z_u\cdot (-\z_v)\geq 0$.  
\\

Finally, we prove Item \eqref{item: G1 G2 thresholds}.  
Suppose $u,v$ are distinct vertices such that either $u,v\in P$ or $u,v\in N$.  
Allowing the possibility that $0\in \{\z_u, \z_v\}$, the following equivalence holds: 
\begin{align*}
    \varphi(u) + \varphi(v) &\geq 0 & \text{ if and only if }\\
    \log\left|
        \dfrac{\x_u\x_v}{\z_u\z_v}
    \right|
    &\geq 1 
    & \text{ if and only if }\\
    \x_u\x_v - |\z_u\z_v| &\geq 0.  
\end{align*}
Since $\z_u,\z_v$ have the same sign, Item \eqref{item: G1 G2 thresholds}.  
This completes the proof.  
\end{proof}

From \cite{thresholdbook}, we recall the following useful characterization in terms of ``nesting'' neighborhoods: 
    $G$ is a threshold graph if and only there exists a numbering $v_1,\cdots, v_n$ of $V(G)$ such that for all $1\leq i<j\leq n$, if $v_k\in V(G)\setminus\{v_i,v_j\}$, $v_jv_k\in E(G)$ implies that $v_iv_k\in E(G)$.  
Given this ordering, if $k$ is the smallest natural number such that $v_kv_{k+1}\in E(G)$ then we have that the set $\{v_1,\cdots, v_k\}$ induces a clique and the set $\{v_{k+1},\cdots, v_n\}$ induces an independent set.  

The next lemma shows that both $P$ and $N$ have linear size.

\begin{lemma}\label{linear size parts}
If $G$ is a spread-extremal graph, then both $P$ and $N$ have size $\Omega(n)$.
\end{lemma}
\begin{proof}
We will show that $P$ and $N$ both have size at least $\frac{n}{100}$. First, since $G$ is spread-extremal it has spread more than $1.1n$ and hence has smallest eigenvalue $\lambda_n < \frac{-n}{10}$. Without loss of generality, for the remainder of this proof we will assume that $|P| \leq |N|$, that $\mathbf{z}$ is normalized to have infinity norm $1$, and that $v$ is a vertex satisfying $| \mathbf{z}_v| = 1$. By way of contradiction, assume that $|P|< \frac{n}{100} $. 

If $v\in N$, then we have 
\[
\lambda_n \mathbf{z}_v = -\lambda_n = \sum_{u\sim v} \mathbf{z}_u \leq \sum_{u\in P} \mathbf{z}_u \leq |P| < \frac{n}{100},
\]
contradicting that $\lambda_n < \frac{-n}{10}$. Therefore, assume that $v\in P$. Then 
\[
\lambda_n^2 \mathbf{z}_v = \lambda_n^2 = \sum_{u\sim v} \sum_{w\sim u}\mathbf{z}_w \leq \sum_{u\sim v} \sum_{\substack{w\sim u\\w\in P}} \mathbf{z}_w \leq |P||N| + 2e(P) \leq |P||N| + |P|^2 \leq \frac{99n^2}{100^2} + \frac{n^2}{100^2}.
\]
This gives $|\lambda_n| \leq \frac{n}{10}$, a contradiction.
\end{proof}

\begin{lemma}\label{upper bound on eigenvector entries}
If $\mathbf{x}$ and $\mathbf{z}$ are unit eigenvectors for $\lambda_1$ and $\lambda_n$, then $\norm{\mathbf{x}}_\infty = O(n^{-1/2})$ and $\norm{\mathbf{z}}_\infty = O(n^{-1/2})$.
\end{lemma}
\begin{proof}
During this proof we will assume that $\hat{u}$ and $\hat{v}$ are vertices satisfying $\norm{\mathbf{x}}_\infty = \mathbf{x}_{\hat{u}}$ and $\norm{\mathbf{z}}_\infty = |\mathbf{z}_{\hat{v}}|$ and without loss of generality that $\hat{v} \in N$. We will use the weak estimates that $\lambda_1 > \frac{n}{2}$ and $\lambda_n < \frac{-n}{10}$. Define sets 
\begin{align*}
    A &= \left\{ w: \mathbf{x}_w > \frac{\mathbf{x}_{\hat{u}}}{4} \right\}\\
    B &= \left\{ w: \mathbf{z}_w > \frac{-\mathbf{z}_{\hat{v}}}{20} \right\}.
\end{align*}
It suffices to show that $A$ and $B$ both have size $\Omega(n)$, for then there exists a constant $\epsilon > 0$ such that
\[
1 = \mathbf{x}^T \mathbf{x} \geq \sum_{w\in A} \mathbf{x}_w^2 \geq |A| \frac{\norm{\mathbf{x}}^2_\infty}{16} \geq \epsilon n \norm{\mathbf{x}}^2_\infty,
\]
and similarly 
\[
1 = \mathbf{z}^T \mathbf{z} \geq \sum_{w\in B} \mathbf{z}_w^2 \geq |B| \frac{\norm{\mathbf{z}}^2_\infty}{400} \geq \epsilon n \norm{\mathbf{z}}^2_\infty.
\]
We now give a lower bound on the sizes of $A$ and $B$ using the eigenvalue-eigenvector equation and the weak bounds on $\lambda_1$ and $\lambda_n$.
\[
\frac{n}{2} \norm{\mathbf{x}}_\infty = \frac{n}{2} \mathbf{x}_{\hat{u}} < \lambda_1 \mathbf{x}_{\hat{u}} = \sum_{w\sim \hat{u}} \mathbf{x}_w \leq \norm{\mathbf{x}}_\infty \left(|A| + \frac{1}{4}(n-|A|) \right),
\]
giving that $|A| > \frac{n}{3}$. Similarly,
\[
\frac{n}{10} \norm{\mathbf{z}}_\infty =  - \frac{n}{10} \mathbf{z}_{\hat{v}} < \lambda_n \mathbf{z}_{\hat{v}} = \sum_{w\sim \hat{v}} \mathbf{z}_w \leq \norm{\mathbf{z}}_\infty \left( |B| + \frac{1}{20}(n-|B|)\right),
\]
and so $|B| > \frac{n}{19}$.

\end{proof}
%{\color{red} Here we will put some version of the Lemma labeled as ``john's lemma" in old section 2 that gives bounds on the eigenvector entries}.

%{\color{red} Discrete ellipse equation}

\begin{lemma}\label{discrete ellipse equation}

Assume that $\mathbf{x}$ and $\mathbf{z}$ are unit vectors. Then there exists a constant $C$ such that for any pair of vertices $u$ and $v$, we have 
\[
|(\lambda_1 \mathbf{x}_u^2 - \lambda_n\mathbf{z}_u^2) - (\lambda_1 \mathbf{z}_v^2 - \lambda_n \mathbf{z}_v^2)| < \frac{C}{n}.
\]
\end{lemma}

\begin{proof}
Let $u$ and $v$ be vertices, and create a graph $\tilde{G}$ by deleting $u$ and cloning $v$. That is, $V(\tilde{G}) = \{v'\} \cup V(G) \setminus \{u\}$ and \[E(\tilde{G}) = E(G\setminus \{u\}) \cup \{v'w:vw\in E(G)\}.\]

Note that $v \not\sim v'$. Let $\tilde{A}$ be the adjacency matrix of $\tilde{G}$. Define two vectors $\mathbf{\tilde{x}}$ and $\mathbf{\tilde{z}}$ by 
\[
\mathbf{\tilde{x}}_w = \begin{cases}
\mathbf{x}_w & w\not=v'\\
\mathbf{x}_v & w=v',
\end{cases}
\]
and
\[
\mathbf{\tilde{z}}_w = \begin{cases}
\mathbf{z}_w & w\not=v'\\
\mathbf{z}_v & w=v.
\end{cases}
\]

Then $\mathbf{\tilde{x}}^T \mathbf{\tilde{x}} = 1 - \mathbf{x}_u^2 + \mathbf{x}_v^2$ and  $\mathbf{\tilde{z}}^T \mathbf{\tilde{z}} = 1 - \mathbf{z}_u^2 + \mathbf{z}_v^2$. Similarly,
\begin{align*}
\mathbf{\tilde{x}}^T\tilde{A}\mathbf{\tilde{x}} &= \lambda_1 - 2\mathbf{x}_u \sum_{uw\in E(G)} \mathbf{x}_w + 2\mathbf{x}_{v'} \sum_{vw \in E(G)} \mathbf{x}_w - 2A_{uv}\mathbf{x}_v\mathbf{x}_u \\
&= \lambda_1 - 2\lambda_1\mathbf{x}_u^2 + 2\lambda_1 \mathbf{x}_v^2 - 2A_{uv} \mathbf{x}_u \mathbf{x}_v,
\end{align*}
and
\begin{align*}
\mathbf{\tilde{z}}^T\tilde{A}\mathbf{\tilde{z}} &= \lambda_n - 2\mathbf{z}_u \sum_{uw\in E(G)} \mathbf{z}_w + 2\mathbf{z}_{v'} \sum_{vw \in E(G)} \mathbf{z}_w - 2A_{uv}\mathbf{z}_v\mathbf{z}_u \\
&= \lambda_n - 2\lambda_n\mathbf{z}_u^2 + 2\lambda_n \mathbf{z}_v^2 - 2A_{uv} \mathbf{z}_u \mathbf{z}_v.
\end{align*}
By Equation \eqref{eq: gregory min max}, 
\begin{align*}
    0 & \geq \left(\frac{\mathbf{\tilde{x}}^T\tilde{A}\mathbf{\tilde{x}}}{\mathbf{\tilde{x}}^T \mathbf{\tilde{x}}} - \frac{\mathbf{\tilde{z}}^T\tilde{A}\mathbf{\tilde{z}}}{\mathbf{\tilde{z}}^T \mathbf{\tilde{z}}}  \right) - (\lambda_1 - \lambda_n) \\
    & = \left(\frac{\lambda_1 - 2\lambda_1 \mathbf{x}_u^2 + 2\lambda_1 \mathbf{x}_v^2 - 2A_{uv}\mathbf{x}_u\mathbf{x}_v}{1 - \mathbf{x}_u^2 + \mathbf{x}_v^2}  - \frac{\lambda_n - 2\lambda_n \mathbf{z}_u^2 + 2\lambda_n \mathbf{z}_v^2 - 2A_{uv} \mathbf{z}_u\mathbf{z}_v}{1-\mathbf{z}_u^2 + \mathbf{z}_v^2}\right) - (\lambda_1 - \lambda_n) \\
    & = \frac{-\lambda_1 \mathbf{x}_u^2 + \lambda_1\mathbf{x}_v^2 - 2A_{ij}\mathbf{x}_u\mathbf{x}_v}{1 - \mathbf{x}_u^2 + \mathbf{x}_v^2} - \frac{-\lambda_n \mathbf{z}_u^2 + \lambda_n\mathbf{z}_v^2 - 2A_{ij}\mathbf{z}_u \mathbf{z}_v}{1 - \mathbf{z}_u^2 + \mathbf{z}_v^2}.
\end{align*}

By Lemma \ref{upper bound on eigenvector entries}, we have that $|\mathbf{x}_u|$, $|\mathbf{x}_v|$, $|\mathbf{z}_u|$, and $|\mathbf{z}_v|$ are all $O(n^{-1/2})$, and so it follows that 
\[
|(\lambda_1 \mathbf{x}_u^2 - \lambda_1\mathbf{x}_v^2) - (\lambda_n \mathbf{z}_u^2 - \lambda_n \mathbf{z}_v^2)| < \frac{C}{n},
\]
for some absolute constant $C$. Rearranging terms gives the desired result. 
\end{proof}

% {\color{red} Next we show that an asymptotic upper bound of $cn + o(n)$ actually implies an upper bound of $cn$ for {\em all} $n$.

% \begin{proposition}\label{thm: tensor trick}

%     Let 
%     \begin{align*}
%         c := 
%         \limsup_{n\to\infty}
%             \dfrac{\spr(n)}
%             n.  
%     \end{align*}
%     Then $\spr(n) \leq c\cdot n$ for all $n\geq 1$.  

% \end{proposition}

% \begin{proof}
% By way of contradiction, assume that there is a graph $G$ on $m$ vertices such that $s(G) = cm + \delta$ for some positive $\delta$. Let $G^t$ be the $t$-blowup of $G$. That is $G^t$ is the graph on $mt$ vertices where each vertex in $G$ is replaced by an independent set of size $t$ and each edge is replaced by a $K_{t,t}$ between the two independent sets. One can check that $\lambda_1(G^t) = t\lambda_1(G)$ and $\lambda_n(G^t) = t\lambda_n(G)$, and so 
% \[s(G^t) = ts(G) = t(cm + \delta) = c(mt) + \frac{\delta}{m} mt.
% \]
% Letting $t$ go to infinity gives a sequence of graphs on $n = mt$ vertices satisfying $s(G) = (c + \delta/m)n$, contradicting the hypothesis.
% \end{proof}
% }

%Finally, we include the following proposition which is used mainly in the proof of Theorem \ref{thm: spread maximum graphs} in Section \ref{sub-sec: graphons to graphs}.  
\section{The spread-extremal problem for graphons}\label{sec: graphon background}
%In this section, we translate the asymptotic version of the maximum spread problem into a graph problem.  
%Mirror Terpai's paper.  
%\begin{enumerate}
 %   \item Take any sequence $(G_n)$ of optimal graphs and find a convergent subsequence converging to some graphon $K$.  
  %  By comparing the eigenvalues of $G_n$ to its associated graphon $W_n$, it follows that the asymptotic problem translates to the graphon problem.  
   % \item Apply Lov\'{a}sz Szegedy theorem which states that $\lambda_k$ is a continuous function of graphons, for all $k$.  
    %Here for every positive $k$, $\lambda_k$ is the $k$-th maximum eigenvalue and %$\lambda_{-k}$ is the $k$-th minimum eigenvalue (both counting multiplicities.  
    %\item Infer from compactness that an optimum exists.  
%\end{enumerate}
Graphons (or graph functions) are analytical objects which may be used to study the limiting behavior of large, dense graphs, and were originally introduced in \cite{BCLSV2012GraphLimitsSpectra} and \cite{lovasz2006limits}.

\subsection{Introduction to graphons}
Consider the set $\mathcal{W}$ of all bounded symmetric measurable functions $W:[0,1]^2 \to [0,1]$ (by symmetric, we mean $W(x, y)=W(y,x)$ for all $(x, y)\in [0,1]^2$. A function $W\in \mathcal{W}$ is called a \emph{stepfunction} if there is a partition of $[0,1]$ into subsets $S_1, S_2, \ldots, S_m$ such that $W$ is constant on every block $S_i\times S_j$. Every graph has a natural representation as a stepfunction in $\mathcal{W}$ taking values either 0 or 1 (such a graphon is referred to as a \emph{stepgraphon}). In particular, given a graph $G$ on $n$ vertices indexed $\{1, 2, \ldots, n\}$, we can define a measurable set $K_G \subseteq [0,1]^2$ as 
\[K_G = \bigcup_{u \sim v} \left[\frac{u-1}{n}, \frac{u}{n}\right]\times \left[\frac{v-1}{n}, \frac{v}{n}\right],\]
and this represents the graph $G$ as a bounded symmetric measurable function $W_G$ which takes value $1$ on $K_G$ and $0$ everywhere else. For a measurable subset $U$ we will use $m(U)$ to denote its Lebesgue measure.

This representation of a graph as a measurable subset of $[0,1]^2$ lends itself to a visual presentation sometimes referred to as a \emph{pixel picture}; see, for example, Figure \ref{bipartites_pixel} for two representations of a bipartite graph as a measurable subset of $[0,1]^2.$ Clearly, this indicates that such a representation is not unique; neither is the representation of a graph as a stepfunction. Using an equivalence relation on $\mathcal{W}$ derived from the so-called \emph{cut metric}, we can identify graphons that are equivalent up to relabelling, and up to any differences on a set of measure zero (i.e. equivalent \emph{almost everywhere}). 
\begin{figure}
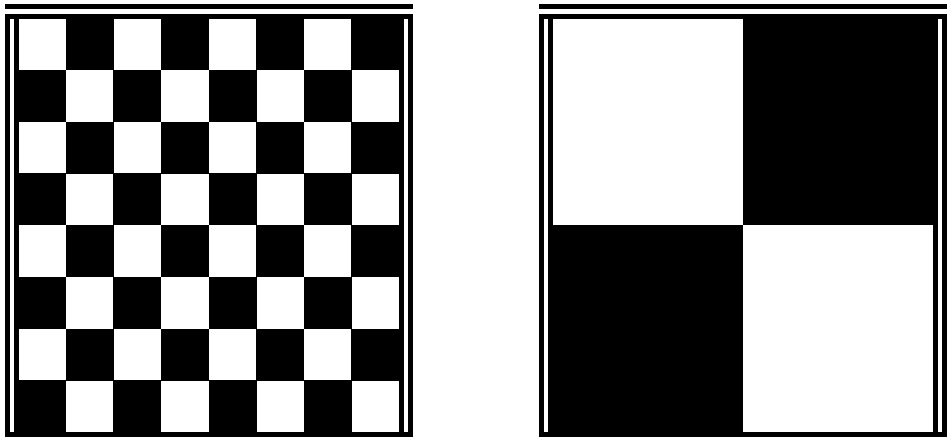

\begin{center}
\begin{tabular}{||cccccccc||}\hline\hline
& \black &  & \black &  & \black & & \black \\
\black &  & \black &  & \black & & \black & \\
& \black &  & \black &  & \black & & \black \\
\black &  & \black &  & \black & & \black & \\
& \black &  & \black &  & \black & & \black \\
\black &  & \black &  & \black & & \black & \\
& \black &  & \black &  & \black & & \black \\
\black &  & \black &  & \black & & \black & \\\hline\hline
\end{tabular} \hspace{40pt} \begin{tabular}{||cccccccc||}\hline\hline
& & & & \black & \black & \black & \black \\
& & & & \black & \black & \black & \black \\
& & & & \black & \black & \black & \black \\
& & & & \black & \black & \black & \black \\
\black & \black & \black & \black & & & & \\
\black & \black & \black & \black & & & & \\
\black & \black & \black & \black & & & & \\
\black & \black & \black & \black & & & & \\\hline\hline
\end{tabular}
\end{center}
\label{bipartites_pixel}
\caption{Two presentations of a bipartite graph as a stepfunction.}
\end{figure}

    For all symmetric, bounded Lebesgue-measurable functions $W:[0,1]^2\to \RR$, we let 
    \[
        \|W\|_\square 
        = \sup_{S, T\subseteq [0,1]} 
            \left|\int_{S\times T} W(x,y)\,dx\,dy 
        \right|.
    \]

%    This is referred to as the \emph{cut norm}. One can also define a semidistance $\delta_\square$ on $\mathcal{W}$. To account for `relabeling', consider the set $\mathcal{S}$ of all measure-preserving bijections of $[0,1]$, and define    \[\delta_\square(W_1, W_2) = \inf_{\phi\in \mathcal{S}} \{\|W_1-W_2\circ\phi\|_\square\}.\]    We denote by $\hat{\mathcal{W}}$ the quotient space of $\mathcal{W}$ under the equivalence relation \[W_1 \sim W_2 \Leftrightarrow \delta_\square(W_1, W_2) = 0.\] On $\hat{\mathcal{W}}$, $\delta_\square$ is a metric, and by \cite[Theorem 5.1]{lovasz2007szemeredi}, $\hat{\mathcal{W}}$ is a compact space.

    Here, $\|\cdot\|_\square$ is referred to as the \emph{cut norm}. 
    Next, one can also define a semidistance $\delta_\square$ on $\mathcal{W}$ as follows.  
    First, we define \emph{weak isomorphism} of graphons.  
    Let $\mathcal{S}$ be the set of of all measure-preserving functions on $[0,1]$.  
    For every $\varphi\in \mathcal{S}$ and every $W\in \ww$, define $W^\varphi:[0,1]^2\to [0,1]$ by 
    \begin{align*}
        W^\varphi(x,y) 
        :=
            W(\varphi(x),\varphi(y))
    \end{align*}
    for a.e. $(x,y)\in [0,1]^2$.  
    Now for any $W_1,W_2\in \ww$, let 
    \[
        \delta_\square(W_1, W_2) 
        = \inf_{\phi\in \mathcal{S}}
        \{
            \|W_1-W_2\circ\phi\|_\square
        \}.
    \] 
    Define the equivalence relation $\sim$ on $\ww$ as follows: for all $W_1, W_2\in \ww$, $W_1\sim W_2$ if and only if $\delta_\square(W_1,W_2) = 0$.  
    Furthermore, let $\wwhat := \ww/\sim$ be the quotient space of $\ww$ under $\sim$.  
    Note that $\delta_\square$ induces a metric on $\wwhat$.  
    Crucially, by \cite[Theorem 5.1]{lovasz2007szemeredi}, $\hat{\mathcal{W}}$ is a compact metric space.

Given $W\in \wwhat$, {we} define the Hilbert-Schmidt operator $A_W: \elltwo \to \elltwo$ {by} 
\[
    (A_Wf)(x) := \int_0^1 W(x,y)f(y) \,dy.
\]
for all $f\in \elltwo$ and a.e. $x\in [0,1]$.

    Since $W$ is symmetric and bounded, $A_W$ is a compact Hermitian operator.  
    In particular, $A_W$ has a discrete, real spectrum whose only possible accumulation point is $0$ (c.f. \cite{aubin2011applied}). 
    In particular, the maximum and minimum eigenvalues exist and we focus our attention on these extremes.  
    Let $\mu(W)$ and $\nu(W)$ be the maximum and minimum eigenvalue of $A_W$, respectively, and define the \emph{spread} of $W$ as
    \[
        \spr(W) := \mu(W) - \nu(W).
    \]

By the {M}in-{M}ax Theorem, 
% {\color{red}(Citation needed?)}
    % I don't think so
we have that 
\[
    \mu(W) = \max_{\|f\|_2 = 1} \int_0^1\int_0^1 W(x,y)f(x)f(y) \, dx\, dy,
\]
and
\[
\nu(W) = \min_{\|f\|_2 = 1} \int_0^1\int_0^1 W(x,y)f(x)f(y) \, dx\, dy.
\]

Both $\mu$ and $\nu$ are continuous functions with respect to $\delta_\square$: in particular we have the following.

\begin{theorem}[c.f. Theorem 6.6 from \cite{BCLSV2012GraphLimitsSpectra} or Theorem 11.54 in \cite{Lovasz2012Hombook}]\label{thm: graphon eigenvalue continuity}
	Let $\{W_i\}_i$ be a sequence of graphons converging to $W$ with respect to $\cutdist$.  
	Then as $n\to\infty$, 
	\begin{align*}
		\mu(W_n)\to\mu(W)
		\quad\text{ and }\quad
		\nu(W_n)\to\nu(W).  
	\end{align*}
\end{theorem}
~\\

If $W \sim W'$ then $\mu(W) = \mu(W')$ and $\nu(W) = \nu(W')$. By compactness, we may consider the optimization problem on the factor space $\hat{\mathcal{W}}$ 
\[
\spr(\hat{\mathcal{W}})=\max_{W\in \hat{\mathcal{W}}}, \spr(W)
\]
%Since $\hat{\mathcal{W}}$ is a compact space and $s$ is a continuous function with respect to $\delta_\square$, $\spr(\hat{\mathcal{W}})$ is well-defined 
and furthermore there is a $W \in \hat{\mathcal{W}}$ that attains the maximum. 
Since every graph is represented by $W_G\in \hat{\mathcal{W}}$, this allows us to give an upper bound for $s(n)$ in terms of $\spr(\hat{\mathcal{W}})$. 
Indeed, by replacing the eigenvectors of $G$ with their corresponding stepfunctions, the following {proposition} can be shown.

\begin{proposition}\label{graph to graphon eigenvalue scaling}
    Let $G$ be a graph on $n$ vertices. 
    Then 
    \begin{align*}
        \lambda_1(G) = n \blue{\, \cdot\, } \mu({W_G})
        \quad \text{ and }\quad 
        \lambda_n(G) = n \blue{\, \cdot\, } \nu({W_G}).
    \end{align*}
\end{proposition}

Proposition \ref{graph to graphon eigenvalue scaling} implies that $s(n) \leq n\cdot\spr(\hat{\mathcal{W}}) $ for all $n$. Combined with Theorem \ref{thm: functional analysis spread}, this gives the following corollary.
\begin{corollary}
For all $n$, $s(n) \leq \frac{2n}{\sqrt{3}}$.
\end{corollary}

This can be proved more directly using Theorem \ref{thm: spread maximum graphs} and taking tensor powers.

\subsection{Properties of spread-extremal graphons}
Our main objective in the next sections is to solve the maximum spread problem for graphons in order to determine this upper bound for $s(n)$. As such, in this subsection we set up some preliminaries to the solution which largely comprise a translation of what is known in the graph setting (see Section~\ref{sec:graphs}). Specifically, we define what it means for a graphon to be connected, and show that spread-extremal graphons must be connected. We then prove a standard corollary of the Perron-Frobenius theorem. Finally, we prove graphon versions of Lemma~\ref{lem: graph join} and Lemma~\ref{discrete ellipse equation}.

% For the following propositions, we briefly discuss direct sums and connectivity in graphons.  
% For a full discussion, see \cite{BollobasJansonRiordan2012Connectivity} and \cite{Lovasz2012Hombook}.  

Let $W_1$ and $W_2$ be graphons and let $\alpha_1,\alpha_2$ be positive real numbers with $\alpha_1+\alpha_2 = 1$.  
We define the \textit{direct sum} of $W_1$ and $W_2$ with weights $\alpha_1$ and $\alpha_2$, denoted $W = \alpha_1W_1\oplus \alpha_2W_2$, as follows.  
Let $\varphi_1$ and $\varphi_2$ be the increasing affine maps which send $J_1 := [0,\alpha_1]$ and $J_2 := [\alpha_1,1]$ to $[0,1]$, respectively.  
Then for all $(x,y)\in [0,1]^2$, let 
\begin{align*}
    W(x,y) := \left\{\begin{array}{rl}
        W_i(\varphi_i(x),\varphi_i(y)), &\text{if }(x,y)\in J_i\times J_i\text{ for some }i\in \{1,2\}\\
        0, & \text{otherwise}
    \end{array}\right.
    .  
\end{align*}
A graphon $W$ is \textit{connected} if $W$ is not weakly isomorphic to a direct sum $\alpha_1W_1\oplus \alpha_2W_2$ where $\alpha_1\neq 0,1$.  
Equivalently, $W$ is connected if there does not exist a measurable subset $A\subseteq [0,1]$ of positive measure such that $W(x,y) = 0$ for a.e. $(x,y)\in A\times A^c$.  \\

\begin{proposition}\label{prop: disconnected spectrum}
    Suppose $W_1,W_2$ are graphons and $\alpha_1,\alpha_2$ are positive real numbers summing to $1$.  
    Let $W:=\alpha_1W_1\oplus\alpha_2W_2$.  
    Then as multisets, 
    \begin{align*}
        \Lambda(W) 
        = 
            \{\alpha_1 u : u \in \Lambda(W_1) \}\cup \{ \alpha_2v : v\in \Lambda(W_2)\}.  
    \end{align*}
    Moreover, $\spr(W)\leq \alpha_1 \spr(W_1) + \alpha_2 \spr(W_2)$ with equality if and only $W_1$ or $W_2$ is the all-zeroes graphon.  
\end{proposition}
\begin{proof}
    For convenience, let $\Lambda_i := \{\alpha_iu : u\in\Lambda(W_i)\}$ for each $i\in\{1,2\}$ and $\Lambda := \Lambda(W)$.  
    The first claim holds simply by considering the restriction of eigenfunctions to the intervals $[0,\alpha_1]$ and $[\alpha_1,1]$.  \\
    
    For the second claim, we first write $\spr(W) = \alpha_i\mu-\alpha_j\nu$ where $i,j\in\{1,2\}$.  
    Let $I_i := [\min(\Lambda_i), \max(\Lambda_i)]$ for each $i\in\{1,2\}$ and $I := [\min(\Lambda), \max(\Lambda)]$.  
    Clearly $\alpha_i \spr(W_i) = \diam(I_i)$ for each $i\in\{1,2\}$ and $\spr(W) = \diam(I)$.  
    Moreover, $I = I_1\cup I_2$.  
    Since $0\in I_1\cap I_2$, $\diam(I)\leq \diam(I_1)+\diam(I_2)$ with equality if and only  if either $I_1$ or $I_2$ equals $\{0\}$.  
    So the desired claim holds.  
\end{proof}
Furthermore, the following basic corollary of the Perron-Frobenius holds.  
For completeness, we prove it here.  
\begin{proposition}\label{prop: PF eigenfunction}
    Let $W$ be a connected graphon and write $f$ for an eigenfunction corresponding to $\mu(W)$.  
    Then $f$ is nonzero with constant sign a.e.  
\end{proposition}
\begin{proof}
    Let $\mu = \mu(W)$.  
    Since 
    \begin{align*}
        \mu = \max_{\|h\|_2 = 1}\int_{(x,y)\in [0,1]^2}W(x,y)h(x)h(y), 
    \end{align*}
    it follows without loss of generality that $f\geq 0$ a.e. on $[0,1]$.  
    Let $Z:=\{x\in [0,1] : f(x) = 0\}$.  
    Then for a.e. $x\in Z$, 
    \begin{align*}
        0 = 
            \mu f(x) 
        = 
            \int_{y\in [0,1]}W(x,y)f(y)
        = 
            \int_{y\in Z^c}
                W(x,y)f(y).  
    \end{align*}
    Since $f > 0$ on $Z^c$, it follows that $W(x,y) = 0$ a.e. on $Z\times Z^c$.  
    Clearly $m(Z^c) \neq 0$.  
    If $m(Z) = 0$ then the desired claim holds, so without loss of generality, $0 < m(Z),m(Z^c)<1$.  
    It follows that $W$ is disconnected, a contradiction to our assumption, which completes the proof of the desired claim.  
\end{proof}

We may now prove a graphon version of Lemma \ref{lem: graph join}.
    
\begin{lemma}\label{lem: K = indicator function}
    Suppose $W$ is a graphon achieving maximum spread and let $f,g$ be eigenfunctions for the maximum and minimum eigenvalues for $W$, respectively.  
    Then the following claims hold: 
    \begin{enumerate}[(i)]
        \item For a.e. $(x,y)\in [0,1]^2$, \label{item: K = 0 or 1}
        \begin{align*}
            W(x,y) 
            = \left\{\begin{array}{rl}
                1, & f(x)f(y) > g(x)g(y) \\
                0, & \text{otherwise}
            \end{array}
            \right.
            .  
        \end{align*}
        %\item $f(x)\neq 0$ for a.e. $x\in [0,1]$.  \label{item:f > 0}
        \item $f(x)f(y)-g(x)g(y) \neq 0$ for a.e. $(x,y)\in [0,1]^2$.  \label{item: |diff| > 0}
        \end{enumerate}
\end{lemma}
\begin{proof}
    We proceed in the following order: 
    \begin{itemize}
        \item Prove Item \eqref{item: K = 0 or 1} holds for a.e. $(x,y)\in [0,1]^2$ such that $f(x)f(y)\neq g(x)g(y)$.  
        We will call this Item \eqref{item: K = 0 or 1}*.  
        %\item Prove Item \eqref{item:f > 0}.  
        \item Prove Item \eqref{item: |diff| > 0}.  
        \item Deduce Item \eqref{item: K = 0 or 1} also holds.  
    \end{itemize}
    
    By Propositions \ref{prop: disconnected spectrum} and \ref{prop: PF eigenfunction}, we may assume without loss of generality that $f > 0$ a.e.~on $[0,1]$.  
    For convenience, we define the quantity $d(x,y) := f(x)f(y)-g(x)g(y)$.  
    To prove Item \eqref{item: K = 0 or 1}*, we first define a graphon $W'$ by 
    \begin{align*}
       W'(x,y)
        = \left\{\begin{array}{rl}
            1, & d(x,y) > 0 \\
            0, & d(x,y) < 0 \\
            W(x,y) & \text{otherwise}
        \end{array}
        \right.
        .
    \end{align*}
    Then by inspection, 
    \begin{align*}
        \spr(W') 
        &\geq 
            \int_{(x,y)\in [0,1]^2}
                W'(x,y)(f(x)f(y)-g(x)g(y))
        \\
        &= 
            \int_{(x,y)\in [0,1]^2}
                W(x,y)(f(x)f(y)-g(x)g(y)) \\&+ \int_{d(x,y) > 0}
                (1-W(x,y))d(x,y)
            - \int_{d(x,y) < 0}
                W(x,y)d(x,y)
        \\
        &= 
            \spr(W) + \int_{d(x,y) > 0}
                (1-W(x,y))d(x,y)
            - \int_{d(x,y) < 0}
                W(x,y)d(x,y).  
    \end{align*}
    Since $W$ maximizes spread, both integrals in the last line must be $0$ and hence Item \eqref{item: K = 0 or 1}* holds.  \\
    
    Now, we prove Item \eqref{item: |diff| > 0}.  
    For convenience, we define $U$ to be the set of all pairs $(x,y)\in [0,1]^2$ so that $d(x,y) = 0$.  
    Now let $W'$ be any graphon which differs from $W$ only on $U$.  
    Then 
    \begin{align*}
        \spr(W') 
        &\geq 
            \int_{(x,y)\in [0,1]^2}
                W'(x,y)(f(x)f(y)-g(x)g(y))
        \\
        &= 
            \int_{(x,y)\in [0,1]^2}
                W(x,y)(f(x)f(y)-g(x)g(y))\\&+ \int_{(x,y)\in U}
                (W'(x,y)-W(x,y))(f(x)f(y)-g(x)g(y))
        \\
        &= 
            \spr(W).  
    \end{align*}
    Since $\spr(W)\geq \spr(W')$, $f$ and $g$ are eigenfunctions for $W'$ and we may write $\mu'$ and $\nu'$ for the corresponding eigenvalues.  
    Now, we define 
    \begin{align*}
        I_{W'}(x) 
        &:= 
            (\mu'-\mu)f(x) \\
        &= 
            \int_{y\in [0,1]}(W'(x,y)-W(x,y))f(y)\\
        &= \int_{y\in [0,1], \, (x,y)\in U}
            (W'(x,y)-W(x,y))f(y).  
    \end{align*}
    Similarly, we define 
    \begin{align*}
        J_{W'}(x) 
        &:= 
            (\nu'-\nu)g(x) \\
        &= 
            \int_{y\in [0,1]}(W'(x,y)-W(x,y))g(y)\\
        &= \int_{y\in [0,1],\, (x,y)\in U}
            (W'(x,y)-W(x,y))g(y).  
    \end{align*}
    Since $f$ and $g$ are orthogonal, 
    \begin{align*}
        0 &= \int_{x\in [0,1]}
            I_{W'}(x)J_{W'}(x).  
    \end{align*}
 By definition of $U$, we have that for a.e. $(x,y) \in U$, $0 = d(x,y) = f(x)f(y) - g(x)g(y)$. In particular, since $f(x),f(y) > 0$ for a.e. $(x,y)\in [0,1]^2$, then a.e. $(x,y)\in U$ has $g(x)g(y) > 0$.  
    So by letting 
    \begin{align*}
        U_+ &:= \{(x,y)\in U: g(x),g(y) > 0\},  \\
        U_- &:= \{(x,y)\in U:g(x),g(y)<0\},\text{ and}  \\
        U_0 &:= U\setminus (U_+\cup U_-), 
    \end{align*}
    $U_0$ has measure $0$.  
    \\
    
    First, let $W'$ be the graphon defined by 
    \begin{align*}
        W'(x,y) 
        &= \left\{\begin{array}{rl}
            1, & (x,y)\in U_+\\
            W(x,y), &\text{otherwise}
        \end{array}
        \right. 
        .  
    \end{align*}
    For this choice of $W'$, 
    \begin{align*}
        I_{W'}(x) 
        &= 
            \int_{y\in [0,1], \, (x,y)\in U_+}
                (1-W(x,y))f(y), \text{ and} \\
        J_{W'}(x) 
        &= 
            \int_{y\in [0,1], \, (x,y)\in U_+}
                (1-W(x,y))g(y).  
    \end{align*}
    Clearly $I_{W'}$ and $J_{W'}$ are nonnegative functions so $I_{W'}(x)J_{W'}(x) = 0$ for a.e. $x\in [0,1]$.  
    Since $f(y)$ and $g(y)$ are positive for a.e. $(x,y)\in U$, $W(x,y) = 1$ for a.e. on $U_+$.  \\
    
    If instead we let $W'(x,y)$ be $0$ for all $(x,y)\in U_+$, it follows by a similar argument that $W(x,y) = 0$ for a.e. $(x,y)\in U_+$.  
    So $U_+$ has measure $0$.  
    Repeating the same argument on $U_-$, we similarly conclude that $U_-$ has measure $0$.  
    This completes the proof of Item \eqref{item: |diff| > 0}.  \\
    
    Finally we note that Items \eqref{item: K = 0 or 1}* and \eqref{item: |diff| > 0} together implies Item \eqref{item: K = 0 or 1}. 
\end{proof}

From here, it is easy to see that any graphon maximizing the spread is a join of two threshold graphons.   Next we prove the graphon version of Lemma \ref{discrete ellipse equation}.

\begin{lemma}\label{lem: local eigenfunction equation}
    If $W$ is a graphon achieving the maximum spread with corresponding eigenfunctions $f,g$, then $\mu f^2 - \nu g^2 = \mu-\nu$ almost everywhere.  
\end{lemma}
\begin{proof}
    We will use the notation $(x,y) \in W$ to denote that $(x,y)\in [0,1]^2$ satisfies $W(x,y) =1$. Let $\varphi:[0,1]\to[0,1]$ be an arbitrary homeomorphism which is {\em orientation-preserving} in the sense that $\varphi(0) = 0$ and $\varphi(1) = 1$.  
    Then $\varphi$ is a continuous strictly monotone increasing function which is differentiable almost everywhere.  
    Now let $\tf := \varphi'\cdot (f\circ \varphi)$, $\tg := \varphi'\cdot (g\circ \varphi)$ and $\tilde{W} := \{ (x,y)\in [0,1]^2 : (\varphi(x),\varphi(y))\in W \}$.  
    Using the substitutions $u = \varphi(x)$ and $v = \varphi(y)$, 
    \begin{align*}
        \tf \tilde{W} \tf 
        &= 
            \int_{ (x,y)\in [0,1]^2 }
                \chi_{(\varphi(x),\varphi(y))\in \tilde{W}}\; 
                \varphi'(x)\varphi'(y)\cdot 
                f(\varphi(x))f(\varphi(y))
                dx\, dy\\
        &= \int_{ (x,y)\in [0,1]^2 }
            \chi_{(x,y)\in W}\; 
            f(u)f(v)
            du\, dv\\
        &= \mu.  
    \end{align*}
    Similarly, $\tg\tilde{W}\tg = \nu$.  
    \\
    \\
    Note however that the $L_2$ norms of $\tf,\tg$ may not be $1$.  
    Indeed using the substitution $u = \varphi(x)$, 
    \[
        \|\tf\|_2^2 
        = 
            \int_{x\in[0,1]}
                \varphi'(x)^2f(\varphi(x))^2\, dx
        = 
            \int_{u\in [0,1]}
                \varphi'(\varphi^{-1}(u))\cdot f(u)^2\, du
        .
    \]
    We exploit this fact as follows.  
    Suppose $I,J$ are disjoint subintervals of $[0,1]$ of the same positive length $m(I) = m(J) = \ell > 0$ and for any $\varepsilon > 0$ sufficiently small (in terms of $\ell$), let $\varphi$ be the (unique) piecewise linear function which  stretches $I$ to length $(1+\varepsilon)m(I)$, shrinks $J$ to length $(1-\varepsilon)m(J)$, and shifts only the elements in between $I$ and $J$.  
    Note that for a.e. $x\in [0,1]$, 
    \[
        \varphi'(x) = \left\{\begin{array}{rl}
            1+\varepsilon,  & x\in I \\
            1-\varepsilon,  & x\in J \\
            1,              & \text{otherwise}. 
        \end{array}\right.
    \]
    Again with the substitution $u = \varphi(x)$, 
    \begin{align*}
        \|\tf\|_2^2 
        &= 
            \int_{x\in[0,1]}
                \varphi'(x)^2\cdot f(\varphi(x))^2\, dx\\
        &= 
            \int_{[u\in [0,1]}
                \varphi'(\varphi^{-1}(u))f(u)^2\, du \\
        &= 
            1
            +
            \varepsilon\cdot (
                \|\chi_If\|_2^2
                -\|\chi_Jf\|_2^2
            ).  
    \end{align*}
    The same equality holds for $\tg$ instead of $\tf$.  
    After normalizing $\tf$ and $\tg$, by optimality of $W$, we get a difference of Rayleigh quotients as 
    \begin{align*}
        0 
        &\leq 
         (fWf-gWg) - \dfrac{\tf\tilde{W}\tf}{\|\tf\|_2^2} - \dfrac{\tg\tilde{W}\tg}{\|\tg\|_2^2} 
            \\
        &= 
            \dfrac{
                \mu \varepsilon\cdot (
                    \|\chi_If\|_2^2
                    -\|\chi_Jf\|_2^2)
                }
                {1+\varepsilon\cdot (
                    \|\chi_If\|_2^2
                    -\|\chi_Jf\|_2^2)
                }
            - 
                \dfrac{
                    \nu\varepsilon\cdot (
                        \|\chi_Ig\|_2^2
                        -\|\chi_Jg\|_2^2)
                    }
                {1+\varepsilon\cdot (
                    \|\chi_Ig\|_2^2
                    -\|\chi_Jg\|_2^2)
                }\\
        &= (1+o(1))\varepsilon\cdot\left(
            \int_I
                (\mu f(x)^2-\nu g(x)^2)dx
            -\int_J
                (\mu f(x)^2-\nu g(x)^2)dx
        \right)
    \end{align*}
    as $\varepsilon\to 0$.  
    It follows that for all disjoint intervals $I,J\subseteq [0,1]$ of the same length that the corresponding integrals are the same.  
    Taking finer and finer partitions of $[0,1]$, it follows that the integrand $\mu f(x)^2-\nu g(x)^2$ is constant almost everywhere.  
    Since the average of this quantity over all $[0,1]$ is $\mu-\nu$, the desired claim holds.  
\end{proof}

\section{From graphons to stepgraphons}\label{sec: graphon spread reduction}

% In this section, our main result is as follows: we show that the maximum spread of graphons is attained by a stepgraphon with unknown stepsizes, with values corresponding to the matrix in Figure \ref{figure: 7x7 graphon}.  
% \begin{figure}[h]
% 	\center
% 	\input{graphics/stepgraphon7x7}
% 	\caption{
% 		A matrix underlying any spread-optimal graphon.  
% 	}
% 	\label{figure: 7x7 graphon}
% \end{figure}
% In Section \ref{sec: graphs to graphons}, we will complete the proof of Theorem \ref{thm: maximum spread}.  

%For any graphon $W$, we denote by $\Lambda(W)$ the spectrum of $W$ as a multiset of eigenvalues.  
%The spectrum of a graphon is discrete and continuous with respect to the cut metric.  
%For this reason, we may order $\Lambda(W)$ as $\lambda_1(W)\geq \lambda_2(W)\geq\cdots\geq 0$ and $\lambda_1'(W)\leq \lambda_2'(W)\leq\cdots\leq 0$ (with an infinite padding of $0$ eigenvalues, when needed).  
%A crucial result is the following theorem.  
%\begin{theorem}[Theorem 6.6 from \cite{BCLSV2012GraphLimitsSpectra}, rephrased as in Theorem 11.54 in \cite{Lovasz2012Hombook}]
%	Let $(W_i)_i$ be a sequence of graphons converging in the $\cutdist$.  
%	Then for every fixed $i\geq 1$ and $n\to\infty$, 
%	\begin{align*}
%		\lambda_i(W_n)\to\lambda_i(W)
%		\quad\text{ and }\quad
%		\lambda_i'(W_n)\to\lambda_i'(W).  
%	\end{align*}
%\end{theorem}
%For convenience, we write $\mu(W) := \lambda_1(W)$, $\nu(W) := \lambda_1'(W)$, and $\spr(W) := \mu(W) - \nu(W)$.  
%It follows that these functions are continuous in $\ww_0$.  
%Since $\ww_0$ is a compact metric space, $\spr$ attains a maximum.  
The main result of this section is as follows.  
\begin{theorem}\label{thm: reduction to stepgraphon}
    Suppose $W$ maximizes $\spr(\hat{\mathcal{W}})$.  
    Then $W$ is a stepfunction taking values $0$ and $1$ of the following form 
    \begin{align*}
    	\begin{tabular}{||cccc||ccc||}\hline\hline
            \black & \black & \black & \black & \black & \black & \black \\
            \black & \black & \black &  & \black & \black & \black  \\
            \black & \black &  &  & \black & \black & \black  \\
            \black &  &  &  & \black & \black & \black  \\\hline\hline
            \black & \black & \black & \black & \black & \black &  \\
            \black & \black & \black & \black & \black &  &  \\
            \black & \black & \black & \black &  &  &  \\\hline\hline
\end{tabular}
    	\quad .
    \end{align*}
    Furthermore, the internal divisions separate according to the sign of the eigenfunction corresponding to the minimum eigenvalue of $W$.  
\end{theorem}

% In Subsection \ref{sub-sec: local behavior}, we show that any graphon which maximizes spread is the join of two threshold graphons whose threshold functions may be defined in terms of the extreme eigenfunctions.  
%We show that the extreme eigenfunctions satisfy an elliptical constraint.  
%Then 
We begin Section \ref{sub-sec: L2 averaging} by mirroring the argument in \cite{terpai} which proved a conjecture of Nikiforov regarding the largest eigenvalue of a graph and its complement, $\mu+\overline{\mu}$.  
There Terpai showed that performing two operations on graphons leads to a strict increase in $\mu+\overline{\mu}$.  
Furthermore based on previous work of Nikiforov from \cite{Nikiforov4}, the conjecture for graphs reduced directly to maximizing $\mu+\overline{\mu}$ for graphons.  Using these operations, Terpai \cite{terpai} reduced to a $4\times 4$ stepgraphon and then completed the proof by hand.

In our case, we are not so lucky and are left with a $7\times 7$ stepgraphon after performing similar but more technical operations, detailed in this section.  
In order to reduce to a $3\times 3$ stepgraphon, we appeal to interval arithmetic (see Section \ref{sub-sec: numerics} and Appendices 
    \ref{sec: ugly}
    and \ref{sec: appendix}). 
Furthermore, our proof requires an additional technical argument to translate the result for graphons (Theorem \ref{thm: spread maximum graphon}) to our main result for graphs (Theorem \ref{thm: spread maximum graphs}).  
In Section \ref{sub-sec: stepgraphon proof}, we prove Theorem \ref{thm: reduction to stepgraphon}.

\subsection{Averaging}\label{sub-sec: L2 averaging}
For convenience, we introduce some terminology. For any graphon $W$ with $\lambda$-eigenfunction $h$, we say that $x\in [0,1]$ is \textit{typical} (with respect to $W$ and $h$) if 
\begin{align*}
    \lambda\cdot h(x) = \int_{y\in [0,1]}W(x,y)h(y).  
\end{align*}
Note that a.e. $x\in [0,1]$ is typical.  
Additionally if $U\subseteq [0,1]$ is measurable with positive measure, then we say that $x_0\in U$ is \textit{average} (on $U$, with respect to $W$ and $h$) if 
\begin{align*}
    h(x_0)^2 = \dfrac{1}{m(U)}\int_{y\in U}h(y)^2.  
\end{align*}
Given $W,h,U$, and $x_0$ as above, we define the $L_2[0,1]$ function  $\av_{U,x_0}h$ by setting 
\begin{align*}
    (\av_{U,x_0}h)(x) := \left\{\begin{array}{rl}
        h(x_0), & x\in U\\
        h(x), & \text{otherwise}
    \end{array}\right.
    .  
\end{align*}
Clearly $\|\av_{U,x_0}h\|_2 = \|h\|_2$.  
Additionally, we define the graphon $\av_{U,x_0}W$ by setting 
\begin{align*}
    \av_{U,x_0}W (x,y)
    := \left\{\begin{array}{rl}
        0, & (x,y)\in U\times U \\
        W(x_0,y), &(x,y)\in U\times U^c \\
        W(x,x_0), &(x,y)\in U^c\times U \\
        W(x,y), & (x,y)\in U^c\times U^c
    \end{array}\right.
    .  
\end{align*}

In the graph setting, this is analogous to replacing $U$ with an independent set whose vertices are clones of $x_0$. The following lemma indicates how this cloning affects the eigenvalues.

\begin{lemma}\label{lem: eigenfunction averaging}
    Suppose $W$ is a graphon with $h$ a $\lambda$-eigenfunction and suppose there exist disjoint measurable subsets $U_1,U_2\subseteq [0,1]$ of positive measures $\alpha$ and $\beta$, respectively.  
    Let $U:=U_1\cup U_2$.  
    Moreover, suppose $W = 0$ a.e. on $(U\times U)\setminus (U_1\times U_1)$. Additionally, suppose $x_0\in U_2$ is typical and average on $U$, with respect to $W$ and $h$.  
    Let $\tilh := \av_{U,x_0}h$ and $\tW := \av_{U,x_0}W$.  
    Then for a.e. $x\in [0,1]$, 
    \begin{align}\label{eq: averaged vector image}
        (A_{\tW}\tilh)(x) 
        &= 
            \lambda\tilh(x) + 
            \left\{\begin{array}{rl}
                0, & x\in U\\
                m(U)\cdot W(x_0,x)h(x_0) - \int_{y\in U} W(x,y)h(y), &\text{otherwise}
            \end{array}\right.
        .  
    \end{align}
    Furthermore, 
    \begin{align}\label{eq: averaged vector product}
        \langle A_{\tW}\tilh, \tilh\rangle
        = 
            \lambda + \int_{(x,y)\in U_1\times U_1}W(x,y)h(x)h(y).  
    \end{align}
\end{lemma}
\begin{proof}
    We first prove Equation \eqref{eq: averaged vector image}.  
    Note that for a.e. $x\in U$, 
    Then 
    \begin{align*}
        (A_{\tW}\tilh)(x) 
        &= 
            \int_{y\in [0,1]}
                \tW(x,y)\tilh(y)
        \\
        &= 
            \int_{y\in U}
                \tW(x,y)\tilh(y) 
            + \int_{y\in [0,1]\setminus U}
                \tW(x,y)\tilh(y)
        \\
        &= 
            \int_{y\in [0,1]\setminus U}
                W(x_0,y)h(y)
        \\
        &= 
            \int_{y\in [0,1]}
                W(x_0,y)h(y) 
                - \int_{y\in U}
                    W(x_0,y)h(y)
        \\
        &= 
            \lambda h(x_0)\\
        &= \lambda \tilh(x), 
    \end{align*}
    as desired.  
    Now note that for a.e. $x\in [0,1]\setminus U$, 
    \begin{align*}
        (A_{\tW}\tilh)(x) 
        &= 
            \int_{y\in [0,1]}
                \tW(x,y)\tilh(y)
        \\
        &= 
            \int_{y\in U}
                \tW(x,y)\tilh(y) 
            + \int_{y\in [0,1]\setminus U}
                \tW(x,y)\tilh(y)
        \\
        &= 
             \int_{y\in U}
                W(x_0,x)h(x_0)
            +  \int_{y\in [0,1]\setminus U}
                W(x,y)h(y) 
        \\
        &= 
             m(U)\cdot W(x_0,x)h(x_0)
             +  \int_{y\in [0,1]}
                W(x,y)h(y) 
            - \int_{y\in U}
                W(x,y)h(y) 
        \\
        &= 
             \lambda h(x)
             + m(U)\cdot W(x_0,x)h(x_0)
             - \int_{y\in U}
                W(x,y)h(y)
        .
    \end{align*}
    So again, the claim holds and this completes the proof of Equation \eqref{eq: averaged vector image}.  
    Now we prove Equation \eqref{eq: averaged vector product}.  
    Indeed by Equation \eqref{eq: averaged vector image}, 
    \begin{align*}
        \langle (A_{\tW}\tilh), \tilh\rangle 
        &= 
            \int_{x\in [0,1]}
                (A_{\tW}\tilh)(x)
                \tilh(x)
        \\
        &= 
            \int_{x\in [0,1]}
                \lambda\tilh(x)^2
            + \int_{x\in [0,1]\setminus U}
                \left(
                    m(U)\cdot W(x_0,x)h(x_0) - \int_{y\in U}W(x,y)h(y)
                \right)
                \cdot h(x)
        \\
        &= 
            \lambda
            + m(U)\cdot h(x_0) \left(
                \int_{x\in [0,1]}
                    W(x_0,x)h(x)
                - \int_{x\in U}
                    W(x_0,x)h(x)
            \right)
        \\
        &\quad 
            - \int_{y\in U}
                \left(
                    \int_{x\in [0,1]}
                        W(x,y)h(x)
                    -\int_{x\in U}
                        W(x,y)h(x)
                \right)
                \cdot h(y)
        \\
        &= 
            \lambda
            + m(U)\cdot h(x_0)\left(
                \lambda h(x_0) 
                - \int_{y\in U}
                    0
            \right)
            - \int_{y\in U}
                \left(
                    \lambda h(y)^2
                    -\int_{x\in U}
                        W(x,y)h(x)h(y)
                \right)
        \\
        &= 
            \lambda 
            + \lambda m(U)\cdot h(x_0)^2 
            - \lambda \int_{y\in U}
                h(y)^2
            + \int_{(x,y)\in U\times U}
                W(x,y)h(x)h(y)
        \\
        &= 
            \lambda 
            + \int_{(x,y)\in U_1\times U_1}
                W(x,y)h(x)h(y), 
    \end{align*}
    and this completes the proof of desired claims.  
\end{proof}

We have the following useful corollary.  
\begin{corollary}\label{cor: averaging reduction}
    Suppose $\spr(W) = \spr(\hat{\mathcal{W}})$ with maximum and minimum eigenvalues $\mu,\nu$ corresponding respectively to eigenfunctions $f,g$.  
    Moreover, suppose that there exist disjoint subsets $A,B\subseteq [0,1]$ and $x_0\in B$ so that the conditions of Lemma \ref{lem: eigenfunction averaging} are met for $W$ with $\lambda = \mu$, $h = f$, $U_1=A $, and $U_2=B$.  
    Then, 
    \begin{enumerate}[(i)]
        \item
            \label{item: U independent} 
            $W(x,y) = 0$ for a.e. $(x,y)\in U^2$, and 
        \item
            \label{item: f,g constant on U}
            $f$ is constant on $U$.  
    \end{enumerate}

\end{corollary}
\begin{proof}
    Without loss of generality, we assume that $\|f\|_2 = \|g\|_2 = 1$.  
    Write $\tW$ for the graphon and $\tf,\tg$ for the corresponding functions produced by Lemma \ref{lem: eigenfunction averaging}.  
    By Lemma \ref{prop: PF eigenfunction}, we may assume without loss of generality that $f > 0$ a.e. on $[0,1]$.  
    We first prove Item \eqref{item: U independent}.  
    Note that 
    \begin{align}
        \spr(\tW)
        &\geq \int_{(x,y)\in [0,1]^2}
           \tilde{W}(x,y)(\tilde{f}(x)\tilde{f}(y)-\tilde{g}(x)\tilde{g}(y))
        \nonumber\\
        &= 
            (\mu-\nu) + \int_{(x,y)\in A\times A}W(x,y)(f(x)f(y)-g(x)g(y))
        \nonumber\\
        &=
            \spr(W) + \int_{(x,y)\in A\times A}W(x,y)(f(x)f(y)-g(x)g(y)).  
        \label{eq: sandwich spread}
    \end{align}
    Since $\spr(W)\geq \spr(\tilde{W})$ and by Lemma \ref{lem: K = indicator function}.\eqref{item: |diff| > 0}, $f(x)f(y)-g(x)g(y) > 0$ for a.e. $(x,y)\in A\times A$ such that $W(x,y) \neq 0$.  
    Item \eqref{item: U independent} follows.  \\
    
    For Item \eqref{item: f,g constant on U}, we first note that $f$ is a $\mu$-eigenfunction for $\tilde{W}$. Indeed, if not, then the inequality in \eqref{eq: sandwich spread} holds strictly, a contradiction to the fact that $\spr(W)\geq \spr(\tW)$.  Again by Lemma \ref{lem: eigenfunction averaging}, 
    \begin{align*}
        m(U)\cdot W(x_0,x)f(x_0) = \int_{y\in U}W(x,y)f(y)
    \end{align*}
    for a.e. $x\in [0,1]\setminus U$.  
    Let $S_1 := \{x\in [0,1]\setminus U : W(x_0,x) = 1\}$ and $S_0 := [0,1]\setminus (U\cup S_1)$.  
    We claim that $m(S_1) = 0$.  
    Assume otherwise.  
    By Lemma \ref{lem: eigenfunction averaging} and by Cauchy-Schwarz, for a.e.  $x\in S_1$
    \begin{align*}
        m(U)\cdot f(x_0)
        &=
            m(U)\cdot W(x_0,x)f(x_0) 
        \\
        &= 
            \int_{y\in U}W(x,y)f(y)
        \\
        &\leq 
            \int_{y\in U} f(y)
        \\
        &\leq m(U)\cdot f(x_0), 
    \end{align*}
    and by sandwiching, $W(x,y) = 1$ and $f(y) = f(x_0)$ for a.e. $y\in U$.  
    Since $m(S_1) > 0$, it follows that $f(y) = f(x_0) = 0$ for a.e. $y\in U$, as desired.  \\
    
    So we assume otherwise, that $m(S_1) = 0$.  
    Then for a.e. $x\in [0,1]\setminus U$, $W(x_0,x) = 0$ and 
    \begin{align*}
        0 
        &= 
            m(U)\cdot W(x_0,x)f(x_0)= 
            \int_{y\in U}
                W(x,y)f(y)
    \end{align*}
    and since $f>0$ a.e. on $[0,1]$, it follows that $W(x,y) = 0$ for a.e. $y\in U$.  
    So altogether, $W(x,y) = 0$ for a.e. $(x,y)\in ([0,1]\setminus U)\times U$.  
    So $W$ is a disconnected, a contradiction to Fact \ref{prop: disconnected spectrum}.  
    So the desired claim holds.  
\end{proof}

\subsection{Proof of Theorem \ref{thm: reduction to stepgraphon}}\label{sub-sec: stepgraphon proof}
\begin{proof}
	For convenience, we write $\mu := \mu(W)$ and $\nu := \nu(W)$ and let $f,g$ denote the corresponding unit eigenfunctions.  
	Moreover by Proposition \ref{prop: PF eigenfunction}, we may assume without loss of generality that $f>0$.  \\
	
	First, we show without loss of generality that $f,g$ are monotone on the sets $P := \{x\in [0,1] : g(x)\geq 0\}$ and $N := [0,1]\setminus P$.  
	Indeed, we define a total ordering $\preccurlyeq$ on $[0,1]$ as follows.  
    For all $x$ and $y$, we let $x \preccurlyeq y$ if: 
    \begin{enumerate}[(i)]
        \item $g(x)\geq 0$ and $g(y)<0$, or 
        \item Item (i) does not hold and $f(x) > f(y)$, or 
        \item Item (i) does not hold, $f(x) = f(y)$, and $x\leq y$.  
    \end{enumerate}
    By inspection, the function 
    $\varphi:[0,1]\to[0,1]$ defined by 
    \begin{align*}
        \varphi(x) 
        := 
            m(\{y\in [0,1] : y\preccurlyeq x\}).  
    \end{align*}
    is a weak isomorphism between $W$ and its entrywise composition with $\varphi$. 
    By invariance of $\spr(\cdot)$ under weak isomorphism, we make the above replacement and write $f,g$ for the replacement eigenfunctions. That is, we are assuming that our graphon is relabeled so that $[0,1]$ respects $\preccurlyeq$.  \\
    
    As above, let $P := \{x\in [0,1] : g(x) \geq 0\}$ and $N:= [0,1]\setminus P$.  
    By Lemma \ref{lem: local eigenfunction equation}, $f$ and $-g$ are monotone {nonincreasing}  on $P$.  
    Additionally, $f$ and $g$ are monotone {nonincreasing} on $N$.  
    Without loss of generality, we may assume that $W$ is of the form from Lemma \ref{lem: K = indicator function}.  
    Now we let $S := \{x\in [0,1] : f(x) < |g(x)|\}$ and $C:=[0,1]\setminus S$. By Lemma \ref{lem: K = indicator function} we have that $W(x,y)=1$ for almost every $x,y\in C$ and $W(x,y)=0$ for almost every $x,y\in S \cap P$ or $x,y\in S\cap N$. We have used the notation $C$ and $S$ because the analogous sets in the graph setting form a clique or a stable set respectively. 
    We first prove the following claim.  \\
    \\
    \textbf{Claim A:} Except on a set of measure $0$, $f$ takes on at most $2$ values on $P\cap S$, and at most $2$ values on $N\cap S$.  \\
    
    We first prove this claim for $f$ on $P\cap S$.  
    Let $D$ be the set of all discontinuities of $f$ on the interior of the interval $P\cap S$.  
    Clearly $D$ consists only of jump-discontinuities.  
    By the Darboux-Froda Theorem, $D$ is at most countable and moreover, $(P\cap S)\setminus D$ is a union of at most countably many disjoint intervals $\ii$.  
    Moreover, $f$ is continuous on the interior of each $I\in\ii$.  
    \\
    
    We show now that $f$ is piecewise constant on the interiors of each $I\in \ii$.  
    Indeed, let $I\in \ii$.  
    Since $f$ is a $\mu$-eigenfunction function for $W$, 
    \begin{align*}
        \mu f(x) = \int_{y\in [0,1]}W(x,y)f(y)
    \end{align*}
    for a.e. $x\in [0,1]$ and by continuity of $f$ on the interior of $I$, this equation holds everywhere on the interior of $I$.  
    Additionally since $f$ is continuous on the interior of $I$, by the Mean Value Theorem, there exists some $x_0$ in the interior of $I$ so that 
    \begin{align*}
        f(x_0)^2 = \dfrac{1}{m(U)}\int_{x\in U}f(x)^2.  
    \end{align*}
    By Corollary \ref{cor: averaging reduction}, $f$ is constant on the interior of $U$, as desired.  \\
    
    If $|\ii|\leq 2$, the desired claim holds, so we may assume otherwise.  
    Then there exists distinct $I_1,I_2,I_3\in \ii$.  
    Moreover, $f$ equals a constant $f_1,f_2,f_3$ on the interiors of $I_1,I_2,$ and $I_3$, respectively.  
    Additionally since $I_1,I_2,$ and $I_3$ are separated from each other by at least one jump discontinuity, we may assume without loss of generality that $f_1 < f_2 < f_3$.  
    It follows that there exists a measurable subset $U\subseteq I_1\cup I_2\cup I_3$ of positive measure so that 
    \begin{align*}
        f_2^2 &= \dfrac{1}{m(U)}\int_{x\in U}f(x)^2.  
    \end{align*}
    By Corollary \ref{cor: averaging reduction}, $f$ is constant on $U$, a contradiction.  
    So Claim A holds on $P\cap S$.  
    For Claim A on $N\cap S$, we may repeat this argument with $P$ and $N$ interchanged, and $g$ and $-g$ interchanged.  
    \\
    
    Now we show the following claim.  \\
    \\
    \textbf{Claim B:} For a.e. $(x,y) \in (P\times P)\cup (N\times N)$ such that $f(x)\geq f(y)$, we have that for a.e. $z\in [0,1]$, $W(x,z) = 0$ implies that $W(y,z) = 0$.  \\
    
    We first prove the claim for a.e. $(x,y)\in P\times P$.  
    Suppose $W(y,z) = 0$.  By Lemma \ref{lem: K = indicator function}, in this case $z\in P$.
    Then for a.e. such $x,y$, by Lemma \ref{lem: local eigenfunction equation}, $g(x)\leq g(y)$.  
    By Lemma \ref{lem: K = indicator function}.\eqref{item: K = 0 or 1}, $W(x,z) = 0$ implies that $f(x)f(z) < g(x)g(z)$.  
    Since $f(x)\geq f(y)$ and $g(x)\leq g(y)$, $f(y)f(z) < g(y)g(z)$.  
    Again by Lemma \ref{lem: K = indicator function}.\eqref{item: K = 0 or 1}, $W(y,z) = 0$ for a.e. such $x,y,z$, as desired.  
    So the desired claim holds for a.e. $(x,y)\in P\times P$ such that $f(x)\geq f(y)$.  
    We may repeat the argument for a.e. $(x,y)\in N\times N$ to arrive at the same conclusion.  
    \\
    \\
    The next claim follows directly from Lemma \ref{lem: local eigenfunction equation}.    \\
    \\
    \textbf{Claim C:} For a.e. $x\in [0,1]$, $x\in C$ if and only if $f(x) \geq 1$, if and only if $|g(x)| \leq 1$.  
    \\
    \\
    Finally, we show the following claim.  
    \\
    \\
    \textbf{Claim D:} Except on a set of measure $0$,  $f$ takes on at most $3$ values on $P\cap C$, and at most $3$ values on $N\cap C$.  \\
    
    For a proof, we first write $P\cap S = S_1 \cup S_2$ so that $S_1,S_2$ are disjoint and $f$ equals some constant $f_1$ a.e. on $S_1$ and $f$ equals some constant $f_2$ a.e. on $S_2$.  
    By Lemma \ref{lem: local eigenfunction equation}, $g$ equals some constant $g_1$ a.e. on $S_1$ and $g$ equals some constant $g_2$ a.e. on $S_2$.  
    By definition of $P$, $g_1,g_2\geq 0$.  
    Now suppose $x\in P\cap C$ so that 
    \begin{align*}
        \mu f(x) &= \int_{y\in [0,1]}W(x,y)f(y).  
    \end{align*}
    Then by Lemma \ref{lem: K = indicator function}.\eqref{item: K = 0 or 1}, 
    \begin{align*}
        \mu f(x) 
        &= 
            \int_{y\in (P\cap C)\cup N}
                f(y)
            + \int_{y\in S_1}W(x,y)f(y)
            + \int_{y\in S_2}W(x,y)f(y)
        % \\
        % &=
        %     \int_{y\in (P\cap B)\cup N}
        %         f(y)
        %     + \int_{y\in W_1}
        %         \left\{\begin{array}{rl}
        %             f_1, & f(x)f_1 > g(x)g_1\\
        %             0, &\text{otherwise}
        %         \end{array}\right. 
        %     + \int_{y\in W_2}
        %         \left\{\begin{array}{rl}
        %             f_2, & f(x)f_2 > g(x)g_2\\
        %             0, &\text{otherwise}
        %         \end{array}\right.
        .  
    \end{align*}
    By Claim B, this expression for $\mu f(x)$ may take on at most $3$ values.  
    So the desired claim holds on $P\cap C$.  
    Repeating the same argument, the claim also holds on $N\cap C$.  \\
    
    We are nearly done with the proof of the theorem, as we have now reduced $W$ to a $10\times 10$ stepgraphon. To complete the proof, we show that we may reduce to at most $7\times 7$.  
    We now partition $P\cap C, P\cap S, N\cap C$, and $N\cap S$ so that $f$ and $g$ are constant a.e. on each part as: 
    \begin{itemize}
        \item $P\cap C = U_1\cup U_2\cup U_3$, 
        \item $P\cap S = U_4\cup U_5$, 
        \item $N\cap C = U_6\cup U_7\cup U_8$, and 
        \item $N\cap S = U_9\cup U_{10}$.  
    \end{itemize}
    Then by Lemma \ref{lem: K = indicator function}.\eqref{item: K = 0 or 1}, there exists a matrix $(m_{ij})_{i,j\in [10]}$ so that for all $(i,j)\in [10]\times [10]$, 
    \begin{itemize}
        \item $m_{ij}\in \{0,1\}$, 
        \item $W(x,y) = m_{ij}$ for a.e. $(x,y)\in U_i\times U_j$, 
        \item $m_{ij} = 1$ if and only if $f_if_j > g_ig_j$, and 
        \item $m_{ij} = 0$ if and only if $f_if_j < g_ig_j$.  
    \end{itemize}
    Additionally, we set $\alpha_i = m(U_i)$ and also denote by $f_i$ and $g_i$ the constant values of $f,g$ on each $U_i$, respectively, for each $i= 1, \ldots, 10$.  
    Furthermore, by Claim C and Lemma \ref{lem: K = indicator function} we assume without loss of generality that that $f_1 > f_2 > f_3 \geq 1 > f_4 > f_5$ and that $f_6 > f_7 > f_8 \geq 1 > f_9 > f_{10}$.  
    Also by Lemma \ref{lem: local eigenfunction equation}, $0 \leq g_1 < g_2 < g_3 \leq 1 < g_4 < g_5$ and $0 \leq -g_1 < -g_2 < -g_3 \leq 1 < -g_4 < -g_5$.  
    Also, by Claim B, no two columns of $m$ are identical within the sets $\{1,2,3,4,5\}$ and within $\{6,7,8,9,10\}$.  
    Shading $m_{ij} = 1$ black and $m_{ij} = 0$ white, we let 
    \begin{align*}
        M = 
        \begin{tabular}{||ccc||cc|||ccc||cc||}\hline\hline
            \black & \black & \black & \black & \black & \black & \black & \black & \black & \black \\
            \black & \black & \black & \black &  & \black & \black & \black & \black & \black \\
            \black & \black & \black &  &  & \black & \black & \black & \black & \black \\\hline\hline
            \black & \black &  &  &  & \black & \black & \black & \black & \black \\
            \black &  &  &  &  & \black & \black & \black & \black & \black \\\hline\hline\hline
            \black & \black & \black & \black & \black & \black & \black & \black & \black & \black \\
            \black & \black & \black & \black & \black & \black & \black & \black & \black &  \\
            \black & \black & \black & \black & \black & \black & \black & \black &  &  \\\hline\hline
            \black & \black & \black & \black & \black & \black & \black &  &  &  \\
            \black & \black & \black & \black & \black & \black &  &  &  &  \\\hline\hline
        \end{tabular}
        \quad .
    \end{align*}
   Therefore, $W$ is a stepgraphon with values determined by $M$ and the size of each block determined by the $\alpha_i$.  \\
    
    We claim that $0\in \{\alpha_3, \alpha_4, \alpha_5\}$ and $0\in\{\alpha_8,\alpha_9,\alpha_{10}\}$.  
    For the first claim,  assume to the contrary that all of $\alpha_3, \alpha_4, \alpha_5$ are positive and note that there exists some $x_4\in U_4$ such that 
    \begin{align*}
        \mu f_4 = \mu f(x_4) = \int_{y\in [0,1]}W(x_4,y)f(y).  
    \end{align*}
    Moreover for some measurable subsets $U_3'\subseteq U_3$ and $U_5'\subseteq U_5$ of positive measure so that with $U := U_3'\cup U_4\cup U_5'$, 
    \begin{align*}
        f(x_4)^2 = \dfrac{1}{m(U)}\int_{y\in U}f(y)^2.  
    \end{align*}
    
     Note that by Lemma \ref{lem: local eigenfunction equation}, we may assume that $x_4$ is average on $U$ with respect to $g$ as well. The conditions of Corollary \ref{cor: averaging reduction} are met for $W$ with $A=U_3', B=U_4\cup U_5',x_0 = x_4$.  
    Since $\int_{A\times A}W(x,y)f(x)f(y) > 0$, this is a contradiction to the corollary, so the desired claim holds.  
    The same argument may be used to prove that $0\in \{\alpha_8,\alpha_9,\alpha_{10}\}$.  
    \\
    
    We now form the principal submatrix $M'$ by removing the $i$-th row and column from $M$ if and only if $\alpha_i = 0$.  Since $\alpha_i=0$, $W$ is a stepgraphon with values determined by $M'$.
    Let $M_P'$ denote the principal submatrix of $M'$ corresponding to the indices $i\in\{1,\dots,5\}$ so that $\alpha_i>0$. That is, $M_P'$ corresponds to the upper left hand block of $M$.
    We use red to indicate rows and columns present in $M$ but not $M_P'$.  When forming the submatrix $M_P'$, we borrow the internal subdivisions which are present in the definition of $M$ above to denote where $f\geq 1$ and where $f<1$ (or between $S \cap P$ and $C \cap P$). Note that this is not the same as what the internal divisions denote in the statement of the theorem.
    Since $0\in\{\alpha_3,\alpha_4,\alpha_5\}$, it follows that $M_P'$ is a principal submatrix of 
    \begin{align*}
        \begin{tabular}{||ccc||cc||}\hline\hline
            \black & \black & \redcell & \black & \black \\
            \black & \black & \redcell & \black & \\
            \redcell & \redcell & \redcell & \redcell & \redcell \\\hline\hline
            \black & \black & \redcell & & \\
            \black & & \redcell & & \\
            \hline\hline
        \end{tabular}
        \quad , \quad
        \begin{tabular}{||ccc||cc||}\hline\hline
            \black & \black & \black & \redcell & \black \\
            \black & \black & \black & \redcell & \\
            \black & \black & \black & \redcell & \\\hline\hline
            \redcell & \redcell & \redcell & \redcell & \redcell \\
            \black & & & \redcell & \\
            \hline\hline
        \end{tabular}
        \quad, 
        \text{ or }\quad 
        \begin{tabular}{||ccc||cc||}\hline\hline
            \black & \black & \black & \black & \redcell \\
            \black & \black & \black & \black & \redcell \\
            \black & \black & \black & & \redcell\\\hline\hline
            \black & \black & & & \redcell \\
            \redcell & \redcell & \redcell & \redcell & \redcell \\
            \hline\hline
        \end{tabular}
        \quad .
    \end{align*}
    In the second case, columns $2$ and $3$ are identical in $M'$, and in the third case, columns $1$ and $2$ are identical in $M'$.  
    So without loss of generality, $M_P'$ is a principal submatrix of one of 
    \begin{align*}
        \begin{tabular}{||ccc||cc||}\hline\hline
            \black & \black & \redcell & \black & \black \\
            \black & \black & \redcell & \black & \\
            \redcell & \redcell & \redcell & \redcell & \redcell \\\hline\hline
            \black & \black & \redcell & & \\
            \black &  & \redcell & & \\
            \hline\hline
        \end{tabular}
        \quad , \quad 
        \begin{tabular}{||ccc||cc||}\hline\hline
            \black & \redcell & \black & \redcell & \black \\
            \redcell & \redcell & \redcell & \redcell & \redcell \\
            \black & \redcell & \black & \redcell & \\\hline\hline
            \redcell & \redcell & \redcell & \redcell & \redcell \\
            \black & \redcell & & \redcell & \\
            \hline\hline
        \end{tabular}
        \quad, 
        \text{ or }\quad 
        \begin{tabular}{||ccc||cc||}\hline\hline
            \black & \redcell & \black & \black & \redcell \\
            \redcell & \redcell & \redcell & \redcell & \redcell \\
            \black & \redcell & \black & & \redcell\\\hline\hline
            \black & \redcell & & & \redcell \\
            \redcell & \redcell & \redcell & \redcell & \redcell \\
            \hline\hline
        \end{tabular}
        \quad .
    \end{align*}
    In each case, $M_P'$ is a principal submatrix of 
    \begin{align*}
        \begin{tabular}{||cc||cc||}\hline\hline
            \black & \black & \black & \black \\
            \black & \black & \black & \\\hline\hline
            \black & \black & & \\
            \black & & & \\
            \hline\hline
        \end{tabular}
        \quad .
    \end{align*}
    An identical argument shows that the principal submatrix of $M'$ on the indices $i\in\{6,\dots,10\}$ such that $\alpha_i>0$ is a principal submatrix of 
    \begin{align*}
        \begin{tabular}{||cc||cc||}\hline\hline
            \black & \black & \black & \black \\
            \black & \black & \black & \\\hline\hline
            \black & \black & & \\
            \black & & & \\
            \hline\hline
        \end{tabular}
        \quad .
    \end{align*}
    Finally, we note that $0\in \{\alpha_1,\alpha_6\}$.  
    Indeed otherwise the corresponding columns are identical in $M'$, a contradiction.  
    So without loss of generality, row and column $6$ were also removed from $M$ to form $M'$.  
    This completes the proof of the theorem.  
\end{proof}

\section{Spread maximum graphons}\label{sec:spread_graphon}

In this section, we complete the proof of the graphon version of the spread conjecture of Gregory, Hershkowitz, and Kirkland from \cite{gregory2001spread}. In particular, we prove the following theorem.
For convenience and completeness, we state this result in the following level of detail.  

\begin{theorem}\label{thm: spread maximum graphon}
    If $W$ is a graphon that maximizes spread, then $W$ may be represented as follows.  
    For all $(x,y)\in [0,1]^2$, 
    \begin{align*}
        W(x,y)
        =\left\{\begin{array}{rl}
            0, &(x,y)\in [2/3, 1]^2\\
            1, &\text{otherwise}
        \end{array}\right. .
    \end{align*}
    Furthermore,
    \begin{align*}
        \mu 
        &= 
            \dfrac{1+\sqrt{3}}{3}
        \quad 
        \text{ and }
        \quad 
        \nu 
        = 
            \dfrac{1-\sqrt{3}}{3}
    \end{align*}
    are the maximum and minimum eigenvalues of $W$, respectively, and if $f,g$ are unit eigenfunctions associated to $\mu,\nu$, respectively, then, up to a change in sign, they may be written as follows.  
    For every $x\in [0,1]$, 
    \begin{align*}
        f(x) 
        &= 
            \dfrac{1}{2\sqrt{3+\sqrt{3}}}
            \cdot 
            \left\{\begin{array}{rl}
                3+\sqrt{3}, 
                &
                    x\in [0,2/3]
                \\
                2\cdot\sqrt{3}
                &
                    \text{otherwise}
            \end{array}\right. 
        , \text{ and }\\
        g(x) 
        &= 
            \dfrac{1}{2\sqrt{3-\sqrt{3}}}
            \cdot 
            \left\{\begin{array}{rl}
                3-\sqrt{3}, 
                &
                    x\in [0,2/3]
                \\
                -2\cdot\sqrt{3}
                &
                    \text{otherwise}
            \end{array}\right. 
        .  
    \end{align*}
    \end{theorem}

To help outline our proof of Theorem \ref{thm: spread maximum graphon}, let the spread-extremal graphon have block sizes $\alpha_1,\ldots, \alpha_7$. Note that the spread of the graphon is the same as the spread of matrix $M^*$ in Figure \ref{fig: 7-vertex loop graph}, and so we will optimize the spread of $M^*$ over choices of $\alpha_1,\dots, \alpha_7$. Let $G^*$ be the unweighted graph (with loops) corresponding to the matrix. 

We proceed in the following steps.  
\begin{enumerate}[1. ]
    \item 
        In Section \ref{appendix 17 cases}, we reduce the proof of Theorem \ref{thm: spread maximum graphon} to $17$ cases, each corresponding to a subset $S$ of $V(G^*)$. For each such $S$ we define an optimization problem $\SPRS$, the solution to which gives us an upper bound on the spread of any graphon in the case corresponding to $S$.
    \item 
        In Section \ref{sub-sec: numerics}, we appeal to interval arithmetic to translate these optimization problems into algorithms.  
        Based on the output of the $17$ programs we wrote, we eliminate $15$ of the $17$ cases.  
        We address the multitude of formulas used throughout and relocate their statements and proofs to Appendix \ref{sub-sec: formulas}.  
    \item 
        Finally in Section \ref{sub-sec: cases 4|57 and 1|7}, we complete the proof of Theorem \ref{thm: spread maximum graphon} by analyzing the $2$ remaining cases.  
        Here, we apply Vi\`ete's Formula for roots of cubic equations and make a direct argument.  
\end{enumerate}
\begin{comment}
\begin{figure}[ht]
    \centering
\[\left[\begin{array}{ccccccc}
        \alpha_1 &\sqrt{\alpha_1\alpha_2}&\sqrt{\alpha_1\alpha_3}&\sqrt{\alpha_1\alpha_4}&\sqrt{\alpha_1\alpha_5}&\sqrt{\alpha_1\alpha_6}&\sqrt{\alpha_1\alpha_7}\\
        \sqrt{\alpha_1\alpha_2}&\alpha_2&\sqrt{\alpha_2\alpha_3}&0&\sqrt{\alpha_2\alpha_5}&\sqrt{\alpha_2\alpha_6}&\sqrt{\alpha_2\alpha_7}\\
        \sqrt{\alpha_1\alpha_3}&\sqrt{\alpha_2\alpha_3}&0 &0&\sqrt{\alpha_3\alpha_5}&\sqrt{\alpha_3\alpha_6}&\sqrt{\alpha_3\alpha_7}\\
        \sqrt{\alpha_1\alpha_4}&0&0&0&\sqrt{\alpha_4\alpha_5}&\sqrt{\alpha_4\alpha_6}&\sqrt{\alpha_4\alpha_7}\\
        \sqrt{\alpha_1\alpha_5} &\sqrt{\alpha_2\alpha_5}&\sqrt{\alpha_3\alpha_5}&\sqrt{\alpha_4\alpha_5}&\alpha_5&\sqrt{\alpha_5\alpha_6}&0 \\
        \sqrt{\alpha_1\alpha_6}&\sqrt{\alpha_2\alpha_6}&\sqrt{\alpha_3\alpha_6}&\sqrt{\alpha_4\alpha_6}&\sqrt{\alpha_5\alpha_6}&0&0\\
        \sqrt{\alpha_1\alpha_7}&\sqrt{\alpha_2\alpha_7}&\sqrt{\alpha_3\alpha_7}&\sqrt{\alpha_4\alpha_7}&0&0&0
    \end{array}\right]\]
\newline
    \scalebox{.8}[.8]{
             \input{graphics/7-vertex-graph-all-edges}
         }
    \caption{{\color{blue}The graph $G^*$ corresponding to the matrix $M^*$. } }
    \label{fig: 7-vertex loop graph}
    
\end{figure}
\end{comment}

\begin{figure}[ht]
    \centering
    \begin{minipage}{0.55\textwidth}
\[M^* := D_\alpha^{1/2}\left[\begin{array}{cccc|ccc}
        1 & 1 & 1 & 1 & 1 & 1 & 1 \\
        1 & 1 & 1 & 0 & 1 & 1 & 1 \\
        1 & 1 & 0 & 0 & 1 & 1 & 1 \\
        1 & 0 & 0 & 0 & 1 & 1 & 1 \\\hline
        1 & 1 & 1 & 1 & 1 & 1 & 0 \\
        1 & 1 & 1 & 1 & 1 & 0 & 0 \\
        1 & 1 & 1 & 1 & 0 & 0 & 0 
    \end{array}\right]D_\alpha^{1/2}\]\end{minipage}\quad\begin{minipage}{0.4\textwidth}
    \scalebox{.7}[.7]{
             \begin{tikzpicture}
	
	\usetikzlibrary{backgrounds}
	\usetikzlibrary{patterns}
	\usetikzlibrary{decorations,calc}
	
	\pgfdeclarelayer{bg}
	\pgfdeclarelayer{fg}
	\pgfsetlayers{bg,main,fg}

	\tikzstyle{my-vx}=[draw, circle, very thick, draw=black, fill=white]
	\tikzstyle{my-arc}=[draw=black, very thick]

	\begin{pgfonlayer}{fg}
	
		\node [my-vx] (v1) at (-3.5,2) {$1$};
		\node [my-vx] (v2) at (-3.5,0.5) {$2$};
		\node [my-vx] (v3) at (-3.5,-1) {$3$};
		\node [my-vx] (v4) at (-3.5,-2.5) {$4$};
		
		\node [my-vx] (v5) at (1.5,1.5) {$5$};
		\node [my-vx] (v6) at (1.5,0) {$6$};
		\node [my-vx] (v7) at (1.5,-1.5) {$7$};

		\draw [my-arc] (v1) -- (v2);
		\draw [my-arc] (v1) edge[loop, looseness=4] (v1);
		\draw [my-arc] (v1) edge[bend right=25] (v3);
		\draw [my-arc] (v1) edge[bend right=35] (v4);
		
		\draw [my-arc] (v2) edge[loop, looseness=3.5, min distance=0mm, in = 45, out=315] (v2);
		\draw [my-arc] (v2) edge (v3);
		
		\draw [my-arc] (v5) edge[loop, looseness=3.5, min distance=0mm, in=45, out=135] (v5);
		\draw [my-arc] (v5) edge (v6);

		%\draw [my-arc] (v1) edge (v5);
		%\draw [my-arc] (v1) edge (v6);
		%\draw [my-arc] (v1) edge (v7);
		
		%\draw [my-arc] (v2) edge (v5);
		%\draw [my-arc] (v2) edge (v6);
		%\draw [my-arc] (v2) edge (v7);
		
		%\draw [my-arc] (v3) edge (v5);
		%\draw [my-arc] (v3) edge (v6);
		%\draw [my-arc] (v3) edge (v7);
		
		%\draw [my-arc] (v4) edge (v5);
		%\draw [my-arc] (v4) edge (v6);
		%\draw [my-arc] (v4) edge (v7);

	\end{pgfonlayer}
	
	\begin{pgfonlayer}{main}
		\draw [my-arc, fill=cyan] (-3.5,0) ellipse (1.1 and 3);
		\draw [my-arc, fill=pink] (1.5,0) ellipse (0.9 and 2.5);
		
		%\draw [my-arc, fill=pink]  node (v9) {} ellipse (2.5 and 1.5);
		%\draw [my-arc, fill=cyan]  node (v8) {} ellipse (3.0 and 2);
	\end{pgfonlayer}

\draw [my-arc] (v1) edge (v5);
\draw [my-arc] (v1) edge (v6);
\draw [my-arc] (v1) edge (v7);
\draw [my-arc] (v2) edge (v5);
\draw [my-arc] (v2) edge (v6);
\draw [my-arc] (v2) edge (v7);
\draw [my-arc] (v3) edge (v5);
\draw [my-arc] (v3) edge (v6);
\draw [my-arc] (v3) edge (v7);
\draw [my-arc] (v4) edge (v5);
\draw [my-arc] (v4) edge (v6);
\draw [my-arc] (v4) edge (v7);
	\begin{pgfonlayer}{bg}
		%\draw [my-arc, line width=100] (v8.center) edge (v9.center);
		\coordinate (c1) at (-3.5,-3) {} {} {} {} {} {} {} {} {} {} {} {} {} {} {} {} {} {};
		\coordinate (c2) at (1.5,-2.5) {} {} {} {} {} {} {} {} {} {} {} {} {} {} {} {} {} {};
		\coordinate (c3) at (1.5,2.5) {} {} {} {} {} {} {} {} {} {} {} {} {} {} {} {} {};
		\coordinate (c4) at (-3.5,3) {} {} {} {} {} {} {} {} {} {} {} {} {} {} {} {};
		
		%\draw [fill=black] (c1) -- (c2) -- (c3) -- (c4) -- cycle;

	\end{pgfonlayer}

	%\draw [my-arc, fill = cyan] (-4.5,9.5) rectangle (-3.5,8.5);
	%\draw [my-arc, fill = pink] (-0.5,8.5) rectangle (0.5,9.5);

\end{tikzpicture}
         }
         \end{minipage}
    \caption{{The matrix $M^*$ with corresponding graph $G^*$, where $D_\alpha$ is the diagonal matrix with entries $\alpha_1, \ldots, \alpha_7$. } }
    \label{fig: 7-vertex loop graph}
    
\end{figure}
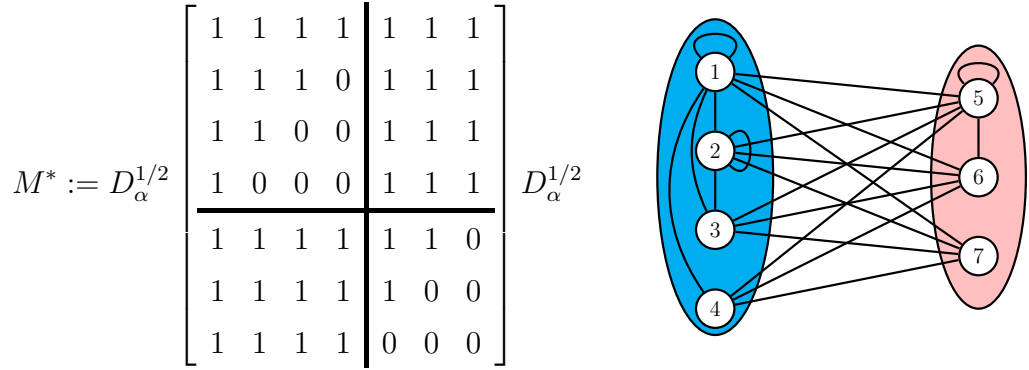

For concreteness, we define $G^*$ on the vertex set $\{1,\dots,7\}$.  
Explicitly, the neighborhoods $N_1,\dots, N_7$ of $1,\dots, 7$ are defined as: 
\begin{align*}
    \begin{array}{ll}
        N_1 := \{1,2,3,4,5,6,7\}
        & 
            N_2 := \{1,2,3,5,6,7\}
        \\
        N_3 := \{1,2,5,6,7\}
        & 
            N_4 := \{1,5,6,7\}
        \\
        N_5 := \{1,2,3,4,5,6\}
        & 
            N_6 := \{1,2,3,4,5\}
        \\
        N_7 := \{1,2,3,4\}
    \end{array}
    .
\end{align*}
More compactly, we may note that 
\begin{align*}
    \begin{array}{llll}
        N_1 = \{1,\dots, 7\}
        & 
            N_2 = N_1\setminus\{4\}
        &
            N_3 = N_2\setminus\{3\}
        & 
            N_4 = N_3\setminus\{2\}
        \\
        &
            N_5 = N_1\setminus\{7\}
        & 
            N_6 = N_5\setminus\{6\}
        &
            N_7 = N_6\setminus\{5\}
    \end{array}
    .
\end{align*}

\subsection{Stepgraphon case analysis}\label{sub-sec: cases}
{Let $W$ be a graphon maximizing spread. By Theorem \ref{thm: reduction to stepgraphon}, we may assume that $W$ is a $7\times 7$ stepgraphon corresponding to $G^\ast$. We will break into cases depending on which of the $7$ weights $\alpha_1, \ldots \alpha_7$ are zero and which are positive. {For some of these combinations the corresponding graphons are isomorphic}, and in this section we will outline how one can show that we need only consider $17$ cases rather than $2^7$. 

We will present each case with the set of indices which have strictly positive weight. Additionally, we will use vertical bars to partition the set of integers according to its intersection with the sets $\{1\}$, $\{2,3,4\}$ and $\{5,6,7\}$. Recall that vertices in block $1$ are dominating vertices and vertices in blocks $5$, $6$, and $7$ have negative entries in the eigenfunction corresponding to $\nu$. For example, we use $4|57$ to refer to the case that $\alpha_4, \alpha_5, \alpha_7$ are all positive and $\alpha_1 = \alpha_2 = \alpha_3 = \alpha_6 = 0$; see Figure \ref{457 fig}.

\begin{figure}[ht]
\begin{center}
\begin{minipage}{0.45\textwidth}\centering
\begin{tabular}{||c||cc||}\hline\hline
            & \black & \black  \\ \hline \hline
            \black & \black &  \\ 
            \black &  &  \\\hline\hline
\end{tabular} 
\end{minipage}\quad\begin{minipage}{0.45\textwidth}\centering \scalebox{.7}[.7]{\begin{tikzpicture}
	
	\usetikzlibrary{backgrounds}
	\usetikzlibrary{patterns}
	\usetikzlibrary{decorations,calc}
	
	\pgfdeclarelayer{bg}
	\pgfdeclarelayer{fg}
	\pgfsetlayers{bg,main,fg}

	\tikzstyle{my-vx}=[draw, circle, very thick, draw=black, fill=white]
	\tikzstyle{my-arc}=[draw=black, very thick]

	\begin{pgfonlayer}{fg}
	
		%\node [my-vx] (v1) at (-3.5,2) {$1$};
		%\node [my-vx] (v2) at (-3.5,0.5) {$2$};
		%\node [my-vx] (v3) at (-3.5,-1) {$3$};
		\node [my-vx] (v4) at (-2,0) {$4$};
		
		\node [my-vx] (v5) at (1.5,.75) {$5$};
		%\node [my-vx] (v6) at (1.5,0) {$6$};
		\node [my-vx] (v7) at (1.5,-.75) {$7$};

		%\draw [my-arc] (v1) -- (v2);
		%\draw [my-arc] (v1) edge[loop, looseness=4] (v1);
		%\draw [my-arc] (v1) edge[bend right=25] (v3);
		%\draw [my-arc] (v1) edge[bend right=35] (v4);
		
		%\draw [my-arc] (v2) edge[loop, looseness=3.5, min distance=0mm, in = 45, out=315] (v2);
		%\draw [my-arc] (v2) edge (v3);
		
		\draw [my-arc] (v5) edge[loop, looseness=3.5, min distance=0mm, in=45, out=135] (v5);
		%\draw [my-arc] (v5) edge (v6);

		%\draw [my-arc] (v1) edge (v5);
		%\draw [my-arc] (v1) edge (v6);
		%\draw [my-arc] (v1) edge (v7);
		
		%\draw [my-arc] (v2) edge (v5);
		%\draw [my-arc] (v2) edge (v6);
		%\draw [my-arc] (v2) edge (v7);
		
		%\draw [my-arc] (v3) edge (v5);
		%\draw [my-arc] (v3) edge (v6);
		%\draw [my-arc] (v3) edge (v7);
		
		%\draw [my-arc] (v4) edge (v5);
		%\draw [my-arc] (v4) edge (v6);
		%\draw [my-arc] (v4) edge (v7);

	\end{pgfonlayer}
	
	\begin{pgfonlayer}{main}
		\draw [my-arc, fill=cyan] (-2,0) ellipse (.7 and 0.8);
		\draw [my-arc, fill=pink] (1.5,0) ellipse (0.8 and 2.0);
		
		%\draw [my-arc, fill=pink]  node (v9) {} ellipse (2.5 and 1.5);
		%\draw [my-arc, fill=cyan]  node (v8) {} ellipse (3.0 and 2);
	\end{pgfonlayer}

%\draw [my-arc] (v1) edge (v5);
%\draw [my-arc] (v1) edge (v6);
%\draw [my-arc] (v1) edge (v7);
%\draw [my-arc] (v2) edge (v5);
%\draw [my-arc] (v2) edge (v6);
%\draw [my-arc] (v2) edge (v7);
%\draw [my-arc] (v3) edge (v5);
%\draw [my-arc] (v3) edge (v6);
%\draw [my-arc] (v3) edge (v7);
\draw [my-arc] (v4) edge (v5);
%\draw [my-arc] (v4) edge (v6);
\draw [my-arc] (v4) edge (v7);
	\begin{pgfonlayer}{bg}
		%\draw [my-arc, line width=100] (v8.center) edge (v9.center);

		%\draw [fill=black] (c1) -- (c2) -- (c3) -- (c4) -- cycle;

	\end{pgfonlayer}

	%\draw [my-arc, fill = cyan] (-4.5,9.5) rectangle (-3.5,8.5);
	%\draw [my-arc, fill = pink] (-0.5,8.5) rectangle (0.5,9.5);

\end{tikzpicture}}\end{minipage}
\end{center}
    \caption{The family of graphons and the graph corresponding to case $4|57$}
    \label{457 fig}
\end{figure}
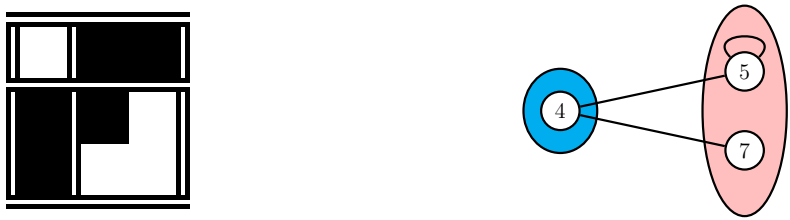

To give an upper bound on the spread of any graphon { corresponding to case} $4|57$ we solve a constrained optimization problem. Let $f_4, f_5, f_7$ and $g_4, g_5, g_7$ denote the eigenfunction entries for unit eigenfunctions $f$ and $g$ of the graphon. Then we maximize $\mu - \nu$ subject to

\begin{align*}
    \alpha_4 + \alpha_5 + \alpha_7 &=1\\
    \alpha_4f_4^2 + \alpha_5f_5^2 + \alpha_7 f_7^2 &=1 \\
    \alpha_4g_4^2 + \alpha_5g_5^2 + \alpha_7 g_7^2 &= 1\\
    \mu f_i^2 - \nu g_i^2 &= \mu-\nu \quad \text{for all } i \in \{4,5,7\}\\
  \mu f_4 = \alpha_5 f_5 + \alpha_7 f_7, \quad \mu f_5 &= \alpha_4 f_4 + \alpha_5 f_5, \quad  \mu f_7 = \alpha_4 f_4\\
   \nu g_4 = \alpha_5 g_5 + \alpha_7 g_7, \quad \nu g_5 &= \alpha_4 g_4 + \alpha_5 g_5, \quad  \nu g_7 = \alpha_4 g_4
\end{align*}

The first three constraints say that the weights sum to $1$ and that $f$ and $g$ are unit eigenfunctions. The fourth constraint is from Lemma \ref{lem: local eigenfunction equation}. The final two lines of constraints say that $f$ and $g$ are eigenfunctions for $\mu$ and $\nu$ respectively. Since these equations must be satisfied for any spread-extremal graphon, the solution to this optimization problem gives an upper bound on any spread-extremal graphon {corresponding to} case $4|57$. For each case we formulate a similar optimization problem in Appendix \ref{appendix 17 cases}.

First, if two distinct blocks of vertices have the same neighborhood, then without loss of generality we may assume that only one of them has positive weight. For example, see Figure \ref{123567 fig}: in case $123|567$, blocks $1$ and $2$ have the same neighborhood, and hence without loss of generality we may assume that only block $1$ has positive weight. Furthermore, in this case the resulting graphon could be considered as case $13|567$ or equivalently as case $14|567$; the graphons corresponding to these cases are isomorphic. Therefore cases $123|567$, $13|567$, and $14|567$ reduce to considering only case $14|567$. 

\begin{figure}[ht]
\begin{center}
\begin{minipage}{0.3\textwidth}\centering
\begin{tabular}{||ccc||ccc||}\hline\hline
            \black & \black & \black & \black & \black & \black \\
            \black & \black & \black & \black & \black & \black  \\
            \black & \black &  & \black & \black & \black  \\ \hline \hline
            \black & \black & \black & \black & \black &  \\
            \black & \black & \black & \black &  &  \\
            \black & \black & \black &  &  &  \\\hline\hline
\end{tabular} 
\end{minipage}\quad \begin{minipage}{0.3\textwidth}\centering
\begin{tabular}{||cc||ccc||}\hline\hline
            \black & \black & \black & \black & \black \\
            \black &  & \black & \black & \black  \\ \hline \hline
            \black & \black & \black & \black &  \\
            \black & \black & \black &  &  \\
            \black & \black &  &  &  \\\hline\hline
\end{tabular}\end{minipage}\quad\begin{minipage}{0.3\textwidth}\centering
\begin{tabular}{||cc||ccc||}\hline\hline
            \black & \black & \black & \black & \black \\
            \black &  & \black & \black & \black  \\ \hline \hline
            \black & \black & \black & \black &  \\
            \black & \black & \black &  &  \\
            \black & \black &  &  &  \\\hline\hline
\end{tabular} 
\end{minipage}
\newline\vspace{10pt}

\begin{minipage}{0.3\textwidth}\centering \scalebox{.7}[.7]{\begin{tikzpicture}
	
	\usetikzlibrary{backgrounds}
	\usetikzlibrary{patterns}
	\usetikzlibrary{decorations,calc}
	
	\pgfdeclarelayer{bg}
	\pgfdeclarelayer{fg}
	\pgfsetlayers{bg,main,fg}

	\tikzstyle{my-vx}=[draw, circle, very thick, draw=black, fill=white]
	\tikzstyle{my-arc}=[draw=black, very thick]

	\begin{pgfonlayer}{fg}
	
		\node [my-vx] (v1) at (-3,1.5) {$1$};
		\node [my-vx] (v2) at (-3,0) {$2$};
		\node [my-vx] (v3) at (-3,-1.5) {$3$};
		
		\node [my-vx] (v5) at (1.5,1.5) {$5$};
		\node [my-vx] (v6) at (1.5,0) {$6$};
		\node [my-vx] (v7) at (1.5,-1.5) {$7$};

		\draw [my-arc] (v1) -- (v2);
		\draw [my-arc] (v1) edge[loop, looseness=4] (v1);
		\draw [my-arc] (v1) edge[bend right=45] (v3);
		
		\draw [my-arc] (v2) edge[loop, looseness=3.5, min distance=0mm, in = 225, out=135] (v2);
		\draw [my-arc] (v2) edge (v3);
		
		\draw [my-arc] (v5) edge[loop, looseness=3.5, min distance=0mm, in=45, out=135] (v5);
		\draw [my-arc] (v5) edge (v6);

		%\draw [my-arc] (v1) edge (v5);
		%\draw [my-arc] (v1) edge (v6);
		%\draw [my-arc] (v1) edge (v7);
		
		%\draw [my-arc] (v2) edge (v5);
		%\draw [my-arc] (v2) edge (v6);
		%\draw [my-arc] (v2) edge (v7);
		
		%\draw [my-arc] (v3) edge (v5);
		%\draw [my-arc] (v3) edge (v6);
		%\draw [my-arc] (v3) edge (v7);
		
		%\draw [my-arc] (v4) edge (v5);
		%\draw [my-arc] (v4) edge (v6);
		%\draw [my-arc] (v4) edge (v7);

	\end{pgfonlayer}
	
	\begin{pgfonlayer}{main}
		\draw [my-arc, fill=cyan] (-3,0) ellipse (1.2 and 2.5);
		\draw [my-arc, fill=pink] (1.5,0) ellipse (0.9 and 2.5);
		
		%\draw [my-arc, fill=pink]  node (v9) {} ellipse (2.5 and 1.5);
		%\draw [my-arc, fill=cyan]  node (v8) {} ellipse (3.0 and 2);
	\end{pgfonlayer}

\draw [my-arc] (v1) edge (v5);
\draw [my-arc] (v1) edge (v6);
\draw [my-arc] (v1) edge (v7);
\draw [my-arc] (v2) edge (v5);
\draw [my-arc] (v2) edge (v6);
\draw [my-arc] (v2) edge (v7);
\draw [my-arc] (v3) edge (v5);
\draw [my-arc] (v3) edge (v6);
\draw [my-arc] (v3) edge (v7);
	\begin{pgfonlayer}{bg}
		%\draw [my-arc, line width=100] (v8.center) edge (v9.center);
		\coordinate  {} {} {} {} {} {} {} {} {} {} {} {} {} {} {} {} {} {};
		\coordinate  {} {} {} {} {} {} {} {} {} {} {} {} {} {} {} {} {} {};
		\coordinate  {} {} {} {} {} {} {} {} {} {} {} {} {} {} {} {} {};
		\coordinate  {} {} {} {} {} {} {} {} {} {} {} {} {} {} {} {};
		
		%\draw [fill=black] (c1) -- (c2) -- (c3) -- (c4) -- cycle;

	\end{pgfonlayer}

	%\draw [my-arc, fill = cyan] (-4.5,9.5) rectangle (-3.5,8.5);
	%\draw [my-arc, fill = pink] (-0.5,8.5) rectangle (0.5,9.5);

\end{tikzpicture}}\end{minipage}\quad\begin{minipage}{0.3\textwidth}\centering \scalebox{.7}[.7]{\begin{tikzpicture}

	\usetikzlibrary{backgrounds}
	\usetikzlibrary{patterns}
	\usetikzlibrary{decorations,calc}
	
	\pgfdeclarelayer{bg}
	\pgfdeclarelayer{fg}
	\pgfsetlayers{bg,main,fg}

	\tikzstyle{my-vx}=[draw, circle, very thick, draw=black, fill=white]
	\tikzstyle{my-arc}=[draw=black, very thick]

	\begin{pgfonlayer}{fg}
	
		\node [my-vx] (v1) at (-3,1) {$1$};
		\node [my-vx] (v3) at (-3,-1) {$3$};
		
		\node [my-vx] (v5) at (1.5,1.5) {$5$};
		\node [my-vx] (v6) at (1.5,0) {$6$};
		\node [my-vx] (v7) at (1.5,-1.5) {$7$};

		\draw [my-arc] (v1) edge[loop, looseness=4] (v1);
		\draw [my-arc] (v1) edge[bend right=0] (v3);
		
		\draw [my-arc] (v5) edge[loop, looseness=3.5, min distance=0mm, in=45, out=135] (v5);
		\draw [my-arc] (v5) edge (v6);

		%\draw [my-arc] (v1) edge (v5);
		%\draw [my-arc] (v1) edge (v6);
		%\draw [my-arc] (v1) edge (v7);
		
		%\draw [my-arc] (v2) edge (v5);
		%\draw [my-arc] (v2) edge (v6);
		%\draw [my-arc] (v2) edge (v7);
		
		%\draw [my-arc] (v3) edge (v5);
		%\draw [my-arc] (v3) edge (v6);
		%\draw [my-arc] (v3) edge (v7);
		
		%\draw [my-arc] (v4) edge (v5);
		%\draw [my-arc] (v4) edge (v6);
		%\draw [my-arc] (v4) edge (v7);

	\end{pgfonlayer}
	
	\begin{pgfonlayer}{main}
		\draw [my-arc, fill=cyan] (-3,0) ellipse (0.9 and 2.2);
		\draw [my-arc, fill=pink] (1.5,0) ellipse (0.9 and 2.5);
		
		%\draw [my-arc, fill=pink]  node (v9) {} ellipse (2.5 and 1.5);
		%\draw [my-arc, fill=cyan]  node (v8) {} ellipse (3.0 and 2);
	\end{pgfonlayer}

\draw [my-arc] (v1) edge (v5);
\draw [my-arc] (v1) edge (v6);
\draw [my-arc] (v1) edge (v7);
\draw [my-arc] (v3) edge (v5);
\draw [my-arc] (v3) edge (v6);
\draw [my-arc] (v3) edge (v7);
	\begin{pgfonlayer}{bg}
		%\draw [my-arc, line width=100] (v8.center) edge (v9.center);
		\coordinate  {} {} {} {} {} {} {} {} {} {} {} {} {} {} {} {} {} {};
		\coordinate  {} {} {} {} {} {} {} {} {} {} {} {} {} {} {} {} {} {};
		\coordinate  {} {} {} {} {} {} {} {} {} {} {} {} {} {} {} {} {};
		\coordinate  {} {} {} {} {} {} {} {} {} {} {} {} {} {} {} {};
		
		%\draw [fill=black] (c1) -- (c2) -- (c3) -- (c4) -- cycle;

	\end{pgfonlayer}

	%\draw [my-arc, fill = cyan] (-4.5,9.5) rectangle (-3.5,8.5);
	%\draw [my-arc, fill = pink] (-0.5,8.5) rectangle (0.5,9.5);

\end{tikzpicture}}\end{minipage}\quad\begin{minipage}{0.3\textwidth}\centering \scalebox{.7}[.7]{\begin{tikzpicture}
	
	\usetikzlibrary{backgrounds}
	\usetikzlibrary{patterns}
	\usetikzlibrary{decorations,calc}
	
	\pgfdeclarelayer{bg}
	\pgfdeclarelayer{fg}
	\pgfsetlayers{bg,main,fg}

	\tikzstyle{my-vx}=[draw, circle, very thick, draw=black, fill=white]
	\tikzstyle{my-arc}=[draw=black, very thick]

	\begin{pgfonlayer}{fg}
	
		\node [my-vx] (v1) at (-3,1) {$1$};
		\node [my-vx] (v3) at (-3,-1) {$4$};
		
		\node [my-vx] (v5) at (1.5,1.5) {$5$};
		\node [my-vx] (v6) at (1.5,0) {$6$};
		\node [my-vx] (v7) at (1.5,-1.5) {$7$};

		\draw [my-arc] (v1) edge[loop, looseness=4] (v1);
		\draw [my-arc] (v1) edge[bend right=0] (v3);

		\draw [my-arc] (v5) edge[loop, looseness=3.5, min distance=0mm, in=45, out=135] (v5);
		\draw [my-arc] (v5) edge (v6);

		%\draw [my-arc] (v1) edge (v5);
		%\draw [my-arc] (v1) edge (v6);
		%\draw [my-arc] (v1) edge (v7);
		
		%\draw [my-arc] (v2) edge (v5);
		%\draw [my-arc] (v2) edge (v6);
		%\draw [my-arc] (v2) edge (v7);
		
		%\draw [my-arc] (v3) edge (v5);
		%\draw [my-arc] (v3) edge (v6);
		%\draw [my-arc] (v3) edge (v7);
		
		%\draw [my-arc] (v4) edge (v5);
		%\draw [my-arc] (v4) edge (v6);
		%\draw [my-arc] (v4) edge (v7);

	\end{pgfonlayer}
	
	\begin{pgfonlayer}{main}
		\draw [my-arc, fill=cyan] (-3,0) ellipse (.9 and 2.0);
		\draw [my-arc, fill=pink] (1.5,0) ellipse (0.9 and 2.5);
		
		%\draw [my-arc, fill=pink]  node (v9) {} ellipse (2.5 and 1.5);
		%\draw [my-arc, fill=cyan]  node (v8) {} ellipse (3.0 and 2);
	\end{pgfonlayer}

\draw [my-arc] (v1) edge (v5);
\draw [my-arc] (v1) edge (v6);
\draw [my-arc] (v1) edge (v7);
\draw [my-arc] (v3) edge (v5);
\draw [my-arc] (v3) edge (v6);
\draw [my-arc] (v3) edge (v7);
	\begin{pgfonlayer}{bg}
		%\draw [my-arc, line width=100] (v8.center) edge (v9.center);
		\coordinate  {} {} {} {} {} {} {} {} {} {} {} {} {} {} {} {} {} {};
		\coordinate  {} {} {} {} {} {} {} {} {} {} {} {} {} {} {} {} {} {};
		\coordinate  {} {} {} {} {} {} {} {} {} {} {} {} {} {} {} {} {};
		\coordinate  {} {} {} {} {} {} {} {} {} {} {} {} {} {} {} {};
		
		%\draw [fill=black] (c1) -- (c2) -- (c3) -- (c4) -- cycle;

	\end{pgfonlayer}

	%\draw [my-arc, fill = cyan] (-4.5,9.5) rectangle (-3.5,8.5);
	%\draw [my-arc, fill = pink] (-0.5,8.5) rectangle (0.5,9.5);

\end{tikzpicture}}\end{minipage}
\end{center}
\caption{Redundancy, then renaming: we can assume $\alpha_2=0$ in the family of graphons corresponding to $123|567$, which produces families of graphons corresponding to both cases $13|567$ and $14|567$.}
\label{123567 fig}
\end{figure}
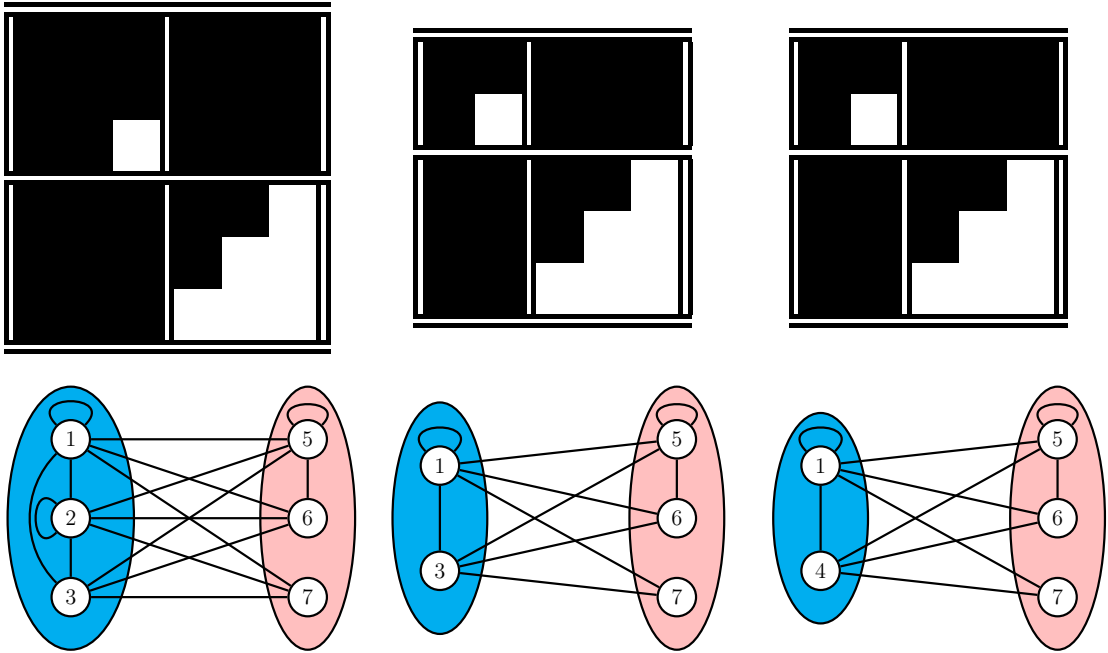

Additionally, if there is no dominant vertex, then some pairs cases may correspond to isomorphic graphons and the optimization problems are equivalent up to flipping the sign of the eigenvector corresponding to $\nu$. For example, see Figure \ref{23457 24567}, in which cases $23|457$ and $24|567$ reduce to considering only a single one. However, because of how we choose to order the eigenfunction entries when setting up the constraints of the optimization problems, there are some examples of cases corresponding to isomorphic graphons that we solve as separate optimization problems. For example, the graphons corresponding to cases $1|24|7$ and $1|4|57$ are isomorphic, but we will consider them separate cases; see Figure \ref{1247 1457}.

% you can swap these two

\begin{figure}[ht]
\begin{center}
\begin{minipage}{0.45\textwidth}\centering
\begin{tabular}{||ccc||cc||}\hline\hline
            \black & \black & & \black & \black  \\
            \black &  &  & \black & \black  \\ 
            &  &  & \black & \black \\ \hline \hline
            \black & \black & \black & \black &  \\
            \black & \black & \black &  &  \\\hline\hline
\end{tabular} 
\end{minipage}\quad\begin{minipage}{0.45\textwidth}\centering \scalebox{.7}[.7]{\begin{tikzpicture}
	
	\usetikzlibrary{backgrounds}
	\usetikzlibrary{patterns}
	\usetikzlibrary{decorations,calc}
	
	\pgfdeclarelayer{bg}
	\pgfdeclarelayer{fg}
	\pgfsetlayers{bg,main,fg}

	\tikzstyle{my-vx}=[draw, circle, very thick, draw=black, fill=white]
	\tikzstyle{my-arc}=[draw=black, very thick]

	\begin{pgfonlayer}{fg}
	
		\node [my-vx] (v2) at (-3.5,1.5) {$2$};
		\node [my-vx] (v3) at (-3.5,0) {$3$};
		\node [my-vx] (v4) at (-3.5,-1.5) {$4$};
		
		\node [my-vx] (v5) at (1.5,1) {$5$};
		\node [my-vx] (v7) at (1.5,-1) {$7$};

		\draw [my-arc] (v2) edge[loop, looseness=3.5, min distance=0mm, in = 135, out=45] (v2);
		\draw [my-arc] (v2) edge (v3);
		
		\draw [my-arc] (v5) edge[loop, looseness=3.5, min distance=0mm, in=45, out=135] (v5);

		%\draw [my-arc] (v1) edge (v5);
		%\draw [my-arc] (v1) edge (v6);
		%\draw [my-arc] (v1) edge (v7);
		
		%\draw [my-arc] (v2) edge (v5);
		%\draw [my-arc] (v2) edge (v6);
		%\draw [my-arc] (v2) edge (v7);
		
		%\draw [my-arc] (v3) edge (v5);
		%\draw [my-arc] (v3) edge (v6);
		%\draw [my-arc] (v3) edge (v7);
		
		%\draw [my-arc] (v4) edge (v5);
		%\draw [my-arc] (v4) edge (v6);
		%\draw [my-arc] (v4) edge (v7);

	\end{pgfonlayer}
	
	\begin{pgfonlayer}{main}
		\draw [my-arc, fill=cyan] (-3.5,0) ellipse (1.1 and 3);
		\draw [my-arc, fill=pink] (1.5,0) ellipse (0.9 and 2.5);
		
		%\draw [my-arc, fill=pink]  node (v9) {} ellipse (2.5 and 1.5);
		%\draw [my-arc, fill=cyan]  node (v8) {} ellipse (3.0 and 2);
	\end{pgfonlayer}

\draw [my-arc] (v2) edge (v5);
\draw [my-arc] (v2) edge (v7);
\draw [my-arc] (v3) edge (v5);
\draw [my-arc] (v3) edge (v7);
\draw [my-arc] (v4) edge (v5);
\draw [my-arc] (v4) edge (v7);
	\begin{pgfonlayer}{bg}
		%\draw [my-arc, line width=100] (v8.center) edge (v9.center);

		%\draw [fill=black] (c1) -- (c2) -- (c3) -- (c4) -- cycle;

	\end{pgfonlayer}

	%\draw [my-arc, fill = cyan] (-4.5,9.5) rectangle (-3.5,8.5);
	%\draw [my-arc, fill = pink] (-0.5,8.5) rectangle (0.5,9.5);

\end{tikzpicture}}\end{minipage}
\newline \vspace{20pt}

\begin{minipage}{0.45\textwidth}\centering
\begin{tabular}{||cc||ccc||}\hline\hline
            \black &  & \black & \black & \black \\
            &  & \black & \black & \black  \\ \hline \hline
            \black & \black & \black & \black &  \\ 
            \black & \black & \black &  &  \\
            \black & \black &  &  &  \\\hline\hline
\end{tabular} 
\end{minipage}\quad\begin{minipage}{0.45\textwidth}\centering \scalebox{.7}[.7]{\begin{tikzpicture}
	
	\usetikzlibrary{backgrounds}
	\usetikzlibrary{patterns}
	\usetikzlibrary{decorations,calc}
	
	\pgfdeclarelayer{bg}
	\pgfdeclarelayer{fg}
	\pgfsetlayers{bg,main,fg}

	\tikzstyle{my-vx}=[draw, circle, very thick, draw=black, fill=white]
	\tikzstyle{my-arc}=[draw=black, very thick]

	\begin{pgfonlayer}{fg}
	
		\node [my-vx] (v2) at (5.5,1.5) {$5$};
		\node [my-vx] (v3) at (5.5,0) {$6$};
		\node [my-vx] (v4) at (5.5,-1.5) {$7$};
		
		\node [my-vx] (v5) at (1.5,1) {$2$};
		\node [my-vx] (v7) at (1.5,-1) {$4$};

		\draw [my-arc] (v2) edge[loop, looseness=3.5, min distance=0mm, in = 135, out=45] (v2);
		\draw [my-arc] (v2) edge (v3);
		
		\draw [my-arc] (v5) edge[loop, looseness=3.5, min distance=0mm, in=45, out=135] (v5);

		%\draw [my-arc] (v1) edge (v5);
		%\draw [my-arc] (v1) edge (v6);
		%\draw [my-arc] (v1) edge (v7);
		
		%\draw [my-arc] (v2) edge (v5);
		%\draw [my-arc] (v2) edge (v6);
		%\draw [my-arc] (v2) edge (v7);
		
		%\draw [my-arc] (v3) edge (v5);
		%\draw [my-arc] (v3) edge (v6);
		%\draw [my-arc] (v3) edge (v7);
		
		%\draw [my-arc] (v4) edge (v5);
		%\draw [my-arc] (v4) edge (v6);
		%\draw [my-arc] (v4) edge (v7);

	\end{pgfonlayer}
	
	\begin{pgfonlayer}{main}
		\draw [my-arc, fill=pink] (5.5,0) ellipse (1.1 and 3);
		\draw [my-arc, fill=cyan] (1.5,0) ellipse (0.9 and 2.5);
		
		%\draw [my-arc, fill=pink]  node (v9) {} ellipse (2.5 and 1.5);
		%\draw [my-arc, fill=cyan]  node (v8) {} ellipse (3.0 and 2);
	\end{pgfonlayer}

\draw [my-arc] (v2) edge (v5);
\draw [my-arc] (v2) edge (v7);
\draw [my-arc] (v3) edge (v5);
\draw [my-arc] (v3) edge (v7);
\draw [my-arc] (v4) edge (v5);
\draw [my-arc] (v4) edge (v7);
	\begin{pgfonlayer}{bg}
		%\draw [my-arc, line width=100] (v8.center) edge (v9.center);

		%\draw [fill=black] (c1) -- (c2) -- (c3) -- (c4) -- cycle;

	\end{pgfonlayer}

	%\draw [my-arc, fill = cyan] (-4.5,9.5) rectangle (-3.5,8.5);
	%\draw [my-arc, fill = pink] (-0.5,8.5) rectangle (0.5,9.5);

\end{tikzpicture}}\end{minipage}
\end{center}
\caption{Changing the sign of $g$: the optimization problems in these cases are equivalent.}
\label{23457 24567}
\end{figure}
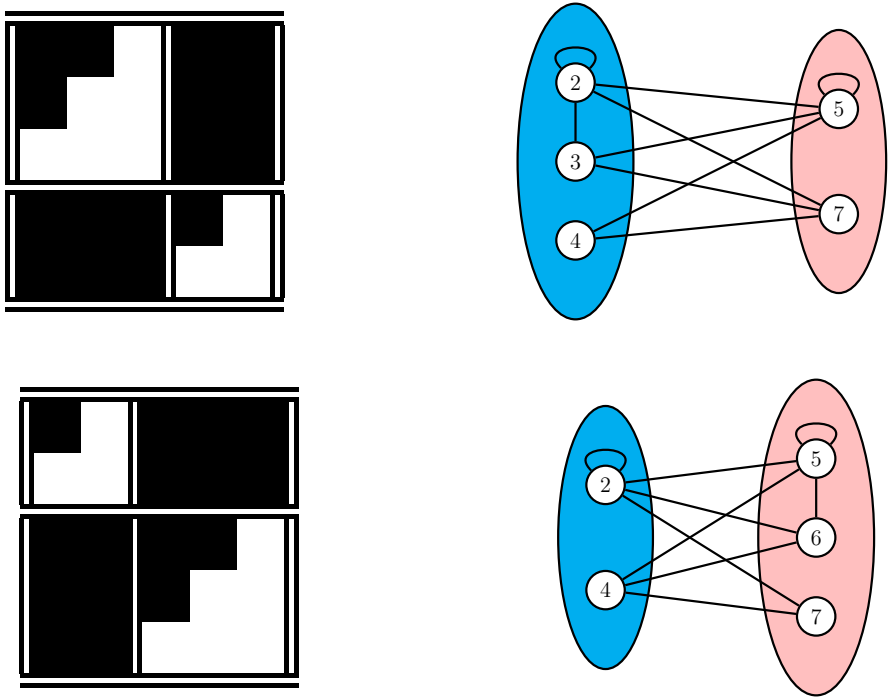
 
%% 24 567

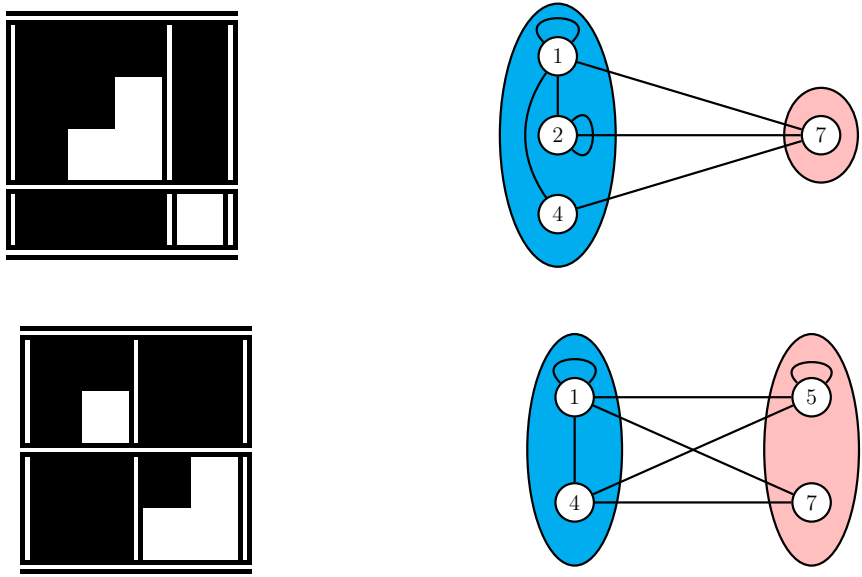
\begin{figure}
    \begin{center}
    \begin{minipage}{0.45\textwidth}\centering
    \begin{tabular}{||ccc||c||}\hline\hline
            \black & \black & \black & \black \\
            \black & \black &  & \black  \\
            \black &  &  & \black \\ \hline \hline
            \black & \black & \black &  \\\hline\hline
\end{tabular} 
    \end{minipage}\quad\begin{minipage}{0.45\textwidth}\centering \scalebox{.7}[.7]{\begin{tikzpicture}

	\usetikzlibrary{backgrounds}
	\usetikzlibrary{patterns}
	\usetikzlibrary{decorations,calc}
	
	\pgfdeclarelayer{bg}
	\pgfdeclarelayer{fg}
	\pgfsetlayers{bg,main,fg}

	\tikzstyle{my-vx}=[draw, circle, very thick, draw=black, fill=white]
	\tikzstyle{my-arc}=[draw=black, very thick]

	\begin{pgfonlayer}{fg}
	
		\node [my-vx] (v1) at (-3.5,1.5) {$1$};
		\node [my-vx] (v2) at (-3.5,0) {$2$};
		
		\node [my-vx] (v4) at (-3.5,-1.5) {$4$};

		\node [my-vx] (v7) at (1.5,0) {$7$};

		\draw [my-arc] (v1) -- (v2);
		\draw [my-arc] (v1) edge[loop, looseness=4] (v1);
		\draw [my-arc] (v1) edge[bend right=35] (v4);
		
		\draw [my-arc] (v2) edge[loop, looseness=3.5, min distance=0mm, in = 45, out=315] (v2);
		
	\end{pgfonlayer}
	
	\begin{pgfonlayer}{main}
		\draw [my-arc, fill=cyan] (-3.5,0) ellipse (1.1 and 2.5);
		\draw [my-arc, fill=pink] (1.5,0) ellipse (0.7 and .9);
		
		%\draw [my-arc, fill=pink]  node (v9) {} ellipse (2.5 and 1.5);
		%\draw [my-arc, fill=cyan]  node (v8) {} ellipse (3.0 and 2);
	\end{pgfonlayer}

\draw [my-arc] (v1) edge (v7);
\draw [my-arc] (v2) edge (v7);
\draw [my-arc] (v4) edge (v7);
	\begin{pgfonlayer}{bg}
		%\draw [my-arc, line width=100] (v8.center) edge (v9.center);
		\coordinate  {} {} {} {} {} {} {} {} {} {} {} {} {} {} {} {} {} {};
		\coordinate  {} {} {} {} {} {} {} {} {} {} {} {} {} {} {} {} {} {};
		\coordinate  {} {} {} {} {} {} {} {} {} {} {} {} {} {} {} {} {};
		\coordinate  {} {} {} {} {} {} {} {} {} {} {} {} {} {} {} {};
		
		%\draw [fill=black] (c1) -- (c2) -- (c3) -- (c4) -- cycle;

	\end{pgfonlayer}

	%\draw [my-arc, fill = cyan] (-4.5,9.5) rectangle (-3.5,8.5);
	%\draw [my-arc, fill = pink] (-0.5,8.5) rectangle (0.5,9.5);

\end{tikzpicture}}\end{minipage}
    \newline \vspace{20pt}

    \begin{minipage}{0.45\textwidth}\centering
    \begin{tabular}{||cc||cc||}\hline\hline
            \black & \black & \black & \black \\
            \black &  & \black & \black  \\ \hline \hline
            \black & \black & \black &  \\ 
            \black & \black &  &  \\\hline\hline
\end{tabular} 
    \end{minipage}\quad\begin{minipage}{0.45\textwidth}\centering \scalebox{.7}[.7]{\begin{tikzpicture}
	
	\usetikzlibrary{backgrounds}
	\usetikzlibrary{patterns}
	\usetikzlibrary{decorations,calc}
	
	\pgfdeclarelayer{bg}
	\pgfdeclarelayer{fg}
	\pgfsetlayers{bg,main,fg}

	\tikzstyle{my-vx}=[draw, circle, very thick, draw=black, fill=white]
	\tikzstyle{my-arc}=[draw=black, very thick]

	\begin{pgfonlayer}{fg}
	
		\node [my-vx] (v1) at (-3,1) {$1$};
		\node [my-vx] (v3) at (-3,-1) {$4$};
		
		\node [my-vx] (v5) at (1.5,1) {$5$};
		\node [my-vx] (v7) at (1.5,-1) {$7$};

		\draw [my-arc] (v1) edge[loop, looseness=4] (v1);
		\draw [my-arc] (v1) edge[bend right=0] (v3);
		
		\draw [my-arc] (v5) edge[loop, looseness=3.5, min distance=0mm, in=45, out=135] (v5);

		%\draw [my-arc] (v1) edge (v5);
		%\draw [my-arc] (v1) edge (v6);
		%\draw [my-arc] (v1) edge (v7);
		
		%\draw [my-arc] (v2) edge (v5);
		%\draw [my-arc] (v2) edge (v6);
		%\draw [my-arc] (v2) edge (v7);
		
		%\draw [my-arc] (v3) edge (v5);
		%\draw [my-arc] (v3) edge (v6);
		%\draw [my-arc] (v3) edge (v7);
		
		%\draw [my-arc] (v4) edge (v5);
		%\draw [my-arc] (v4) edge (v6);
		%\draw [my-arc] (v4) edge (v7);

	\end{pgfonlayer}
	
	\begin{pgfonlayer}{main}
		\draw [my-arc, fill=cyan] (-3,0) ellipse (0.9 and 2.2);
		\draw [my-arc, fill=pink] (1.5,0) ellipse (0.9 and 2.2);
		
		%\draw [my-arc, fill=pink]  node (v9) {} ellipse (2.5 and 1.5);
		%\draw [my-arc, fill=cyan]  node (v8) {} ellipse (3.0 and 2);
	\end{pgfonlayer}

\draw [my-arc] (v1) edge (v5);
\draw [my-arc] (v1) edge (v7);
\draw [my-arc] (v3) edge (v5);
\draw [my-arc] (v3) edge (v7);
	\begin{pgfonlayer}{bg}
		%\draw [my-arc, line width=100] (v8.center) edge (v9.center);
		\coordinate  {} {} {} {} {} {} {} {} {} {} {} {} {} {} {} {} {} {};
		\coordinate  {} {} {} {} {} {} {} {} {} {} {} {} {} {} {} {} {} {};
		\coordinate  {} {} {} {} {} {} {} {} {} {} {} {} {} {} {} {} {};
		\coordinate  {} {} {} {} {} {} {} {} {} {} {} {} {} {} {} {};
		
		%\draw [fill=black] (c1) -- (c2) -- (c3) -- (c4) -- cycle;

	\end{pgfonlayer}

	%\draw [my-arc, fill = cyan] (-4.5,9.5) rectangle (-3.5,8.5);
	%\draw [my-arc, fill = pink] (-0.5,8.5) rectangle (0.5,9.5);

\end{tikzpicture}}\end{minipage}
    \end{center}
    \caption{{The cases $1|24|7$ and $1|4|57$ correspond to the same family graphons but we consider the optimization problems separately, due to our prescribed ordering of the vertices.}}
    \label{1247 1457}
\end{figure}

Repeated applications of these three principles show that there are only $17$ distinct cases that we must consider. The details are straightforward to verify, see Lemma \ref{lem: 19 cases}.

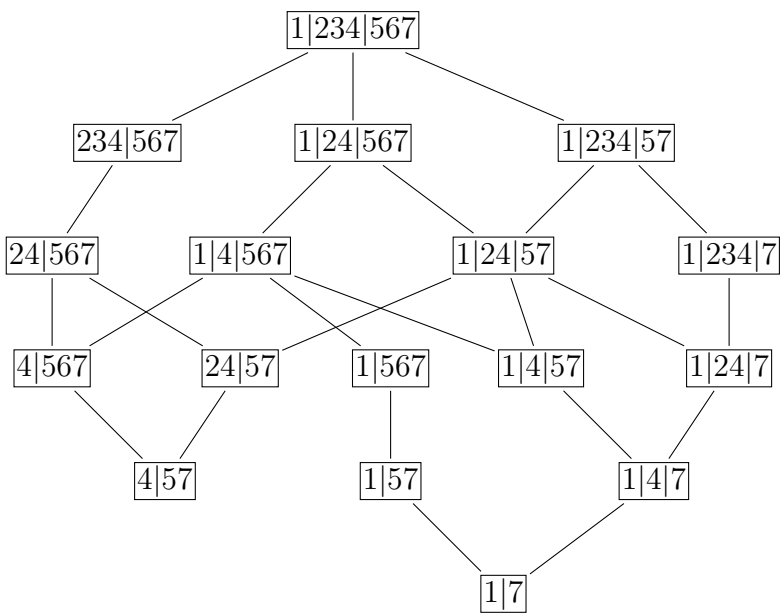
\begin{figure}[ht]
    \centering
    \scalebox{1.0}{
        \begin{tikzpicture}
	\tikzstyle{mybox} = [draw,outer sep=1.5,inner sep=1,minimum size=10]

% 7 VERTICES	
\node [mybox] (v1) at (-6,4) {$1|234|567$};

% 6 VERTICES
\node [mybox] (v2) at (-9,2.5) {$234|567$};
\node [mybox] (v3) at (-6,2.5) {$1|24|567$};
\node [mybox] (v4) at (-2.5,2.5) {$1|234|57$};

	\draw  (v1) edge (v2);
	\draw  (v1) edge (v3);
	\draw  (v1) edge (v4);

% 5 VERTICES
\node [mybox] (v5) at (-10,1) {$24|567$};
\node [mybox] (v6) at (-7.5,1) {$1|4|567$};
\node [mybox] (v7) at (-4,1) {$1|24|57$};
\node [mybox] (v8) at (-1,1) {$1|234|7$};
\node [mybox] (v9) at (-10,-0.5) {$4|567$};

	\draw  (v2) edge (v5);
	\draw  (v3) edge (v6);
	\draw  (v3) edge (v7);
	\draw  (v4) edge (v7);
	\draw  (v4) edge (v8);

% 4 VERTICES
\node [mybox] (v10) at (-7.5,-0.5) {$24|57$};
\node [mybox] (v11) at (-5.5,-0.5) {$1|567$};
\node [mybox] (v12) at (-3.5,-0.5) {$1|4|57$};
\node [mybox] (v13) at (-1,-0.5) {$1|24|7$};

	\draw  (v5) edge (v9);
	\draw  (v5) edge (v10);
	\draw  (v6) edge (v11);
	\draw  (v6) edge (v9);
	\draw  (v6) edge (v12);
	\draw  (v7) edge (v10);
	\draw  (v7) edge (v12);
	\draw  (v7) edge (v13);
	\draw  (v8) edge (v13);

% 3 VERTICES
\node [mybox] (v14) at (-8.5,-2) {$4|57$};
\node [mybox] (v15) at (-5.5,-2) {$1|57$};
\node [mybox] (v16) at (-2,-2) {$1|4|7$};

	\draw  (v9) edge (v14);
	\draw  (v10) edge (v14);
	\draw  (v11) edge (v15);
	\draw  (v12) edge (v16);
	\draw  (v13) edge (v16);

% 2 VERTICES
\node [mybox] (v17) at (-4,-3.5) {$1|7$};

	\draw  (v15) edge (v17);
	\draw  (v16) edge (v17);

\end{tikzpicture}
    }
    %\resizebox{.5}{\input{graphics/17-cases}}
    \caption{
        The set $\sset$, as a poset ordered by inclusion.  
        Each element is a subset of $V(G^*) = \{1,\dots,7\}$, written without braces and commas.  
        As noted in the proof of Lemma \ref{lem: 19 cases}, the sets $\{1\}$, $\{2,3,4\}$, and $\{5,6,7\}$ have different behavior in the problems $\SPRS$.  
        For this reason, we use vertical bars to separate each $S\in\sset$ according to the corresponding partition.  
    }
    \label{fig: 17 cases}
\end{figure}

%Following the same convention as Figure \ref{fig: 17 cases}, we suppress braces and commons when writing sets $S\subseteq V(G^*)$, and use vertical bars to partition $S$ according to its intersections with the sets $\{1\}$, $\{2,3,4\}$, and $\{5,6,7\}$.  
The distinct cases that we must consider are the following, summarized in Figure \ref{fig: 17 cases}.
\begin{align*}
    \sset 
    &:= 
    \left\{\begin{array}{r}
    1|234|567, 1|24|567, 1|234|57, 1|4|567, 1|24|57, 1|234|7, 234|567, 
    \\
    24|567, 4|567, 24|57, 1|567, 1|4|57, 1|2|47, 1|57, 4|57, 1|4|7, 1|7
    \end{array}
    \right\}
\end{align*}

}

\subsection{Interval arithmetic}\label{sub-sec: numerics}

Interval arithmetic is a computational technique which bounds errors that accumulate during computation.  
For convenience, let $\RR^* := [-\infty, +\infty]$ be the extended real line.  
To enhance order floating point arithmetic, we replace extended real numbers with unions of intervals which are guaranteed to contain them.  
Moreover, we extend the basic arithmetic operations $+, -, \times, \div$, and $\sqrt{}$ to operations on unions of intervals.  
This technique has real-world applications in the hard sciences, but has also been used in computer-assisted proofs.  
For two famous examples, we refer the interested reader to \cite{2002HalesKepler} for Hales' proof of the Kepler Conjecture on optimal sphere-packing in $\RR^2$, and to \cite{2002WarwickInterval} for Warwick's solution of Smale's $14$th problem on the Lorenz attractor as a strange attractor.  

As stated before, we consider extensions of the binary operations $+, -, \times,$ and $\div$ as well as the unary operation $\sqrt{}$ defined on $\RR$ to operations on unions of intervals of extended real numbers.  
For example if $[a,b], [c,d]\subseteq \RR$, then we may use the following extensions of $+, -, $ and $\times$: 
\begin{align*}
    [a,b] + [c,d] 
    &= 
        [a+c,b+d], \\
    [a,b] - [c,d] 
    &= 
        [a-d,b-c], \text{and}\\
    [a,b] \times [c,d] 
    &= 
    \left[ 
        \min\{ac, ad, b c, b d\}, 
        \max\{a c, a d, b c, b d\} 
    \right].  
\end{align*}
For $\div$, we must address the cases $0\in [c,d]$ and $0\notin [c,d]$.  
Here, we take the extension 
\begin{align*}
    [a,b] \div [c,d] 
    &= \left[ 
        \min\bigg\{\frac{a}{c}, \frac{a}{d}, \frac{b}{c}, \frac{b}{d}\bigg\}, \max\bigg\{\frac{a}{c}, \frac{a}{d}, \frac{b}{c}, \frac{b}{d}\bigg\} 
    \right] 
\end{align*}
where
\begin{align*}
    1 \div [c,d] 
    &= 
    \left\{\begin{array}{rl}
        \left[
            \min\{
                c^{-1}, d^{-1}
            \}, 
            \max\{
                c^{-1}, d^{-1}
            \}
        \right], 
        & 0\notin [c,d]
        \\
        \left[ 
            d^{-1},
            +\infty
        \right], 
        & \text{c=0}
        \\
        \left[ 
            -\infty,
            c^{-1}
        \right], 
        & 
        d = 0
        \\
        \left[ 
            -\infty,\frac{1}{c} 
        \right] 
        \cup \left[ 
            \frac{1}{d},+\infty
        \right], 
        & c < 0 < d
    \end{array}\right. 
    .
\end{align*}
Additionally, we may let 
\begin{align*}
    \sqrt{[a,b]}
    &= 
    \left\{\begin{array}{rl}
        \emptyset, & b < 0\\
        \left[
            \sqrt{\max\left\{
                0, a
            \right\}}
            , 
            \sqrt{b}
        \right], 
        & \text{othewise}
    \end{array}\right.
    .  
\end{align*}
When endpoints of $[a,b]$ and $[c,d]$ include $-\infty$ or  $+\infty$, the definitions above must be modified slightly in a natural way.  

We use interval arithmetic to prove the strict upper bound $<2/\sqrt{3}$ for the maximum graphon spread claimed in Theorem \ref{thm: spread maximum graphon}, for any solutions to $15$ of the $17$ constrained optimization problems $\SPRS$ stated in Lemma \ref{lem: 19 cases}.  
The constraints in each $\SPRS$ allow us to derive equations for the variables $(\alpha_i,f_i,g_i)_{i\in S}$ in terms of each other, and $\mu$ and $\nu$.  
For the reader's convenience, we relocate these formulas and their derivations to Appendix \ref{sub-sec: formulas}.  
In the programs corresponding to each set $S\in\sset$, we find we find two indices $i\in S\cap\{1,2,3,4\}$ and $j\in S\cap \{5,6,7\}$ such that for all $k\in S$, $\alpha_k,f_k,$ and $g_k$ may be calculated, step-by-step, from $\alpha_i, \alpha_j, \mu,$ and $\nu$.  
See Table \ref{tab: i j program search spaces} for each set $S\in\sset$, organized by the chosen values of $i$ and $j$.  
\begin{table}[ht]
    \centering
    \begin{tabular}{c||r|r|r|r}
     & $1$                      & $2$        & $3$                         & $4$      \\ \hline\hline
    \multirow{2}{*}{$5$}   & \multirow{2}{*}{$1|57$}  & $24|57$    & \multirow{2}{*}{$1|234|57$} & $4|57$   \\
                           &                          & $1|24|57$  &                             & $1|4|57$ \\ \hline
    \multirow{2}{*}{$6$}   & \multirow{2}{*}{$1|567$} & $24|567$   & $234|567$                   & $4|567$  \\
                           &                          & $1|24|567$ & $1|234|567$                 & $1|4|57$ \\ \hline
    $7$                    & $1|7$                    & $1|24|7$   & $1|234|7$                   & $1|4|7$ 
    \end{tabular}
    \caption{
        The indices $i,j$ corresponding to the search space used to bound solutions to $\SPRS$.  
    }
    \label{tab: i j program search spaces}
\end{table}

In the program corresponding to a set $S\in\sset$, we search a carefully chosen set $\Omega \subseteq [0,1]^3\times [-1,0]$ for values of $(\alpha_i,\alpha_j,\mu,\nu)$ which satisfy $\SPRS$.  
We first divide $\Omega$ into a grid of ``boxes''.  
Starting at depth $0$, we test each box $B$ for feasibility by assuming that $(\alpha_i,\alpha_j,\mu,\nu)\in B$ and that $\mu-\nu\geq 2/\sqrt{3}$.  
Next, we calculate $\alpha_k,f_k,$ and $g_k$ for all $k\in S$ in interval arithmetic using the formulas from Section \ref{sec: appendix}.  
When the calculation detects that a constraint of $\SPRS$ is not satisfied, e.g., by showing that some $\alpha_k, f_k,$ or $g_k$ lies in an empty interval, or by constraining $\sum_{i\in S}\alpha_i$ to a union of intervals which does not contain $1$, then the box is deemed infeasible.  
Otherwise, the box is split into two boxes of equal dimensions, with the dimension of the cut alternating cyclically.  

For each $S\in\sset$, the program $\SPRS$ has $3$ norm constraints, $2|S|$ linear eigenvector constraints, $|S|$ elliptical constraints, $\binom{|S|}{2}$ inequality constraints, and $3|S|$ interval membership constraints.  
By using interval arithmetic, we have a computer-assisted proof of the following result.  
\begin{lemma}\label{lem: 2 feasible sets}
    Suppose $S\in\sset\setminus\{\{1,7\}, \{4,5,7\}\}$.  
    Then any solution to $\SPR_S$ attains a value strictly less than $2/\sqrt{3}$.  
\end{lemma}
To better understand the role of interval arithmetic in our proof, consider the following example.  
\begin{example}\label{ex: infeasible box}
    Suppose $\mu,\nu$, and $(\alpha_i,f_i,g_i)$ is a solution to $\SPR_{\{1,\dots,7\}}$.  
    We show that $(\alpha_3, \mu,\nu)\notin [.7,.8]\times [.9,1]\times [-.2,-.1]$.  
    By Proposition \ref{prop: fg23 assume 23}, 
    $\displaystyle{g_3^2 = \frac{\nu(\alpha_3 + 2 \mu)}{\alpha_3(\mu + \nu) + 2 \mu \nu}}$.  
    Using interval arithmetic, 
    \begin{align*}
        \nu(\alpha_3 + 2 \mu) 
        &= 
            [-.2,-.1] \times \big([.7,.8] + 2 \times [.9,1] \big) 
        \\ 
        &= [-.2,-.1] \times [2.5,2.8] = [-.56,-.25], \text{ and }
        \\
        \alpha_3(\mu + \nu) + 2 \mu \nu 
        &= 
            [.7,.8]\times 
            \big([.9,1] 
            + [-.2,-.1]\big) 
            + 2 \times [.9,1] 
            \times [-.2,-.1] 
        \\
        &= 
            [.7,.8] 
            \times [.7,.9] 
            + [1.8,2] 
            \times  [-.2,-.1] 
        \\
        &= 
            [.49,.72] 
            + [-.4,-.18] 
        = [.09,.54].  
    \end{align*}
Thus 
\begin{align*}
    g_3^2 
    &= 
    \frac{
        \nu
        (\alpha_3 + 2 \mu)
    }
    {
        \alpha_3(\mu + \nu) 
        + 2 \mu \nu
    } 
    = 
        [-.56,-.25] 
        \div [.09,.54] 
    = 
        [-6.\overline{2},-.4\overline{629}].
\end{align*}
Since $g_3^2\geq 0$, we have a contradiction.  
\end{example}

Example \ref{ex: infeasible box} illustrates a number of key elements. 
First, we note that through interval arithmetic, we are able to provably rule out the corresponding region. 
However, the resulting interval for the quantity $g_3^2$ is over fifty times bigger than any of the input intervals. 
This growth in the size of intervals is common, and so, in some regions, fairly small intervals for variables are needed to provably illustrate the absence of a solution. 
For this reason, using a computer to complete this procedure is ideal, as doing millions of calculations by hand would be untenable. 

However, the use of a computer for interval arithmetic brings with it another issue.  
Computers have limited memory, and therefore cannot represent all numbers in $\mathbb{R}^*$. 
Instead, a computer can only store a finite subset of numbers, which we will denote by $F\subsetneq \mathbb{R}^*$. 
This set $F$ is not closed under the basic arithmetic operations, and so when some operation is performed and the resulting answer is not in $F$, some rounding procedure must be performed to choose an element of $F$ to approximate the exact answer. 
This issue is the cause of roundoff error in floating point arithmetic, and must be treated in order to use computer-based interval arithmetic as a proof.

PyInterval is one of many software packages designed to perform interval arithmetic in a manner which accounts for this crucial feature of floating point arithmetic.  
Given some $x \in \mathbb{R}^*$, let $fl_-(x)$ be the largest $y \in F$ satisfying $y \le x$, and $fl_+(x)$ be the smallest $y \in F$ satisfying $y \ge x$. 
Then, in order to maintain a mathematically accurate system of interval arithmetic on a computer, once an operation is performed to form a union of intervals  $\bigcup_{i=1}^k[a_i, b_i]$, the computer forms a union of intervals containing $[fl_-(a_i),fl_+(b_i)]$ for all $1\leq i\leq k$.  
The programs which prove Lemma \ref{lem: 2 feasible sets} can be found at \cite{2021riasanovsky-spread}.  

\subsection{Completing the proof of Theorem \ref{thm: spread maximum graphon}}\label{sub-sec: cases 4|57 and 1|7}
Finally, we complete the second main result of this paper. We will need the following lemma.

\begin{lemma} \label{lem: SPR457}
 If $(\alpha_4,\alpha_5,\alpha_7)$ is a solution to $\SPR_{\{4,5,7\}}$, then $\alpha_7=0.$
\end{lemma}

We delay the proof of Lemma \ref{lem: SPR457} to Section \ref{sec: ugly} because it is technical. We now proceed with the Proof of Theorem \ref{thm: spread maximum graphon}.

\begin{proof}[Proof of Theorem \ref{thm: spread maximum graphon}]
Let $W$ be a graphon such that $\spr(W) = \max_{U\in\ww}\spr(U)$.  
By Lemma \ref{lem: 2 feasible sets} and Lemma \ref{lem: SPR457}, $W$ is a $2\times 2$ stepgraphon. Let the weights of the parts be $\alpha_1$ and $1-\alpha_1$.

\begin{comment}
By Lemma \ref{lem: SPR457}, 

Finally, we complete the proof of the desired claim.  
Note that the 
\begin{align*}
    \left[\begin{array}{ccc}
        \alpha_5 & 0 & 0\\
        0 & 0 & 0\\
        \alpha_5 & 1-\alpha_5 & 0
    \end{array}\right]
    \quad \text{ and }\quad 
    \left[\begin{array}{ccc}
        \alpha_5 & 0\\
        \alpha_5 & 1-\alpha_5
    \end{array}\right]
\end{align*}
have the same multiset of nonzero eigenvalues.  
By Claim F and the definition of $\SPR_{\{4,5,7\}}$ and $\SPR_{\{4,5\}}$, we have the following.  
If$\mu, \nu, $ and $(\alpha_4,\alpha_5,\alpha_7)$ are part of a solution to $\SPR_{\{4,5,7\}}$, then $\alpha_7 = 0$ and the quantities $\mu, \nu, $ and $(\alpha_4,1-\alpha_4)$ are part of a solution to $\SPR_{\{4,5\}}$.  
By setting $(\alpha_1',\alpha_7') := (\alpha_5,1-\alpha_5)$, it follows by inspection of the definition of $\SPR_{\{4,5\}}$ and $\SPR_{\{1,7\}}$ that $\mu, \nu$, and $(\alpha_1',\alpha_7')$ are part of a solution to $\SPR_{\{1,7\}}$.  

\end{comment}

Thus, it suffices to demonstrate the uniqueness of the desired solution $\mu, \nu, $ and $(\alpha_i,f_i,g_i)_{i\in \{1,7\}}$ to $\SPR_{\{1,7\}}$.  
Indeed, we first note that with 
\begin{align*}
    N(\alpha_1)
    := \left[\begin{array}{ccc}
        \alpha_1 & 1-\alpha_1\\
        \alpha_1 & 0 
    \end{array}\right], 
\end{align*}
the quantities $\mu$ and $\nu$ are precisely the eigenvalues of the characteristic polynomial 
\begin{align*}
    p(x) 
    = 
        x^2-\alpha_1x-\alpha_1(1-\alpha_1).  
\end{align*}
In particular, 
\begin{align*}
    \mu 
    &=
        \dfrac{
            \alpha_1 + \sqrt{\alpha_1(4-3\alpha_1)}
        }
        {2}
    , \quad 
    \nu 
    =
        \dfrac{
            \alpha_1 - \sqrt{\alpha_1(4-3\alpha_1)}
        }
        {2}
    , 
\end{align*}
and 
\begin{align*}
    \mu-\nu 
    &=
        \sqrt{\alpha_1(4-3\alpha_1)}.  
\end{align*}
Optimizing, it follows that $(\alpha_1, 1-\alpha_1) = (2/3, 1/3)$. Calculating the eigenfunctions and normalizing them gives that $\mu, \nu, $ and their respective eigenfunctions match those from the statement of Theorem \ref{thm: spread maximum graphon}.

%By considering the equations corresponding to $\mu$ and $\nu$ as eigenvalues of $N(2/3, 1/3)$ and the norm constrains on the functions $f$ and $g$, this completes the proof.  
\end{proof}
\section{From graphons to graphs}\label{sub-sec: graphons to graphs}

In this section, we show that Theorem \ref{thm: spread maximum graphon} implies Conjecture \ref{thm: spread maximum graphs} for all $n$ sufficiently large; that is, the solution to the problem of maximizing the spread of a graphon implies the solution to the problem of maximizing the spread of a graph for sufficiently large $n$.  

The outline for our argument is as follows.  
First, we define the spread-maximum graphon $W$ as in Theorem \ref{thm: spread maximum graphon}.  
Let $\{G_n\}$ be any sequence where each $G_n$ is a spread-maximum graph on $n$ vertices and denote by $\{W_n\}$ the corresponding sequence of graphons.  
We show that, after applying measure-preserving transformations to each $W_n$, the extreme eigenvalues and eigenvectors of each $W_n$ converge suitably to those of $W$.  
It follows for $n$ sufficiently large that except for $o(n)$ vertices, $G_n$ is a join of a clique of $2n/3$ vertices and an independent set of $n/3$ vertices (Lemma \ref{lem: few exceptional vertices}).  
Using results from Section \ref{sec:graphs}, we precisely estimate the extreme eigenvector entries on this $o(n)$ set.  
Finally, Lemma \ref{lem: no exceptional vertices} shows that the set of $o(n)$ exceptional vertices is actually empty, completing the proof.  
\\

Before proceeding with the proof, we state the following corollary of the Davis-Kahan theorem \cite{DK}, stated for graphons.

\begin{corollary}\label{cor: DK eigenfunction perturbation}
    Suppose $W,W':[0,1]^2\to [0,1]$ are %bounded symmetric kernels.  
    graphons
    Let $\mu$ be an eigenvalue of $W$ with $f$ a corresponding unit eigenfunction.  
Let $\{h_k\}$ be an orthonormal eigenbasis for $W'$ with corresponding eigenvalues $\{\mu_k'\}$.  
    Suppose that $|\mu_k'-\mu| > \delta$ for all $k\neq 1$.  
    Then 
   
    \begin{align*}
        \sqrt{1 - \langle h_1,f\rangle^2}
        \leq 
            \dfrac{\|A_{W'-W}f\|_2}{\delta}.  
    \end{align*}

\end{corollary}

%\begin{proof}[Proof of Theorem \ref{thm: spread maximum graphs}.]

Before proving Theorem \ref{thm: spread maximum graphs}, we prove the following approximate result.  
For all nonnegative integers $n_1,n_2,n_3$, let $G(n_1,n_2,n_3) := (K_{n_1}\dot{\cup} K_{n_2}^c)\vee K_{n_3}^c$.  

\begin{lemma}\label{lem: few exceptional vertices}
    For all positive integers integers $n$, let $G_n$ denote a graph on $n$ vertices which maximizes spread.  
    Then $G_n = G(n_1,n_2,n_3)$ for some nonnegative integers $n_1,n_2,n_3$ such that $n_1 = (2/3+o(1))n$, $n_2 = o(n)$, and $n_3 = (1/3+o(1))n$.  
\end{lemma}
\begin{proof}
    Our argument outline is: 
    \begin{enumerate}
        \item 
            show that the eigenvectors for the spread-extremal graphs resemble the eigenfunctions of the spread-extremal graphon in an $L_2$ sense
        \item 
            show that with the exception of a small proportion of vertices, a spread-extremal graph is the join of a clique and an independent set
    \end{enumerate}
    
    Let $\pp := [0,2/3]$ and $\nn := [0,1]\setminus \pp$.  
    By Theorem \ref{thm: spread maximum graphon}, the graphon $W$ which is the indicator function of the set $[0,1]^2\setminus \nn^2$ maximizes spread.  
    Denote by $\mu$ and $\nu$ its maximum and minimum eigenvalues, respectively.  
    For every positive integer $n$, let $G_n$ denote a graph on $n$ vertices which maximizes spread, let $W_n$ be any stepgraphon corresponding to $G_n$, and let $\mu_n$ and $\nu_n$ denote the maximum and minimum eigenvalues of $W_n$, respectively.  
    By Theorems \ref{thm: graphon eigenvalue continuity} and \ref{thm: spread maximum graphon}, and compactness of $\hatW$,
    \begin{align*}
        \max\left\{
            |\mu-\mu_n|, 
            |\nu-\nu_n|, 
            \delta_\square(W, W_n)
        \right\}\to 0.  
    \end{align*}
    Moreover, we may apply measure-preserving transformations to each $W_n$ so that without loss of generality,  $\|W-W_n\|_\square\to 0$.  
    As in Theorem \ref{thm: spread maximum graphon}, let $f$ and $g$ be unit eigenfunctions which take values $f_1,f_2, g_1, g_2$.  
    Furthermore, let $\varphi_n$ be a nonnegative unit $\mu_n$-eigenfunction for $W_n$ and let $\psi_n$ be a $\nu_n$-eigenfunction for $W_n$.  
    \\
    
    We show that without loss of generality, $\varphi_n\to f$ and $\psi_n\to g$ in the $L_2$ sense.  
    Since $\mu$ is the only positive eigenvalue of $W$ and it has multiplicity $1$, taking $\delta := \mu/2$, Corollary \ref{cor: DK eigenfunction perturbation} implies that 
    \begin{align*}
    	1-\langle f,\varphi_n\rangle^2
    	&\leq
    		\dfrac{4\|A_{W-W_n}f\|_2^2}
    		{\mu^2}
    	\\
    	&=
    	    \dfrac{4}{\mu^2}
    	    \cdot
    	    \left\langle
    	        A_{W-W_n}f, 
    	        A_{W-W_n}f
    	    \right\rangle
	    \\
	    &\leq 
	        \dfrac{4}{\mu^2}
	        \cdot
	        \|
	            A_{W-W_n}f
            \|_1
            \cdot 
            \|
	            A_{W-W_n}f
            \|_\infty
        \\
        &\leq 
            \dfrac{4}{\mu^2}
	        \cdot
	        \left(
	            \|A_{W-W_n}\|_{\infty\to1}\|f\|_\infty
	        \right)
	        \cdot
	        \|f\|_\infty
        \\
        &\leq
            \dfrac{
                16\|W-W_n\|_\square
                \cdot \|f\|_\infty^2}
            {\mu^2},
    \end{align*}
    where the last inequality follows from Lemma 8.11 of \cite{Lovasz2012Hombook}.
    Since $\|f\|_\infty\leq 1/\mu$, this proves the first claim.  
    The second claim follows by replacing $f$ with $g$, and $\mu$ with $|\nu|$.  
    \\
    \\
    \textbf{Note: }For the remainder of the proof, we will introduce quantities $\varepsilon_i > 0$ 
    in lieu of writing complicated expressions explicitly.  
    When we introduce a new $\varepsilon_i$, we will remark that given $\varepsilon_0,\dots,\varepsilon_{i-1}$ sufficiently small, $\varepsilon_i$ can be made sufficiently small enough to meet some other conditions.  
    \\
    
    Let $\varepsilon_0 > 0$ and for all $n\geq 1$, define 
    \begin{align*}
        \pp_n 
        &:= 
            \{
                x\in [0,1] : 
                |\varphi_n(x) - f_1| < \varepsilon_0
                \text{ and }
                |\psi_n(x)-g_1| < \varepsilon_0
            \}, 
        \\
        \nn_n 
        &:= 
            \{
                x\in [0,1] : 
                |\varphi_n(x) - f_2| < \varepsilon_0
                \text{ and }
                |\psi_n(x)-g_2| < \varepsilon_0
            \}, \text{ and }
        \\
        \ee_n 
        &:= 
            [0,1]\setminus (\pp_n\cup \nn_n).  
    \end{align*}
    Since 
    \begin{align*}
        \int_{|\varphi_n-f|\geq \varepsilon_0}
            |\varphi_n-f|^2
        &\leq 
        \int_{}
            |\varphi_n-f|^2
        \to 0, 
        \text{ and }\\
        \int_{|\psi_n-g|\geq \varepsilon_0}
            |\psi_n-g|^2
        &\leq 
        \int_{}
            |\psi_n-g|^2
        \to 0, 
    \end{align*}
    it follows that 
    \begin{align*}
        \max\left\{
            m(\pp_n\setminus\pp), 
            m(\nn_n\setminus\nn), 
            m(\ee_n)
        \right\}
        \to 0.  
    \end{align*}
    
    For all $u\in V(G_n)$, let $S_u$ be the subinterval of $[0,1]$ corresponding to $u$ in $W_n$, and denote by $\varphi_u$ and $\psi_u$ the constant values of $\varphi_n$ on $S_u$.  
    For convenience, we define the following discrete analogues of $\pp_n, \nn_n, \ee_n$: 
    \begin{align*}
        P_n 
        &:= 
            \{
                u\in V(G_n) : 
                |\varphi_u - f_1| < \varepsilon_0
                \text{ and }
                |\psi_u-g_1| < \varepsilon_0
            \}, 
        \\
        N_n 
        &:= 
            \{
                u\in V(G_n) : 
                |\varphi_u - f_2| < \varepsilon_0
                \text{ and }
                |\psi_u-g_2| < \varepsilon_0
            \}, \text{ and }
        \\
        E_n 
        &:= 
            V(G_n) \setminus (P_n\cup N_n).  
    \end{align*}
    
    Let $\varepsilon_1>0$.  
    By Lemma \ref{discrete ellipse equation} and using the fact that $\mu_n\to \mu$ and $\nu_n\to \nu$, 
    \begin{align}\label{eq: recall graph ellipse equation}
        \left|
            \mu \varphi_u^2 - \nu\psi_u^2 - (\mu-\nu)
        \right|
        &<
            \varepsilon_1
        \quad \text{ for all }u\in V(G_n) 
    \end{align}
    for all $n$ sufficiently large.  
    Let $\varepsilon_0'>0$.  
    We next need the following claim, which says that the eigenvector entries of the exceptional vertices behave as if they have neighborhood $N_n$.
    \\
    \\
    \textbf{Claim I.  }
    Suppose $\varepsilon_0$ is sufficiently small and $n$ is sufficiently large in terms of $\varepsilon_0'$.   
    Then for all $v\in E_n$, 
    \begin{align}\label{eq: exceptional vertex entries}
        \max\left\{
            \left|
                \varphi_v - \dfrac{f_2}{3\mu}
            \right|, 
            \left|
                \psi_v - \dfrac{g_2}{3\nu}
            \right|
        \right\}
        < \varepsilon_0'.  
    \end{align}
    Indeed, suppose $v \in E_n$ and let 
    \begin{align*}
        U_n 
        := 
        \{w\in V(G_n) : vw\in E(G_n)\}
        \quad\text{ and }\quad 
        \uu_n 
        := 
        \bigcup_{w\in U_n}
            S_w.  
    \end{align*}
    We take two cases, depending on the sign of $\psi_v$.  
    \\
    \\
    \textbf{Case A: $\psi_v \geq 0$.  }
    
    Recall that $f_2 > 0 > g_2$.  
    Furthermore, $\varphi_v \geq 0$ and by assumption, $\psi_v\geq 0$.  
    It follows that for all $n$ sufficiently large, $f_2\varphi_v - g_2\psi_v > 0$, so by Lemma \ref{lem: graph join}, $N_n\subseteq U_n$.  
    Since $\varphi_n$ is a $\mu_n$-eigenfunction for $W_n$, 
    \begin{align*}
        \mu_n \varphi_v
        &=
            \int_{y\in [0,1]}W_n(x,y)\varphi_n(y)
            \\
        &= 
            \int_{y\in \pp_n\cap \uu_n}\varphi_n(y)
            + \int_{y\in \nn_n}\varphi_n(y)
            + \int_{y\in \ee_n\cap \uu_n}\varphi_n(y).  
    \end{align*}
    Similarly, 
    \begin{align*}
        \nu_n \psi_v
        &=
            \int_{y\in [0,1]}K_n(x,y)\psi_n(y)
            \\
        &= 
            \int_{y\in \pp_n\cap \uu_n}\psi_n(y)
            + \int_{y\in \nn_n}\psi_n(y)
            + \int_{y\in \ee_n\cap \uu_n}\psi_n(y).  
    \end{align*}
    Let $\rho_n := m(\pp_n\cap \uu_n)$.  
    Note that for all $\varepsilon_2 > 0$, as long as $n$ is sufficiently large and $\varepsilon_1$ is sufficiently small, then 
    \begin{align}\label{eq: eigenvector entries with rho (case A)}
        \max\left\{
        \left|
            \varphi_v-\dfrac{3\rho_n f_1 + f_2}{3\mu}
        \right|
        , 
        \left|
            \psi_v-\dfrac{3\rho_n g_1 + g_2}{3\nu}
        \right|
        \right\}
        < 
            \varepsilon_2.  
    \end{align}
    Let $\varepsilon_3 > 0$.  
    By Equations \eqref{eq: recall graph ellipse equation} and \eqref{eq: eigenvector entries with rho (case A)} and with $\varepsilon_1,\varepsilon_2$ sufficiently small, 
    \begin{align*}
        \left|
            \mu\cdot \left(
                \dfrac{3\rho_n f_1 + f_2}{3\mu}
            \right)^2
            -\nu\cdot \left(
                \dfrac{3\rho_n g_1 + g_2}{3\nu}
            \right)^2
            - (\mu-\nu)
        \right|
        < \varepsilon_3.  
    \end{align*}
    Substituting the values of $f_1,f_2, g_1,g_2$ from Theorem \ref{thm: spread maximum graphon} and simplifying, it follows that 
    \begin{align*}
        \left|
            \dfrac{\sqrt{3}}{2}
            \cdot 
            \rho_n(3\rho_n-2)
        \right|
        < \varepsilon_3
    \end{align*}
    Let $\varepsilon_4 > 0$.  
    It follows that if $n$ is sufficiently large and $\varepsilon_3$ is sufficiently small, then 
    \begin{align}\label{eq: (case A) rho estimates}
        \min\left\{
            \rho_n, 
            |2/3-\rho_n|
        \right\}
        < \varepsilon_4.  
    \end{align}
    Combining Equations \eqref{eq: eigenvector entries with rho (case A)} and \eqref{eq: (case A) rho estimates}, it follows that with $\varepsilon_2,\varepsilon_4$ sufficiently small, then 
    \begin{align*}
        \max\left\{
            \left|
                \varphi_v - \dfrac{f_2}{3\mu}
            \right|, 
            \left|
                \psi_v - \dfrac{g_2}{3\mu}
            \right|
        \right\}
        &< \varepsilon_0', \text{ or }
        \\
        \max\left\{
            \left|
                \varphi_v - \dfrac{2f_1 + f_2}{3\mu}
            \right|, 
            \left|
                \psi_v - \dfrac{2g_1 + g_2}{3\mu}
            \right|
        \right\}
        &< \varepsilon_0'.  
    \end{align*}
    Note that 
    \begin{align*}
        f_1 &= \dfrac{2f_1 + f_2}{3\mu}
        \quad \text{ and }\quad 
        g_1 = \dfrac{2g_1 + g_2}{3\nu}.  
    \end{align*}
    Since $v\in E_n$, the second inequality does not hold, which completes the proof of the desired claim.  
    \\
    \\
    \textbf{Case B: $\psi_v < 0$.  }
    
    Recall that $f_1 > g_1 > 0$.  
    Furthermore, $\varphi_v \geq 0$ and by assumption, $\psi_v < 0$.  
    It follows that for all $n$ sufficiently large, $f_1\varphi_v - g_1\psi_v > 0$, so by Lemma \ref{lem: graph join}, $P_n\subseteq U_n$.  
    Since $\varphi_n$ is a $\mu_n$-eigenfunction for $W_n$, 
    \begin{align*}
        \mu_n \varphi_v
        &=
            \int_{y\in [0,1]}W_n(x,y)\varphi_n(y)
            \\
        &= 
            \int_{y\in \nn_n\cap \uu_n}\varphi_n(y)
            + \int_{y\in \pp_n}\varphi_n(y)
            + \int_{y\in \ee_n\cap \uu_n}\varphi_n(y).  
    \end{align*}
    Similarly, 
    \begin{align*}
        \nu_n \psi_v
        &=
            \int_{y\in [0,1]}W_n(x,y)\psi_n(y)
            \\
        &= 
            \int_{y\in \nn_n\cap \uu_n}\psi_n(y)
            + \int_{y\in \pp_n}\psi_n(y)
            + \int_{y\in \ee_n\cap \uu_n}\psi_n(y).  
    \end{align*}
    Let $\rho_n := m(\nn_n\cap \uu_n)$.  
    Note that for all $\varepsilon_5 > 0$, as long as $n$ is sufficiently large and $\varepsilon_1$ is sufficiently small, then

    \begin{align}\label{eq: eigenvector entries with rho (case B)}
        \max\left\{
        \left|
            \varphi_v-\dfrac{2f_1 + 3\rho_n f_2}{3\mu}
        \right|
        , 
        \left|
            \psi_v-\dfrac{2g_1 + 3\rho_n g_2}{3\nu}
        \right|
        \right\}
        < 
            \varepsilon_5.  
    \end{align}
    Let $\varepsilon_6 > 0$.  
    By Equations \eqref{eq: recall graph ellipse equation} and \eqref{eq: eigenvector entries with rho (case B)} and with $\varepsilon_1,\varepsilon_2$ sufficiently small, 
    \begin{align*}
        \left|
            \mu\cdot \left(
                \dfrac{2f_1 + 3\rho_n f_2}{3\mu}
            \right)^2
            -\nu\cdot \left(
                \dfrac{2g_1 + 3\rho_n g_2}{3\nu}
            \right)^2
            - (\mu-\nu)
        \right|
        < \varepsilon_6.  
    \end{align*}
    Substituting the values of $f_1,f_2, g_1,g_2$ from Theorem \ref{thm: spread maximum graphon} and simplifying, it follows that 
    \begin{align*}
        \left|
            2\sqrt{3}
            \cdot 
            \rho_n(3\rho_n-1)
        \right|
        < \varepsilon_6
    \end{align*}
    Let $\varepsilon_7 > 0$.  
    It follows that if $n$ is sufficiently large and $\varepsilon_6$ is sufficiently small, then 
    \begin{align}\label{eq: (case B) rho estimates}
        \min\left\{
            \rho_n, 
            |1/3-\rho_n|
        \right\}
        < \varepsilon_7.  
    \end{align}
    Combining Equations \eqref{eq: eigenvector entries with rho (case A)} and \eqref{eq: (case B) rho estimates}, it follows that with $\varepsilon_2,\varepsilon_4$ sufficiently small, then 
    \begin{align*}
        \max\left\{
            \left|
                \varphi_v - \dfrac{2f_1}{3\mu}
            \right|, 
            \left|
                \psi_v - \dfrac{2g_1}{3\mu}
            \right|
        \right\}
        &< \varepsilon_0', \text{ or }
        \\
        \max\left\{
            \left|
                \varphi_v - \dfrac{2f_1 + f_2}{3\mu}
            \right|, 
            \left|
                \psi_v - \dfrac{2g_1 + g_2}{3\mu}
            \right|
        \right\}
        &< \varepsilon_0'.  
    \end{align*}
    Again, note that 
    \begin{align*}
        f_1 &= \dfrac{2f_1 + f_2}{3\mu}
        \quad \text{ and }\quad 
        g_1 = \dfrac{2g_1 + g_2}{3\nu}.  
    \end{align*}
    Since $v\in E_n$, the second inequality does not hold.  
    \\
    
    Similarly, note that 
    \begin{align*}
        f_2 &= \dfrac{2f_1}{3\mu}
        \quad \text{ and }\quad 
        g_2 = \dfrac{2g_1}{3\nu}.  
    \end{align*}
    Since $v\in E_n$, the first inequality does not hold, a contradiction.  
    So the desired claim holds.  
    \\

    We now complete the proof of Lemma \ref{lem: no exceptional vertices} by showing that for all $n$ sufficiently large, $G_n$ is the join of an independent set $N_n$ with a disjoint union of a clique $P_n$ and an independent set $E_n$. 
    
    As above, we let $\varepsilon_0,\varepsilon_0'>0$ be arbitrary.  
    By definition of $P_n$ and $N_n$ and by Equation \eqref{eq: exceptional vertex entries} from Claim I, 
    then for all $n$ sufficiently large, 
    \begin{align*}
        \max\left\{
            \left|
                \varphi_v - f_1
            \right|, 
            \left|
                \psi_v - g_1
            \right|
        \right\}
        &< \varepsilon_0
            &\text{ for all }v\in P_n
        \\
        \max\left\{
            \left|
                \varphi_v - \dfrac{f_2}{3\mu}
            \right|, 
            \left|
                \psi_v - \dfrac{g_2}{3\nu}
            \right|
        \right\}
        &< \varepsilon_0'
            &\text{ for all }v\in E_n
        \\
        \max\left\{
            \left|
                \varphi_v - f_2
            \right|, 
            \left|
                \psi_v - g_2
            \right|
        \right\}
        &< \varepsilon_0
            &\text{ for all }v\in N_n
    \end{align*}
    With rows and columns respectively corresponding to the vertex sets $P_n, E_n,$ and $N_n$, we note the following inequalities: 
    Indeed, note the following inequalities: 
    \begin{align*}
        \begin{array}{c|c||c}
            f_1^2 
            > 
            g_1^2
            &
            f_1\cdot\dfrac{f_2}{3\mu} 
            < 
            g_1\cdot\dfrac{g_2}{3\nu}
            &
            f_1f_2 
            > 
            g_1g_2\\\hline
            &
            \left(\dfrac{f_2}{3\mu}\right)^2 
            < 
            \left(\dfrac{g_2}{3\nu}\right)^2
            &
            \dfrac{f_2}{3\mu}\cdot f_2 
            > \dfrac{g_2}{3\nu}\\\hline\hline
            &&
            f_2^2 < g_2^2
        \end{array}
        \quad .
    \end{align*}
    Let $\varepsilon_0, \varepsilon_0'$ be sufficiently small.  
    Then for all $n$ sufficiently large and for all $u,v\in V(G_n)$, then $\varphi_u\varphi_v-\psi_u\psi_v < 0$ if and only if $u,v\in E_n$, $u,v\in N_n$, or $(u,v)\in (P_n\times E_n)\cup (E_n\times P_n)$.  
    By Lemma \ref{lem: graph join}, since $m(P_n) \to 2/3$ and $m(N_n) \to 1/3$, the proof is complete.  
\end{proof}

We have now shown that the spread-extremal graph is of the form $(K_{n_1}\dot{\cup} K_{n_2}^c)\vee K_{n_3}^c$ where $n_2 = o(n)$. The next lemma refines this to show that actually $n_2 = 0$.

\begin{lemma}\label{lem: no exceptional vertices}
    For all nonnegative integers $n_1,n_2,n_3$, let $G(n_1, n_2, n_3) := (K_{n_1}\cup K_{n_2}^c)\vee K_{n_3}^c$.  
    Then for all $n$ sufficiently large, the following holds.  
    If $\spr(G(n_1,n_2,n_3))$ is maximized subject to the constraint $n_1+n_2+n_3 = n$ and $n_2 = o(n)$, then $n_2 = 0$.  
\end{lemma}

\emph{Proof outline:} We aim to maximize the spread of $G(n_1, n_2, n_3)$ subject to $n_2 = o(n)$. The spread of $G(n_1, n_2, n_3)$ is the same as the spread of the quotient matrix
\[
Q_n = \begin{bmatrix}
n_1 - 1 & 0 & n_3\\
0 & 0 & n_3 \\
n_1 & n_2 & 0
\end{bmatrix}.
\]

We reparametrize with parameters $\varepsilon_1$ and $\varepsilon_2$ representing how far away $n_1$ and $n_3$ are proportionally from $\frac{2n}{3}$ and $\frac{n}{3}$, respectively. Namely, $\varepsilon_1 = \frac{2}{3} - \frac{n_1}{n}$ and $\varepsilon_2 = \frac{1}{3} - \frac{n_3}{n}$. Then $\varepsilon_1 + \varepsilon_2 = \frac{n_2}{n}$. Hence maximizing the spread of $G(n_1,n_2,n_3)$ subject to $n_2 = o(n)$ is equivalent to maximizing the spread of the matrix
\[
n\begin{bmatrix}
\frac{2}{3} - \varepsilon_1 - \frac{1}{n} & 0 & \frac{1}{3} - \varepsilon_2 \\
0 & 0 & \frac{1}{3} - \varepsilon_2 \\
\frac{2}{3} - \varepsilon_1 & \varepsilon_1 + \varepsilon_2 & 0
\end{bmatrix}
\]
subject to the constraint that $\textstyle \frac{2}{3}-\varepsilon_1$ and $\textstyle \frac{1}{3}-\varepsilon_2$ are nonnegative integer multiples of $\frac{1}{n}$ and $\varepsilon_1+\varepsilon_2 = o(1)$. In order to utilize calculus, we instead solve a continuous relaxation of the optimization problem. 

%First, we make a general definition.  
%For all $\varepsilon_1,\varepsilon_2,z\in \RR$, let 
As such, consider the following matrix. 
\begin{align*}
    M_z(\varepsilon_1,\varepsilon_2)
    := \left[\begin{array}{ccc}
        \dfrac{2}{3}-\varepsilon_1-z & 0 & \dfrac{1}{3}-\varepsilon_2\\
        0 & 0 & \dfrac{1}{3}-\varepsilon_2\\
        \dfrac{2}{3}-\varepsilon_1 & \varepsilon_1+\varepsilon_2 & 0
    \end{array}\right]
    .  
\end{align*}

Since $M_z(\varepsilon_1,\varepsilon_2)$ is diagonalizable, we may let $S_z(\varepsilon_1,\varepsilon_2)$ be the difference between the maximum and minimum eigenvalues of $M_z(\varepsilon_1,\varepsilon_2)$.  
We consider the optimization problem $\pp_{z,C}$ defined for all $z\in\RR$ and all $C>0$ such that $|z|$ and $C$ are sufficiently small, by

\begin{align*}
    (\pp_{z,C}): 
    \left\{\begin{array}{rl}
        \max
        & 
            S_z(\varepsilon_1,\varepsilon_2)\\
        \text{s.t}. 
        & 
            \varepsilon_1,\varepsilon_2\in [-C,C].
    \end{array}\right.
\end{align*}

{
We show that as long as $C$ and $|z|$ are sufficiently small, then the optimum of $\pp_{z,C}$ is attained by \begin{align*}
(\varepsilon_1, \varepsilon_2) 
&= 
    \left(
        (1+o(z))\cdot \dfrac{7z}{30}, 
        (1+o(z))\cdot \dfrac{-z}{3}
    \right).  
\end{align*}
Moreover we show that in the feasible region of $\pp_{z,C}$, $S_{z,C}(\varepsilon_1,\varepsilon_2)$ is concave-down in $(\varepsilon_1,\varepsilon_2)$.  
We return to the original problem by imposing the constraint that $ \frac{2}{3}-\varepsilon_1$ and $\frac{1}{3}-\varepsilon_2$ are multiples of $\frac{1}{n}$.  
Together these two observations complete the proof of the lemma.  
Under these added constraints, the optimum is obtained when  
\begin{align*}
    (\varepsilon_1, \varepsilon_2)
    &=
    \left\{\begin{array}{rl}
        (0,0), & n\equiv 0 \pmod{3}\\
        (2/3,-2/3), & n\equiv 1 \pmod{3}\\
        (1/3,-1/3), & n\equiv 2 \pmod{3}
    \end{array}\right.
    .  
\end{align*}
}
Since the details are straightforward but tedious calculus, we delay this part of the proof to Section \ref{sec: 2 by 2 reduction}.  

We may now complete the proof of Theorem \ref{thm: spread maximum graphs}.

\begin{proof}[Proof of Theorem \ref{thm: spread maximum graphs}]
Suppose $G$ is a graph on $n$ vertices which maximizes spread.  
By Lemma \ref{lem: few exceptional vertices}, $G = (K_{n_1}\dot{\cup} K_{n_2}^c)\vee K_{n_3}^c$ for some nonnegative integers $n_1,n_2,n_3$ such that $n_1+n_2+n_3 = n$ where 
\begin{align*}
    (n_1,n_2,n_3)
    &= 
    \left(
        \left(
            \dfrac{2}{3}+o(1)
        \right), 
        o(n), 
        \left(
            \dfrac{1}{3}+o(1)
        \right)
        \cdot n
    \right).  
\end{align*}

By Lemma \ref{lem: no exceptional vertices}, if $n$ is sufficiently large, then $n_2 = 0$.  
To complete the proof of the main result, it is sufficient to find the unique maximum of $\spr( K_{n_1}\vee K_{n_2}^c )$, subject to the constraint that $n_1+n_2 = n$.  { This is determined in \cite{gregory2001spread} to be the join of a clique on $\lfloor\frac{2n}{3}\rfloor$ and an independent set on $\lceil \frac{n}{3} \rceil$ vertices.} The interested reader can prove that $n_1$ is the nearest integer to $(2n-1)/3$ by considering the spread of the quotient matrix
\begin{align*}
   % A(n_1,n_2) 
    %&= 
        \left[\begin{array}{cc}
            n_1-1 & n_2\\
            n_1 & 0
        \end{array}\right]
\end{align*}and optimizing the choice of $n_1$.

\end{proof}

\section{The Bipartite Spread Conjecture}\label{sec:bispread}

In \cite{gregory2001spread}, the authors investigated the structure of graphs which maximize the spread over all graphs with a fixed number of vertices $n$ and edges $m$, denoted by $s(n,m)$. In particular, they proved the upper bound
\begin{equation}\label{eqn:spread_bound}
    s(G) \le \lambda_1 + \sqrt{2 m - \lambda^2_1} \le 2 \sqrt{m},
\end{equation}
and noted that equality holds throughout if and only if $G$ is the union of isolated vertices and $K_{p,q}$, for some $p+q \le n$ satisfying $m=pq$ \cite[Thm. 1.5]{gregory2001spread}. This led the authors to conjecture that if $G$ has $n$ vertices, $m \le \lfloor n^2/4 \rfloor$ edges, and spread $s(n,m)$, then $G$ is bipartite \cite[Conj. 1.4]{gregory2001spread}. In this section, we prove an asymptotic form of this conjecture and provide an infinite family of counterexamples to the exact conjecture which verifies that the error in the aforementioned asymptotic result is of the correct order of magnitude. Recall that $s_b(n,m)$, $m \le \lfloor n^2/4 \rfloor$, is the maximum spread over all bipartite graphs with $n$ vertices and $m$ edges. To explicitly compute the spread of certain graphs, we make use of the theory of equitable partitions. In particular, we note that if $\phi$ is an automorphism of $G$, then the quotient matrix of $A(G)$ with respect to $\phi$, denoted by $A_\phi$, satisfies $\Lambda(A_\phi) \subset \Lambda(A)$, and therefore $s(G)$ is at least the spread of $A_\phi$ (for details, see \cite[Section 2.3]{brouwer2011spectra}). Additionally, we require two propositions, one regarding the largest spectral radius of subgraphs of $K_{p,q}$ of a given size, and another regarding the largest gap between sizes which correspond to a complete bipartite graph of order at most $n$. 

Let $K_{p,q}^m$, $0 \le pq-m <\min\{p,q\}$, be the subgraph of $K_{p,q}$ resulting from removing $pq-m$ edges all incident to some vertex in the larger side of the bipartition (if $p=q$, the vertex can be from either set). In \cite{liu2015spectral}, the authors proved the following result.

\begin{proposition}\label{prop:bi_spr}
If $0 \le pq-m <\min\{p,q\}$, then $K_{p,q}^m$ maximizes $\lambda_1$ over all subgraphs of $K_{p,q}$ of size $m$.
\end{proposition}

 We also require estimates regarding the longest sequence of consecutive sizes $m < \lfloor n^2/4\rfloor$ for which there does not exist a complete bipartite graph on at most $n$ vertices and exactly $e$ edges. As pointed out by \cite{pc1}, the result follows quickly by induction. However, for completeness, we include a brief proof.

\begin{proposition}\label{prop:seq}
The length of the longest sequence of consecutive sizes $m < \lfloor n^2/4\rfloor$ for which there does not exist a complete bipartite graph on at most $n$ vertices and exactly $m$ edges is zero for $n \le 4$ and at most $\sqrt{2n-1}-1$ for $n \ge 5$.
\end{proposition}

\begin{proof}
We proceed by induction. By inspection, for every $n \le 4$, $m \le \lfloor n^2/4 \rfloor$, there exists a complete bipartite graph of size $m$ and order at most $n$, and so the length of the longest sequence is trivially zero for $n \le 4$. When $n =m = 5$, there is no complete bipartite graph of order at most five with exactly five edges. This is the only such instance for $n =5$, and so the length of the longest sequence for $n = 5$ is one.

Now, suppose that the statement holds for graphs of order at most $n-1$, for some $n > 5$. We aim to show the statement for graphs of order at most $n$. By our inductive hypothesis, it suffices to consider only sizes $m \ge\lfloor (n-1)^2/4 \rfloor$ and complete bipartite graphs on $n$ vertices. We have
$$\left( \frac{n}{2} + k \right)\left( \frac{n}{2} - k \right) \ge \frac{(n-1)^2}{4} \qquad \text{ for} \quad |k| \le \frac{\sqrt{2n-1}}{2}.$$
When $1 \le k \le \sqrt{2n-1}/2$, the difference between the sizes of $K_{n/2+k-1,n/2-k+1}$ and $K_{n/2+k,n/2-k}$ is at most
\begin{align*}
    \big| E\big(K_{\frac{n}{2}+k-1,\frac{n}{2}-k+1}\big)\big| - \big| E\big(K_{n/2+k,n/2-k}\big)\big|
    &=2k-1  \le  \sqrt{2n-1} -1.
\end{align*} 
Let $k^*$ be the largest value of $k$ satisfying $k \le \sqrt{2n-1}/2$ and $n/2 + k \in \mathbb{N}$. Then
\begin{align*}
    \big| E\big(K_{\frac{n}{2}+k^*,\frac{n}{2}-k^*}\big)\big| &< \left(\frac{n}{2} + \frac{\sqrt{2n-1}}{2} -1 \right)\left(\frac{n}{2} - \frac{\sqrt{2n-1}}{2} +1 \right) \\
    &= \sqrt{2n-1} + \frac{(n-1)^2}{4}  - 1,
\end{align*}
and the difference between the sizes of $K_{n/2+k^*,n/2-k^*}$ and $K_{\lceil \frac{n-1}{2}\rceil,\lfloor \frac{n-1}{2}\rfloor}$ is at most
\begin{align*}
    \big| E\big(K_{\frac{n}{2}+k^*,\frac{n}{2}-k^*}\big)\big| - \big| E\big(K_{\lceil \frac{n-1}{2}\rceil,\lfloor \frac{n-1}{2}\rfloor}\big)\big| &< \sqrt{2n-1} + \frac{(n-1)^2}{4} -\left\lfloor \frac{(n-1)^2}{4} \right\rfloor - 1 \\
    &< \sqrt{2n-1}.
\end{align*}
Combining these two estimates completes our inductive step, and the proof.
% By induction, it suffices to consider only sizes $e\ge\lfloor (n-1)^2/4 \rfloor$, and only consider complete bipartite graphs on $n$ vertices. \marginparsmall{2 verbs in the following sentence. Don't quite understand what you meant. Probably you need to break the sentence into 2, and say in the 1st one "and thus can focus on $k\leq ...$} The quantity $(n/2-k)(n/2+k)$, $k \ge 0$, is greater than $(n-1)^2/4$ for $k < \sqrt{2n-1}/2$, the gap between complete bipartite graphs is bounded above by
% $$(n/2-k+1)(n/2+k-1) - (n/2-k)(n/2+k) = 2k-1 < \sqrt{2n-1}-1$$
% for $k < \sqrt{2n-1}/2$, and
% $$(n/2-\sqrt{2n-1}/2+1)(n/2+\sqrt{2n-1}/2-1)-\lfloor (n-1)^2/4 \rfloor \le \sqrt{2n-1}.$$
% Noting that the desired quantity equals one for $n=5$ completes the proof.
\end{proof}

We are now prepared to prove an asymptotic version of \cite[Conjecture 1.4]{gregory2001spread}, and provide an infinite class of counterexamples that illustrates that the asymptotic version under consideration is the tightest version of this conjecture possible.

\begin{theorem}
$$s(n,m) - s_b(n,m) \le \frac{1+16 \,m^{-3/4}}{m^{3/4}}\, s(n,m)$$
for all $n,m \in \mathbb{N}$ satisfying $m \le \lfloor n^2/4\rfloor$. In addition, for any $\epsilon>0$, there exists some $n_\epsilon$ such that
$$s(n,m) - s_b(n,m) \ge  \frac{1-\epsilon}{m^{3/4}} \, s(n,m)$$
for all $n\ge n_\epsilon$ and some $m \le \lfloor n^2/4\rfloor$ depending on $n$.
\end{theorem}

\begin{proof}
The main idea of the proof is as follows. To obtain an upper bound on $s(n,m) - s_b(n,m)$, we upper bound $s(n,m)$ by $2 \sqrt{m}$ using Inequality \eqref{eqn:spread_bound}, and we lower bound $s_b(n,m)$ by the spread of some specific bipartite graph. To obtain a lower bound on $s(n,m) - s_b(n,m)$ for a specific $n$ and $m$, we explicitly compute $s_b(n,m)$ using Proposition \ref{prop:bi_spr}, and lower bound $s(n,m)$ by the spread of some specific non-bipartite graph.

First, we analyze the spread of $K_{p,q}^m$, $0 < pq-m <q \le p$, a quantity that will be used in the proof of both the upper and lower bound.
Let us denote the vertices in the bipartition of $K_{p,q}^m$ by $u_1,...,u_p$ and $v_1,...,v_{q}$, and suppose without loss of generality that $u_1$ is not adjacent to $v_1,...,v_{pq-m}$. Then
$$\phi = (u_1)(u_2,...,u_p)(v_1,...,v_{pq-m})(v_{pq-m+1},...,v_{q})$$
is an automorphism of $K^m_{p,q}$. The corresponding quotient matrix is given by
$$ A_\phi = \begin{pmatrix} 0 & 0 & 0 & m-(p-1)q \\ 0 & 0 & pq-m & m-(p-1)q \\ 0 & p-1 & 0  & 0 \\ 1 & p-1 & 0 & 0 \end{pmatrix},$$
has characteristic polynomial
$$Q(p,q,m) = \det[A_\phi - \lambda I] = \lambda^4 -m \lambda^2 + (p-1)(m-(p-1)q)(pq-m),$$
and, therefore,
\begin{equation}\label{eqn:bispread_exact}
    s\left(K^m_{p,q}\right) \ge 2 \left( \frac{m + \sqrt{m^2-4(p-1)(m-(p-1)q)(pq-m)}}{2} \right)^{1/2}.
\end{equation}
For $pq = \Omega(n^2)$ and $n$ sufficiently large, this lower bound is actually an equality, as $A(K^m_{p,q})$ is a perturbation of the adjacency matrix of a complete bipartite graph with each partite set of size $\Omega(n)$ by an $O(\sqrt{n})$ norm matrix. For the upper bound, we only require the inequality, but for the lower bound, we assume $n$ is large enough so that this is indeed an equality.

Next, we prove the upper bound. For some fixed $n$ and $m\le \lfloor n^2/4 \rfloor$, let $m = pq -r$, where $p,q,r \in \mathbb{N}$, $p+q \le n$, and $r$ is as small as possible. If $r = 0$, then by \cite[Thm. 1.5]{gregory2001spread} (described above), $s(n,m) = s_b(n,m)$ and we are done. Otherwise, we note that $0<r < \min \{p,q\}$, and so Inequality \eqref{eqn:bispread_exact} is applicable (in fact, by Proposition \ref{prop:seq}, $r = O(\sqrt{n})$). Using the upper bound $s(n,m) \le 2 \sqrt{m}$ and Inequality \eqref{eqn:bispread_exact}, we have 
\begin{equation}\label{eqn:spr_upper}
\frac{s(n,pq-r)-s\left(K^{m}_{p,q}\right)}{s(n,pq-r)} \le 1 - \left(\frac{1}{2}+\frac{1}{2} \sqrt{1-\frac{4(p-1)(q-r) r}{(pq-r)^2}} \right)^{1/2}.
\end{equation}
To upper bound $r$, we use Proposition \ref{prop:seq} with $n'=\lceil 2 \sqrt{m}\rceil \le n$ and $m$. This implies that
$$ r \le \sqrt{2 \lceil 2 \sqrt{m}\rceil -1} -1 < \sqrt{2 ( 2 \sqrt{m}+1) -1} -1 = \sqrt{ 4 \sqrt{m} +1}-1 \le 2 m^{1/4}.$$
Recall that $\sqrt{1-x} \ge 1 - x/2 - x^2/2$ for all $x \in [0,1]$, and so
\begin{align*}
     1 - \big(\tfrac{1}{2} + \tfrac{1}{2} \sqrt{1-x} \big)^{1/2}
     &\le 1 - \big( \tfrac{1}{2} + \tfrac{1}{2} (1 - \tfrac{1}{2}x - \tfrac{1}{2}x^2) \big)^{1/2} = 1 - \big(1 - \tfrac{1}{4} (x + x^2) \big)^{1/2} \\
     &\le 1 - \big(1 - \tfrac{1}{8}(x + x^2) - \tfrac{1}{32}(x + x^2)^2 \big) \\
     &\le \tfrac{1}{8} x + \tfrac{1}{4} x^2
 \end{align*}
% $1 - \big(\tfrac{1}{2} + \tfrac{1}{2} \sqrt{1-x} \big)^{1/2} \le x/8 + x^2/4$ for all $x \in [0,1]$. \marginparsmall{Put equation in align and show details.} 
for $ x \in [0,1]$. To simplify Inequality \eqref{eqn:spr_upper}, we observe that
$$\frac{4(p-1)(q-r)r}{(pq-r)^2} \le \frac{4r}{m} \le  \frac{8}{m^{3/4}}.$$
Therefore,
$$\frac{s(n,pq-r)-s\left(K^{m}_{p,q}\right)}{s(n,pq-r)} \le \frac{1}{m^{3/4}}+ \frac{16}{m^{3/2}}.$$
This completes the proof of the upper bound.

Finally, we proceed with the proof of the lower bound. Let us fix some $0<\epsilon<1$, and consider some sufficiently large $n$. Let $m = (n/2+k)(n/2-k)+1$, where $k$ is the smallest number satisfying $n/2 + k \in \mathbb{N}$ and $\hat \epsilon:=1 - 2k^2/n < \epsilon/2$ (here we require $n = \Omega(1/\epsilon^2)$). Denote the vertices in the bipartition of $K_{n/2+k,n/2-k}$ by $u_1,...,u_{n/2+k}$ and $v_1,...,v_{n/2-k}$, and consider the graph $K^+_{n/2+k,n/2-k}:=K_{n/2+k,n/2-k} \cup \{(v_1,v_2)\}$ resulting from adding one edge to $K_{n/2+k,n/2-k}$ between two vertices in the smaller side of the bipartition. Then
$$ \phi = (u_1 ,...,u_{n/2+k})(v_1, v_2)(v_3,...,v_{n/2-k})$$
is an automorphism of $K^+_{n/2+k,n/2-k}$, and 
$$A_\phi = \begin{pmatrix} 0 & 2 & n/2- k - 2 \\ n/2+k & 1 & 0 \\ n/2+k & 0 & 0 \end{pmatrix} $$
has characteristic polynomial
\begin{align*}
\det[A_\phi - \lambda I] &= -\lambda^3 +\lambda^2 + \left(n^2/4 - k^2\right) \lambda - (n/2+k)(n/2-k-2) \\
&= -\lambda^3 + \lambda^2 + \left(\frac{ n^2}{4} - \frac{(1-\hat \epsilon) n}{2} \right)\lambda - \left(\frac{n^2}{4} - \frac{(3-\hat \epsilon)n}{2} -\sqrt{2(1-\hat \epsilon)n} \right).
\end{align*}
By matching higher order terms, we obtain
$$ \lambda_{max}(A_\phi) = \frac{n}{2}-\frac{1-\hat \epsilon}{2} + \frac{\left( 8-(1-\hat \epsilon)^2 \right)}{4 n} +o(1/n),$$
$$\lambda_{min}(A_\phi) = -\frac{n}{2}+\frac{1-\hat \epsilon}{2} + \frac{\left( 8+(1-\hat \epsilon)^2 \right)}{4 n} +o(1/n),$$
and
$$s(K^+_{n/2+k,n/2-k}) \ge n-(1-\hat \epsilon)-\frac{(1-\hat \epsilon)^2}{2n} + o(1/n).$$

Next, we aim to compute $s_b(n,m)$, $m = (n/2+k)(n/2-k)+1$. By Proposition \ref{prop:bi_spr}, $s_b(n,m)$ is equal to the maximum of $s(K^m_{n/2+\ell,n/2-\ell})$ over all $\ell \in [0,k-1]$, $k-\ell \in \mathbb{N}$. As previously noted, for $n$ sufficiently large, the quantity $s(K^m_{n/2+\ell,n/2-\ell})$ is given exactly by Equation (\ref{eqn:bispread_exact}), and so the optimal choice of $\ell$ minimizes
\begin{align*}
    f(\ell) &:= (n/2+\ell-1)(k^2-\ell^2-1)(n/2-\ell-(k^2-\ell^2-1))\\
    &=(n/2+\ell)\big((1-\hat \epsilon)n/2-\ell^2\big)\big(\hat \epsilon n/2 +\ell^2-\ell \big) + O(n^2).
\end{align*}
We have
$$ f(k-1) = (n/2+k-2)(2k-2)(n/2-3k+3),$$
and if $\ell \le \frac{4}{5} k$, then $f(\ell) = \Omega(n^3)$. Therefore the minimizing $\ell$ is in $ [\frac{4}{5} k,k]$. The derivative of $f(\ell)$ is given by
\begin{align*}
    f'(\ell) &=(k^2-\ell^2-1)(n/2-\ell-k^2+\ell^2+1)\\
    &\qquad-2\ell(n/2+\ell-1)(n/2-\ell-k^2+\ell^2+1)\\
    &\qquad+(2\ell-1)(n/2+\ell-1)(k^2-\ell^2-1).
\end{align*}
For $\ell \in  [\frac{4}{5} k,k]$,
\begin{align*}
    f'(\ell)
    &\le \frac{n(k^2-\ell^2)}{2}-\ell n(n/2-\ell-k^2+\ell^2)+2\ell(n/2+\ell)(k^2-\ell^2)\\
    &\le \frac{9 k^2 n}{50} - \tfrac{4}{5} kn(n/2-k- \tfrac{9}{25} k^2)+ \tfrac{18}{25} (n/2+k)k^3 \\
    &= \frac{81 k^3 n}{125}-\frac{2 k n^2}{5} + O(n^2)\\
    &=kn^2\left(\frac{81(1-\hat \epsilon)}{250}-\frac{2}{5}\right)+O(n^2)<0
\end{align*}
for sufficiently large $n$. This implies that the optimal choice is $\ell = k-1$, and $s_b(n,m) = s(K^m_{n/2+k-1,n/2-k+1})$. The characteristic polynomial $Q(n/2+k-1,n/2-k+1,n^2/4 -k^2+1)$ equals
$$ \lambda^4 - \left(n^2/4 -k^2+1 \right)\lambda^2+2(n/2+k-2)(n/2-3k+3)(k-1).$$
By matching higher order terms, the extreme root of $Q$ is given by
$$\lambda = \frac{n}{2}  -\frac{1-\hat \epsilon}{2} - \sqrt{\frac{2(1-\hat \epsilon)}{n}}+\frac{27-14\hat \epsilon-\hat \epsilon^2}{4n}+o(1/n),$$
and so
$$ s_b(n,m) = n -(1-\hat \epsilon) - 2 \sqrt{\frac{2(1-\hat \epsilon)}{n}}+\frac{27-14\hat \epsilon-\hat \epsilon^2}{2n}+o(1/n),$$
and 
\begin{align*}
    \frac{s(n,m)-s_b(n,m)}{s(n,m)} &\ge \frac{2^{3/2}(1-\hat \epsilon)^{1/2}}{n^{3/2}} - \frac{14-8\hat \epsilon}{n^2}+o(1/n^2)\\
    &=\frac{(1-\hat \epsilon)^{1/2}}{m^{3/4}} + \frac{(1-\hat \epsilon)^{1/2}}{(n/2)^{3/2}}\bigg[1-\frac{(n/2)^{3/2}}{m^{3/4}}\bigg] - \frac{14-8\hat \epsilon}{n^2}+o(1/n^2)\\
    &\ge \frac{1-\epsilon/2}{m^{3/4}} +o(1/m^{3/4}).
\end{align*}
This completes the proof.
\end{proof}

\section{Concluding remarks}\label{sec: conclusion}

In this work we provided a proof of the spread conjecture for sufficiently large $n$, a proof of an asymptotic version of the bipartite spread conjecture, and an infinite class of counterexamples that illustrates that our asymptotic version of this conjecture is the strongest result possible. There are a number of interesting future avenues of research, some of which we briefly describe below. These avenues consist primarily of considering the spread of more general classes of graphs (directed graphs, graphs with loops) or considering more general objective functions.

Our proof of the spread conjecture for sufficiently large $n$ immediately implies a nearly-tight estimate for the adjacency matrix of undirected graphs with loops, also commonly referred to as symmetric $0-1$ matrices. Given a directed graph $G = (V,\mathcal{A})$, the corresponding adjacency matrix $A$ has entry $A_{i,j} = 1$ if the arc $(i,j) \in \mathcal{A}$, and is zero otherwise. In this case, $A$ is not necessarily symmetric, and may have complex eigenvalues. One interesting question is what digraph of order $n$ maximizes the spread of its adjacency matrix, where spread is defined as the diameter of the spectrum. Is this more general problem also maximized by the same set of graphs as in the undirected case? This problem for either loop-less directed graphs or directed graphs with loops is an interesting question, and the latter is equivalent to asking the above question for the set of all $0-1$ matrices.

Another approach is to restrict ourselves to undirected graphs or undirected graphs with loops, and further consider the competing interests of simultaneously producing a graph with both $\lambda_1$ and $-\lambda_n$ large, and understanding the trade-off between these two goals. To this end, we propose considering the class of objective functions 
$$f(G; \beta) = \beta \lambda_1(G) - (1-\beta) \lambda_n(G), \qquad \beta \in [0,1].$$
When $\beta = 0$, this function is maximized by the complete bipartite graph $K_{\lceil n /2 \rceil, \lfloor n/2 \rfloor}$ and when $\beta = 1$, this function is maximized by the complete graph $K_n$. This paper treats the specific case of $\beta = 1/2$, but none of the mathematical techniques used in this work rely on this restriction. In fact, the structural graph-theoretic results of Section \ref{sec:graphs}, suitably modified for arbitrary $\beta$, still hold (see the thesis \cite[Section 3.3.1]{urschel2021graphs} for this general case). Understanding the behavior of the optimum between these three well-studied choices of $\beta = 0,1/2,1$ is an interesting future avenue of research.

More generally, any linear combination of graph eigenvalues could be optimized over any family of graphs. Many sporadic examples of this problem have been studied and Nikiforov \cite{Nikiforov} proposed a general framework for it and proved some conditions under which the problem is well-behaved. We conclude with some specific instances of the problem that we think are most interesting. 

Given a graph $F$, maximizing $\lambda_1$ over the family of $n$-vertex $F$-free graphs can be thought of as a spectral version of Tur\'an's problem. Many papers have been written about this problem which was proposed in generality in \cite{Nikiforov3}. We remark that these results can often strengthen classical results in extremal graph theory. Maximizing $\lambda_1 + \lambda_n$ over the family of triangle-free graphs has been considered in \cite{Brandt} and is related to an old conjecture of Erd\H{o}s on how many edges must be removed from a triangle-free graph to make it bipartite \cite{erdos}. In general it would be interesting to maximize $\lambda_1 + \lambda_n$ over the family of $K_r$-free graphs. When a graph is regular the difference between $\lambda_1$ and $\lambda_2$ (the spectral gap) is related to the graph's expansion properties. Aldous and Fill \cite{AldousFill} asked to minimize $\lambda_1 - \lambda_2$ over the family of $n$-vertex connected regular graphs. Partial results were given by \cite{quartic1, quartic2, Guiduli1, Guiduli2}. A nonregular version of the problem was proposed by Stani\'c \cite{Stanic} who asked to minimize $\lambda_1-\lambda_2$ over connected $n$-vertex graphs. Finally, maximizing $\lambda_3$ or $\lambda_4$ over the family of $n$-vertex graphs seems to be a surprisingly difficult question and even the asymptotics are not known (see \cite{Nikiforov2}).

\section*{Acknowledgements}
 The work of A. Riasanovsky was supported in part by NSF award DMS-1839918 (RTG). The work of M. Tait was supported in part by NSF award DMS-2011553. The work of J. Urschel was supported in part by ONR Research Contract N00014-17-1-2177. The work of J.~Breen was supported in part by NSERC Discovery Grant RGPIN-2021-03775. The authors are grateful to Louisa Thomas for greatly improving the style of presentation.

{ \small 
	\bibliographystyle{plain}
	\bibliography{bib-pream/spread.bib} }

\appendix
\renewcommand{\thesection}{\Alph{section}}

\section{Technical proofs}\label{sec: ugly}

\subsection{Reduction to 17 cases}\label{appendix 17 cases}

Now, we introduce the following specialized notation.  
For any nonempty set $S\subseteq V(G^*)$ and any labeled partition $(I_i)_{i\in S}$ of $[0,1]$, we define the stepgraphon $W_\ii$ as follows.  
For all $i,j\in S$, $W_\ii$ equals $1$ on $I_i\times I_j$ if and only if $ij$ is an edge (or loop) of $G^*$, and $0$ otherwise.  
If $\alpha = (\alpha_i)_{i\in S}$ where $\alpha_i = m(I_i)$ for all $i\in S$, we may write $W_\alpha$ to denote the graphon $W_\ii$ up to weak isomorphism.  
\\

{To make the observations from Section \ref{sub-sec: cases} more explicit, we note that Theorem \ref{thm: reduction to stepgraphon} implies that a spread-optimal graphon has the form $W = W_\ii$ where $\ii = (I_i)_{i\in S}$ is a labeled partition of $[0,1]$, $S\subseteq [7]$, and each $I_i$ is measurable with positive positive measure.   
Since $W$ is a stepgraphon, its extreme eigenfunctions may be taken to be constant on $I_i$, for all $i\in S$.  
With $f,g$ denoting the extreme eigenfunctions for $W$, we may let $f_i$ and $g_i$ be the constant value of $f$ and $g$, respectively, on step $S_i$, for all $i\in S$.  
Appealing again to Theorem \ref{thm: reduction to stepgraphon}, we may assume without loss of generality that $f_i\geq 0$ for all $i\in S$, and for all $i\in S$, $g_i\geq 0$ implies that $i\in\{1,2,3,4\}$.  
By Lemma \ref{lem: local eigenfunction equation}, for each $i\in S$, $\mu f_i^2-\nu g_i^2 = \mu-\nu$.  
Combining these facts, we note that $f_i$ and $g_i$ belong to specific intervals as in Figure \ref{fig: f g interval}.  }

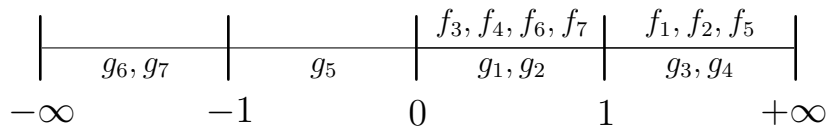
\begin{figure}[ht]
    \centering
    \begin{tikzpicture}
	\node[label={[yshift=-50pt]{\large $-\infty$}}] (A) at (-5, 0) {
	    \Huge $\left|\right.$
    };
    
    \node[label={[yshift=-50pt]{\large $-1$}}] (B) at (-2.5, 0) {
	    \Huge $\left|\right.$
    };
    
    \node[label={[yshift=-50pt]{\large $0$}}] (C) at (0, 0) {
	    \Huge $\left|\right.$
    };
    
    \node[label={[yshift=-50pt]{\large $1$}}] (D) at (2.5, 0) {
	    \Huge $\left|\right.$
    };
    
    \node[label={[yshift=-50pt]{\large $+\infty$}}] (E) at (5, 0) {
	    \Huge $\left.\right|$
    };

    \draw (A.center) -- (B.center) 
    node [midway, below=0pt] {
        $g_6,g_7$
    };
    
    \draw (B.center) -- (C.center) 
    node [midway, below=0pt] {
        $g_5$
    };
    
    \draw (C.center) -- (D.center) 
    node [midway, above=0pt] {
        $f_3,f_4,f_6,f_7$
    };
    
    \draw (C.center) -- (D.center) 
    node [midway, below=0pt] {
        $g_1,g_2$
    };
    
    \draw (D.center) -- (E.center) 
    node [midway, above=0pt] {
        $f_1,f_2,f_5$
    };
    
    \draw (D.center) -- (E.center) 
    node [midway, below=0pt] {
        $g_3,g_4$
    };
\end{tikzpicture}
    \caption{
        Intervals containing the quantities $f_i$ and $g_i$.  
        Note that $f_i$ and $g_i$ are only defined for all $i\in S$.  
    }
\label{fig: f g interval}
\end{figure}

{For convenience, we define the following sets $F_i$ and $G_i$, for all $i\in S$.  
First, let $\uu := [0,1]$ and $\vv := [1,+\infty]$.  
With some abuse of notation, we denote $-\uu = [-1,0]$ and $-\vv = [-\infty,-1]$.  }

For each $i\in V(G^*)$, we define the intervals $F_i$ and $G_i$ by 
\begin{align*}
    (F_i, G_i) 
    &:= 
    \left\{\begin{array}{rl}
        (\vv, \uu), &i\in \{1,2\}\\
        (\uu, \vv), &i\in \{3,4\} \\
        (\vv, -\uu), &i=5 \\
        (\uu, -\vv), &i\in \{6,7\}
    \end{array}\right. 
    .
    % \\
    % \red{\text{ or } (F_i, G_i) }
    % &:= 
    % \red{\left\{\begin{array}{rl}
    %     ([1,+\infty], [0,1]), &i\in \{1,2\}\\
    %     ([0,1], [1,\infty+]), &i\in \{3,4\} \\
    %     ([1,+\infty], [-1,0]), &i=5 \\
    %     ([0,1], [-\infty, -1]), &i\in \{6,7\}
    % \end{array}\right. 
    % .}
\end{align*}

% \begin{table}[]
% \begin{tabular}{r|r|r}
%               & $F_i$         & $G_i$          \\ \hline \hline
% $i\in\{1,2\}$ & $[1,+\infty]$ & $[0,1]$        \\ \hline
% $i\in\{3,4\}$ & $[0,1]$       & $[1,+\infty]$  \\ \hline
% $i = 5$       & $[1,+\infty]$ & $[-1,0]$       \\ \hline
% $i\in\{6,7\}$ & $[0,1]$       & $[-\infty,-1]$
% \end{tabular}
% \end{table}

Given that the set $S$ and the quantities $(\alpha_i,f_i,g_i)_{i\in S}$ are clear from context, we label the following equation: 
\begin{align}
    \sum_{i\in S}
        \alpha_i 
    &= 
    \sum_{i\in S}
        \alpha_if_i^2
    =
    \sum_{i\in S}
        \alpha_ig_i^2
    = 
        1
    \label{eq: program norms}
    .  
\end{align}
Furthermore when $i\in S$ is understood from context, we define the equations 
\begin{align}
    \mu f_i^2 - \nu g_i^2
    &= 
        \mu - \nu
    \label{eq: program ellipse} 
    \\
    \sum_{j\in N_i\cap S}\alpha_jf_j
    &=
        \mu f_i
    \label{eq: program eigen f}
    \\
    \sum_{j\in N_i\cap S}\alpha_jg_j
    &=
        \nu g_i
    \label{eq: program eigen g}
\end{align}
Additionally, we consider the following inequalities.  
For all $S\subseteq V(G^*)$ and all distinct $i,j\in S$, 
\begin{align}\label{ieq: program inequality constraint}
    f_if_j - g_ig_j
    &\left\{\begin{array}{rl}
        \geq 0, &ij\in E(G^*)\\
        \leq 0, &ij\notin E(G^*)
    \end{array}\right.  
\end{align}
Finally, for all nonempty $S\subseteq V(G^*)$, we define the constrained-optimization problem $\SPRS$ by: 
\begin{align*}
    (\text{SPR}_S): 
    \left\{\begin{array}{rll}
        \max 
        &
            \mu-\nu
        \\
        \text{s.t}
        &
            \text{Equation }\eqref{eq: program norms}
        \\
        &
            \text{Equations }
                \eqref{eq: program ellipse}, 
                \eqref{eq: program eigen f}, \text{ and }
                \eqref{eq: program eigen g}
        &
            \text{ for all }i\in S
        \\
        &
            \text{Inequality }
                \eqref{ieq: program inequality constraint}
        &
            \text{ for all distinct }i,j\in S
        \\
        & 
            (\alpha_i,f_i,g_i)\in [0,1] \times  F_i \times G_i
        &
            \text{ for all }i\in S
        \\
        &
            \mu,\nu\in\RR
    \end{array}\right.
    .  
\end{align*}
For completeness, we state and prove the following observation.  

\begin{proposition}\label{prop: problem solutions}
    Let $W\in\ww$ such that $\spr(W) = \max_{U\in \mathcal{W}}\spr(U)$ and write $\mu,\nu$ for the maximum and minimum eigenvalues of $W$, with corresponding unit eigenfunctions $f,g$.  
    Then for some nonempty set $S\subseteq V(G^*)$, the following holds.  
    There exists a triple $(I_i, f_i, g_i)_{i\in S}$, where $(I_i)_{i\in S}$ is a labeled partition of $[0,1]$ with parts of positive measure and $f_i,g_i\in \RR$ for all $i\in S$, such that: 
    \begin{enumerate}[(i)]
        \item\label{item: W = W_I} $W = W_\ii$.  
        \item\label{item: f,g constants} Allowing the replacement of $f$ by $-f$ and of $g$ by $-g$, for all $i\in S$, $f$ and $g$ equal $f_i$ and $g_i$ a.e. on $I_i$.  
        \item\label{item: problem solution} With $\alpha_i := m(I_i)$ for all $i\in S$, $\SPRS$ is solved by $\mu,\nu$, and $(\alpha_i, f_i, g_i)_{i\in S}$.  
    \end{enumerate}
\end{proposition}
\begin{proof}
First we prove Item \eqref{item: W = W_I}.  
By Theorem \ref{thm: reduction to stepgraphon} and the definition of $G^*$, there exists a nonempty set $S\subseteq V(G^*)$ and a labeled partition $\ii = (I_i)_{i\in S}$ such that $W = W_\ii$.  
By merging any parts of measure $0$ into some part of positive measure, we may assume without loss of generality that $m(I_i) > 0$ for all $i\in S$.  
So Item \eqref{item: W = W_I} holds.  

For Item \eqref{item: f,g constants}, the eigenfunctions corresponding to the maximum and minimum eigenvalues of a stepgraphon must be constant on each block by convexity and the Courant-Fischer Min-Max Theorem.

% Now we prove item \eqref{item: both sides}.  
% Again by Theorem \ref{thm: reduction to stepgraphon}, we may assume without loss of generality that $g$ is positive on $\dunion_{i\in S\cap\{1,2,3,4\}}I_i$ and negative on $\dunion_{i\in S\cap\{5,6,7\}}I_i$, and also by Proposition \ref{prop: PF eigenfunction}, we may assume without loss of generality that $f$ is positive on $[0,1]$.  
% Since $\mu\neq \nu$, $f$ and $g$ are orthogonal, and item \eqref{item: both sides} follows.  
% Additionally, we note that item \eqref{item: f,g constants} follows from the fact that $W$ is a stepgraphon whose steps are $(I_i)_{i\in S}$.  
% \\
Finally, we prove Item \eqref{item: problem solution}, we first prove that for all $i\in V(G^*)$, $(f_i,g_i)\in F_i\times G_i$.  
By Lemma \ref{lem: local eigenfunction equation}, 
\begin{align*}
    \mu f_i^2-\nu g_i^2 &= \mu-\nu 
\end{align*}
for all $i\in S$.  
In particular, either $f_i^2\leq 1\leq g_i^2$ or $g_i^2\leq 1\leq f_i^2$.  
By Lemma \ref{lem: K = indicator function}, for all $i,j\in S$, $f_if_j-g_ig_j\neq 0$ and $ij\in E(G)$ if and only if $f_if_j-g_ig_j > 0$. Note that the loops of $G^*$ are $1, 2, $ and $5$.  
It follows that for all $i\in S$, $f_i^2 > 1 > g_i^2$ if and only if $i\in\{1,2,5\}$, and $g_i^2>1>f_i^2$, otherwise.  
Since $f$ is positive on $[0,1]$, this completes the proof that $f_i\in F_i$ for all $i\in S$.  
Similarly since $g$ is positive on $\bigcup_{i\in \{1,2,3,4\}\cap S}I_i$ and negative on $\bigcup_{i\in \{5,6,7\}}I_i$, by inspection $g_i\in G_i$ for all $i\in S$.  
Similarly, Inequalities \eqref{ieq: program inequality constraint} follow directly from Lemma \ref{lem: K = indicator function}.  

Continuing, we note the following.  
Since $W$ is a stepgraphon, if $\lambda\neq 0$ is an eigenvalue of $W$, there exists a $\lambda$-eigenfunction $h$ for $W$ such that for all $i\in S$, $h = h_i$ on $I_i$ for some $h_i\in \RR$.  
Moreover for all $i\in S$, since $m(I_i) > 0$, 
\begin{align*}
    \lambda h_i 
    &=
    \sum_{i\in S}
        \alpha_ih_i.  
\end{align*}
In particular, any solution to $\SPRS$ is at most $\mu-\nu$.  
Since $f,g$ are eigenfunctions corresponding to $W$ and the eigenvalues $\mu,\nu$, respectively, Equations \eqref{eq: program eigen f}, and \eqref{eq: program eigen g} hold.  
Finally since $(I_i)_{i\in S}$ is a partition of $[0,1]$ and since $\|f\|_2^2 = \|g\|_2^2 = 1$, Equation \eqref{eq: program norms} holds.  
So $\mu,\nu$, and $(\alpha_i, f_i, g_i)_{i\in S}$ lie in the domain of $\SPRS$.  
This completes the proof of item \eqref{item: problem solution}, and the desired claim.  
\end{proof}

We enhance Proposition \ref{prop: problem solutions} as follows.  

\begin{lemma}\label{lem: 19 cases}
    Proposition \ref{prop: problem solutions} holds with the added assumption that $S\in\sset$.  
% 	\red{Suppose $W$ is a graphon maximizing spread.  
%     Then without loss of generality, there exists a nonempty set $S\subseteq \{1,\dots,7\}$ and a vector $\alpha$ of positive weights indexed by $S$ and summing to $1$ such that $W = W_{\alpha}$ for some vector $\alpha$ of positive weights, indexed by some set $S\subseteq \{1,\dots,7\}$ where 
% 	\begin{align*}
% 	    S 
% 	    &\in 
% 	    \left\{\begin{array}{l}
% 	        \{1,2,3,4,5,6,7\}, \{2,3,4,5,6,7\}, \{1,2,4,5,6,7\}, \{1,2,3,4,5,7\}, \\
% 	        \{2,4,5,6,7\}, \{1,4,5,6,7\}, \{1,2,4,5,7\},  \{1,2,3,4,7\}, \\
% 	        \{4,5,6,7\}, \{2,4,5,7\}, \{1,5,6,7\}, \{1,4,5,7\}, \{1,2,4,7\},  \\
% 	        \{4,5,7\}, \{1,5,7\}, \{1,4,7\}, \{1,7\}
% 	    \end{array}
% 	    \right\}
% 	    \quad 
% 	    .
% 	\end{align*}
% 	%Moreover, the internal divisions separate according to the sign of the eigenfunction corresponding to the minimum eigenvalue of $K$.  
% 	Moreover, without loss of generality, the eigenfunction for the minimum eigenvalue of $W$ is positive on the steps corresponding to $S\cap \{1,2,3,4\}$ and negative on the steps $S\cap \{5,6,7\}$.  }
\end{lemma}

\begin{proof}
We begin our proof with the following claim.  
\\
\\
{\bf Claim A: }
    Suppose $i\in S$ and $j\in V(G^*)$ are distinct such that $N_i\cap S = N_j\cap S$.  
    Then Proposition \ref{prop: problem solutions} holds with the set $S' := (S\setminus\{i\})\cup \{j\}$ replacing $S$.  

First, we define the following quantities.  
For all $k\in S'\setminus \{j\}$, let $(f_k', g_k', I_k') := (f_k, g_k, I_k)$, and also let $(f_j', g_j') := (f_i, g_i)$.  
If $j\in S$, let $I_j' := I_i\cup I_j$, and otherwise, let $I_j' := I_i$.  
Additionally let $\ii' := (I_k')_{k\in S'}$ and for each $k\in S'$, let $\alpha_k' := m(I_k')$.  
By the criteria from Proposition \ref{prop: problem solutions}, the domain criterion $(\alpha_k', f_k', g_k') \in [0,1]\times F_k\times G_k$ as well as Equation \eqref{eq: program ellipse} holds for all $k\in S'$.  
Since we are reusing $\mu,\nu$, the constraint $\mu,\nu\in \RR$ also holds.  

It suffices to show that Equation \eqref{eq: program norms} holds, and that Equations \eqref{eq: program eigen f} and \eqref{eq: program eigen g} hold for all $k\in S'$.  
To do this, we first note that for all $k\in S'$, $f = f_k'$ and $g = g_k'$ on $I_k'$.  
By definition, $f = f_k$ and $g = g_k$ on $I_k' = I_k$ for all $k\in S'\setminus\{j\}$ as needed by Claim A.  
Now suppose $j\notin S$.  
Then $f = f_i = f_j'$ and $g = g_i = g_j'$ and $I_j' = I_i$ on the set $I_i = I_j'$, matching Claim A.  
Finally, suppose $j\in S$.  
Note by definition that $f = f_i = f_j'$ and $g = g_i = g_j'$ on $I_i$.  
Since and $I_j' = I_i\cup I_j$, it suffices to prove that $f = f_j'$ and $g = g_j'$ on $I_j$.  
We first show that $f_j = f_i$ and $g_j = g_i$.  
Indeed, 
\begin{align*}
    \mu f_j 
    &=
    \sum_{k\in N_j\cap S}
        \alpha_k f_k
    = 
    \sum_{k\in N_i\cap S}
        \alpha_k f_k
    = 
    \mu f_i
\end{align*}
and since $\mu\neq 0$, $f_j = f_i$.  
Similarly, $g_j = g_i$.  
So $f = f_j = f_i = f_j'$ and $g = g_j = g_i = g_j'$ on the set $I_j' = I_i \cup I_j$.

Finally, we claim that $W_{\ii'} = W$.  
Indeed, this follows directly from Lemma \ref{lem: K = indicator function} and the fact that $W = W_\ii$.  
Since $\ii'$ is a partition of $[0,1]$ and since $f,g$ are unit eigenfunctions for $W$ Equation \eqref{eq: program norms} holds, and Equations \eqref{eq: program eigen f} and \eqref{eq: program eigen g} hold for all $k\in S'$.  
This completes the proof of Claim A.  

Next, we prove the following claim.  
\\
{\bf Claim B: }
    If $S$ satisfies the criteria of Proposition \ref{prop: problem solutions}, then without loss of generality the following holds.  
    \begin{enumerate}[(a)]
        \item\label{item: vertex 1}
            If there exists some $i\in S$ such that $N_i = S$, then $i = 1$.  
        \item\label{item: vertices 1234}
            $S\cap \{1,2,3,4\}\neq\emptyset$.  
        \item\label{item: vertices 234}
            $S\cap \{2,3,4\}$ is one of $\emptyset, \{4\}, \{2,4\}$, and $\{2,3,4\}$.  
        \item\label{item: vertices 567}
            $S\cap \{5,6,7\}$ is one of $\{7\}, \{5,7\}$, and $\{5,6,7\}$.  
    \end{enumerate}

Since $N_1\cap S = S = N_i$, item \eqref{item: vertex 1} follows from Claim A applied to the pair $(i,1)$.  
Since $f,g$ are orthogonal and $f$ is positive on $[0,1]$, $g$ is positive on a set of positive measure, so item \eqref{item: vertices 1234} holds.  

To prove item \eqref{item: vertices 234}, we have $4$ cases.  
If $S\cap \{2,3,4\} = \{2\}$, then $N_2\cap S = N_1\cap S$ and we may apply Claim A to the pair $(2,1)$.  
If $S\cap \{2,3,4\} = \{3\}$ or $\{3,4\}$,  then $N_3\cap S = N_4\cap S$ and we may apply Claim A to the pair $(3,4)$.  
If $S\cap \{2,3,4\} = \{2,3\}$, then $N_2\cap S = N_1\cap S$ and we may apply Claim A to the pair $(2,1)$.  
So item \eqref{item: vertices 234} holds.  
For item \eqref{item: vertices 567}, we reduce $S\cap \{5,6,7\}$ to one of $\emptyset, \{7\}, \{5,7\}$, and $\{5,6,7\}$ in the same fashion.  
To eliminate the case where $S\cap \{5,6,7\} = \emptyset$, we simply note that since $f$ and $g$ are orthogonal and $f$ is positive on $[0,1]$, $g$ is negative on a set of positive measure.  
This completes the proof of Claim B.  
\begin{table}[ht]
    \centering
    \begin{tabular}{c||r|r|r|r}
    \multicolumn{1}{l||}{}        & \multicolumn{1}{c|}{$\emptyset$} & \multicolumn{1}{c|}{$\{4\}$} & \multicolumn{1}{c|}{$\{2,4\}$} & \multicolumn{1}{c}{$\{2,3,4\}$} \\ \hline\hline
    \multirow{2}{*}{$\{7\}$}     & \multirow{2}{*}{$1|7$}           & $4|7$                        & $24|7$                         & $234|7$                         \\
                                 &                                  & $1|4|7$                      & $1|24|7$                       & $1|234|7$                       \\ \hline
    \multirow{2}{*}{$\{5,7\}$}   & \multirow{2}{*}{$1|57$}          & $4|57$                       & $24|57$                        & $234|57$                        \\
                                 &                                  & $1|4|57$                     & $1|24|57$                      & $1|234|57$                      \\ \hline
    \multirow{2}{*}{$\{5,6,7\}$} & \multirow{2}{*}{$1|567$}         & $4|567$                      & $24|567$                       & $234|567$                       \\
                                 &                                  & $1|4|567$                    & $1|24|567$                     & $1|234|567$                    
    \end{tabular}
    
    \caption{The $21$ sets which arise from repeated applications of Claim B.  }
    \label{tab: table 21}
\end{table}

After repeatedly applying Claim B, we may replace $S$ with one of the cases found in Table \ref{tab: table 21}.  
Let $\mathcal{S}_{21}$ denote the sets in Table \ref{tab: table 21}.  
By definition, 
\begin{align*}
    \mathcal{S}_{21}
    &= 
    \sset
    \bigcup \left\{
        \{4,7\}, \{2,4,7\}, \{2,3,4,7\}, \{2,3,4,5,7\}
    \right\}
    .  
\end{align*}
Finally, we eliminate the $4$ cases in $\mathcal{S}_{21}\setminus \sset$.  
If $S = \{4,7\}$, then $W$ is a bipartite graphon, hence $\spr(W) \leq 1$, a contradiction since $\max_{U\in\ww}\spr(W) > 1$.  

For the three remaining cases, let $\tau$ be the permutation on $\{2,\dots,7\}$ defined as follows.  
For all $i\in \{2,3,4\}$, $\tau(i) := i+3$ and $\tau(i+3) := i$.  
If $S$ is among $\{2,4,7\}, \{2,3,4,7\}, \{2,3,4,5,7\}$, we apply $\tau$ to $S$ in the following sense.  
Replace $g$ with $-g$ and replace $(\alpha_i, I_i, f_i, g_i)_{i\in S}$ with $(\alpha_{\tau(i)}, I_{\tau(i)}, f_{\tau(i)}, -g_{\tau(i)})_{i\in \tau(S)}$.  
By careful inspection, it follows that $\tau(S)$ satisfies the criteria from Proposition \ref{prop: problem solutions}.  
Since $\tau(\{2,4,7\}) = \{4,5,7\}$, $\tau(\{2,3,4,7\}) = \{4,5,6,7\}$, and $\tau(\{2,3,4,5,7\}) = \{2,4,5,6,7\}$, this completes the proof.  
\end{proof}

\subsection{Proof of Lemma \ref{lem: SPR457}}

Let $(\alpha_4, \alpha_5, \alpha_7)$ be a solution to $\SPR_{457}$. 

First, let $T := 
\{(\varepsilon_1,\varepsilon_2)\in(-1/3, 2/3)\times (-2/3, 1/3) : \varepsilon_1+\varepsilon_2 \in (0,1)\}$, and for all $\varepsilon = (\varepsilon_1,\varepsilon_2)\in T$, let 
\begin{align*}
    M(\varepsilon)
    &:=
        \left[\begin{array}{ccc}
            2/3-\varepsilon_1 & 0 & 1/3-\varepsilon_2\\
            0 & 0 & 1/3-\varepsilon_2\\
            2/3-\varepsilon_1 & \varepsilon_1 + \varepsilon_2 & 0
        \end{array}\right]
    .  
\end{align*}
As a motivation, suppose $\mu,\nu$, and $(\alpha_4,\alpha_5,\alpha_7)$ are part of a solution to $\SPR_{\{4,5,7\}}$.  
Then with $\varepsilon := (\varepsilon_1,\varepsilon_2) = (2/3-\alpha_5, 1/3-\alpha_4)$, $\varepsilon\in T$ and $\mu,\nu$ are the maximum and minimum eigenvalues of $M(\varepsilon)$, respectively.  
By the end of the proof, we show that any solution of $\SPR_{\{4,5,7\}}$ has $\alpha_7 = 0$.  
\\

To proceed, we prove the following claims.  
\\
\\
{\bf Claim A: }
    For all $\varepsilon\in T$, $M(\varepsilon)$ has two distinct positive eigenvalues and one negative eigenvalue.  

Since $M(\varepsilon)$ is diagonalizable, it has $3$ real eigenvalues which we may order as $\mu\geq \delta\geq \nu$.  
Since $\mu\delta\nu = \det(M(\varepsilon)) = -\alpha_4\alpha_5\alpha_7\neq 0 < 0$, $M(\varepsilon)$ has an odd number of negative eigenvalues.  
Since $0 < \alpha_5 = \mu + \delta + \nu$, it follows that $\mu \geq \delta > 0 > \nu$.  
Finally, note by the Perron-Frobenius Theorem that $\mu > \delta$.  
This completes the proof of Claim A.  
\\

Next, we define the following quantities, treated as functions of $\varepsilon$ for all $\varepsilon\in T$.  
For convenience, we suppress the argument ``$\varepsilon$'' in most places.  
Let $k(x) = ax^3+bx^2+cx+d$ be the characteristic polynomial of $M(\varepsilon)$.  
By inspection, 
\begin{align*}
    &a
    =
        1
    &b
    =
        \varepsilon_1-\dfrac{2}{3}
    \\
    &c
    =
        \dfrac{
            (3\varepsilon_2+2)
            (3\varepsilon_2-1)
        }
        {9}
    &d
    =
        \dfrac{
            (\varepsilon_1+\varepsilon_2)
            (3\varepsilon_1-2)
            (3\varepsilon_2-1)
        }
        {9}
\end{align*}
Continuing, let 
\begin{align*}
    &p
    := 
        \dfrac{
            3a
                c
            -b^2
        }
        {
            3a^2
        }
    &q
    := 
        \dfrac{
            2b^3
            -9a
                b
                c
            +27a^2
                d
        }
        {
            27a^3
        }
    \\
    &A
    :=
        2\sqrt{
           \dfrac{-p}
           {3}
        }
    &B
    :=
        \dfrac{
            -b
        }
        {3a}
    \\
    &\phi
    :=
        \arccos\left(
            \dfrac{
                3q
            }
            {
                A
                p
            }
        \right).  
\end{align*}
Let $S(\varepsilon)$ be the difference between the maximum and minimum eigenvalues of $M(\varepsilon)$.  
We show the following claim.  
\\
\\
{\bf Claim B: }
    For all $\varepsilon\in T$, 
    \begin{align*}
        S(\varepsilon) 
        &=
            \sqrt{3}\cdot A(\varepsilon)\cdot\cos\left(
                \dfrac{2\phi(\varepsilon) - \pi}{6}
            \right).  
    \end{align*}
    Moreover, $S$ is analytic on $T$.  

Indeed, by Vi\'{e}te's Formula, using the fact that $k(x,y)$ has exactly $3$ distinct real roots, the quantities $a(\varepsilon),\dots,\phi(x,y)$ are analytic on $T$.  
Moreover, the eigenvalues of $M(\varepsilon)$ are $x_0, x_1, x_2$ where, for all $k\in\{0,1,2\}$, 
\begin{align*}
    x_k(\varepsilon)
    &=
        A(\varepsilon)\cdot \cos\left(
            \dfrac{\phi + 2\pi\cdot k}{3}
        \right)
        + B(\varepsilon)
    .  
\end{align*}
Moreover, $x_0(\varepsilon),x_1(\varepsilon),x_2(\varepsilon)$ are analytic on $T$.  
For all $k,\ell\in \{1,2,3\}$, let 
\begin{align*}
    D(k,\ell,x)
    &:=
        \cos\left(
        x+\dfrac{2\pi k}{3}
    \right)
    - \cos\left(
        x+\dfrac{2\pi \ell}{3}
    \right)
\end{align*}
For all $(k,\ell)\in \{(0,1), (0,2), (2,1)\}$, note the trigonometric identities 
\begin{align*}
    D(k,\ell,x)
    &=
        \sqrt{3}\cdot\left\{\begin{array}{rl}
            \cos\left(
                x - \dfrac{\pi}{6}
            \right), 
                & (k,\ell) = (0, 1)
        \\
            \cos\left(
                x + \dfrac{\pi}{6}
            \right), 
                & (k,\ell) = (0, 2)
        \\
            \sin(x), 
                & (k,\ell) = (2, 1)
        \end{array}\right.
    .  
\end{align*}
By inspection, for all $x\in (0,\pi/3)$, 
\begin{align*}
    D(0,1)
    &>
        \max\left\{
            D(0,2), D(2,1)
        \right\}
    \geq 
    \min\left\{
            D(0,2), D(2,1)
        \right\}
    \geq 
    0.  
\end{align*}
Since $A > 0$ and $\phi \in (0,\pi/3)$, the claimed equality holds.  
Since $x_0(\varepsilon),x_1(\varepsilon)$ are analytic, $S(\varepsilon)$ is analytic on $T$.  
This completes the proof of Claim B.  
\\

Next, we compute the derivatives of  $S(\varepsilon)$ on $T$.  
For convenience, denote by $A_i, \phi_i, $ and $S_i$ for the partial derivatives of $A$ and $\phi$ by $\varepsilon_i$, respectively, for $i\in \{1,2\}$.  
Furthermore, let 
\begin{align*}
    \psi(\varepsilon)
    &:=
        \dfrac{2\phi(\varepsilon)-\pi}{6}.  
\end{align*}
The next claim follows directly from Claim B.  
\\
\\
{\bf Claim C: }
    For all $i\in T$, then on the set $T$, we have 
    \begin{align*}
        3S_i
        &= 
            3A_i\cdot \cos\left(
                \psi
            \right)
            - 
            \cdot A\phi_i \sin\left(
                \psi
            \right)
        .  
    \end{align*}
    Moreover, each expression is analytic on $T$.  

Finally, we solve $\SPR_{\{4,5,7\}}$.  
\\
\\
{\bf Claim D: }
    If $(\alpha_4,\alpha_5,\alpha_7)$ is a solution to $\SPR_{\{4,5,7\}}$, then $0\in\{\alpha_4,\alpha_5,\alpha_7\}$. 
    
With $(\alpha_4,\alpha_5,\alpha_7) := (1/3-\varepsilon_2, 2/3-\varepsilon_1,  \varepsilon_1+\varepsilon_2)$ and using the fact that $S$ is analytic on $T$, it is sufficient to eliminate all common zeroes of $S_1$ and $S_2$ on $T$.  
With the help of a computer algebra system and the formulas for $S_1$ and $S_2$ from Claim C, we replace the system $S_1 = 0$ and $S_2 = 0$ with a polynomial system of equations $P = 0$ and $Q = 0$ whose real solution set contains all previous solutions.  
Here, 
\begin{align*}
    P(\varepsilon) 
    &=
        9\varepsilon_1^3 + 18\varepsilon_1^2\varepsilon_2 + 54\varepsilon_1\varepsilon_2^2 + 18\varepsilon_2^3 - 15\varepsilon_1^2 - 33\varepsilon_1\varepsilon_2 - 27\varepsilon_2^2 + 5\varepsilon_1 + \varepsilon_2 
\end{align*}
and $Q = 43046721\varepsilon_1^{18}\varepsilon_2+\cdots + (-532480\varepsilon_2)$ is a polynomial of degree $19$, with coefficients between $-184862311457373$ and $192054273812559$.  
For brevity, we do not express $Q$ explicitly.  

To complete the proof of Claim D, it suffices to show that no common real solution to $P = Q = 0$ which lies in $T$ also satisfies $S_1 = S_2 = 0$.  
Again using a computer algebra system, we first find all common zeroes of $P$ and $Q$ on $\RR^2$.  
Included are the rational solutions $(2/3, -2/3), (-1/3, 1/3), (0,0), (2/3, 1/3), $ and $(2/3, -1/6)$ which do not lie in $T$.  
Furthermore, the solution $(1.2047\dots, 0.0707\dots)$ may also be eliminated.  
For the remaining $4$ zeroes, $S_1, S_2\neq 0$.  A notebook showing these calculations can be found at \cite{2021riasanovsky-spread}.

{\bf Claim E: }
    If $\mu, \nu$, and $\alpha = (\alpha_4,\alpha_5,\alpha_7)$ is part of a solution to $\SPR_{\{4,5,7\}}$ such that $\mu-\nu\geq 1$, then $\alpha_7 = 0$.  

By definition of $\SPR_{\{4,5,7\}}$, $\mu$ and $\nu$ are eigenvalues of the matrix 
\begin{align*}
    N(\alpha) 
    :=
        \left[\begin{array}{ccc}
            \alpha_5 & 0 & \alpha_4\\
            0 & 0 & \alpha_4\\
            \alpha_5 & \alpha_7 & 0
        \end{array}\right].
\end{align*}
Furthermore, $N(\alpha)$ has characteristic polynomial 
\begin{align*}
    p(x) 
    &= 
        x^3
        - \alpha_5 x^2
        -\alpha_4\cdot(\alpha_5+\alpha_7)
        +\alpha_4\alpha_5\alpha_7
    .  
\end{align*}
Recall that $\alpha_4+\alpha_5+\alpha_7 = 1$.  
By Claim D, $0\in\{4,5,7\}$, and it follows that $p\in\{p_4, p_5, p_7\}$ where 
\begin{align*}
    p_4(x) 
    &:=
        x^2\cdot (x-\alpha_5), 
    \\
    p_5(x) 
    &:=
        x\cdot (x^2-\alpha_4(1-\alpha_4)), 
    \text{ and }
    \\
    p_7(x) 
    &:=
        x\cdot (x^2-(1-\alpha_4)x-\alpha_4(1-\alpha_4)).  
\end{align*}
If $p = p_4$, then $\mu-\nu = \alpha_5\leq 1$, and if $p = p_5$, then $\mu-\nu = 2\sqrt{\alpha_4(1-\alpha_4)} \leq 1$.  
So $p = p_7$, which completes the proof of Claim E.  

This completes the proof of Lemma \ref{lem: SPR457}.

\subsection{Proof of Lemma \ref{lem: no exceptional vertices}} \label{sec: 2 by 2 reduction}
% and $\qq_n$ defined for all $n\in \mathbb{N}$ by 
% \begin{align*}
%     (\qq_{n}): 
%     \left\{\begin{array}{rl}
%         \max 
%         & 
%             S_z(2/3 - n_1/n, 1/3 - n_3/n)\\
%         \text{s.t}. 
%         & 
%             n_1,n_3 \in \mathbb{N}\\
%         &
%             n_1+n_3 \leq n\\
%         & 
%             0 \leq (2/3-n_1/n) + (1/3-n_3/n)
%     \end{array}\right.
%     \quad .  
% \end{align*}
% Note that $\qq_{n}$ is equivalent to the problem of maximizing $\spr(G(n_1,n_2,n_3))$ subject to the constraint that $n_1+n_2+n_3 = n$.  
% If $G = G(n_1,n_2,n_3)$ and 
% \begin{align*}
%     \left(
%         \dfrac{n_1}{n}, 
%         \dfrac{n_2}{n}, 
%         \dfrac{n_3}{n}
%     \right)
%     &=
%     \left(
%         \dfrac{2}{3}-\varepsilon_1,
%         \varepsilon_1+\varepsilon_2,
%         \dfrac{1}{3}-\varepsilon_2
%     \right), 
% \end{align*}
% then $n\cdot M_{n^{-1}}(\varepsilon_1,\varepsilon_2)$ is the ``reduced'' matrix for $A_G$ and thus they same maximum and minimum eigenvalues and the same spread.  
% To complete the proof, we solve $\pp_{z,C}$ for some small constant $C$ and all $z>0$ sufficiently small.  
% Appealing to concavity, we deduce that for all $n$ sufficiently large, $\qq$ has a unique solution $(n_1^*,n_3^*)$, and moreover, $n_1^*+n_3^* = n$.  

First, we find $S_z(\varepsilon_1,\varepsilon_3)$ using Vi\`{e}te's Formula.  
In doing so, we define functions  $k_z(\varepsilon_1,\varepsilon_2;x),\dots,\delta_z(\varepsilon_1,\varepsilon_2)$.  
To ease the burden on the reader, we suppress the subscript $z$ and the arguments $\varepsilon_1,\varepsilon_2$ when convenient and unambiguous.  
Let $k(x) = ax^3+bx^2+cx+d$ be the characteristic polynomial of $M_z(\varepsilon_1,\varepsilon_2)$.  
By inspection, 
\begin{align*}
    &a
    =
        1
    &b
    =
        \varepsilon_1+z-\dfrac{2}{3}
    \\
    &c
    =
        \dfrac{
            (3\varepsilon_2+2)
            (3\varepsilon_2-1)
        }
        {9}
    &d
    =
        \dfrac{
            (\varepsilon_1+\varepsilon_2)
            (3\varepsilon_1+3z-2)
            (3\varepsilon_2-1)
        }
        {9}
\end{align*}
Continuing, let 
\begin{align*}
    &p
    := 
        \dfrac{
            3a
                c
            -b^2
        }
        {
            3a^2
        }
    &q
    := 
        \dfrac{
            2b^3
            -9a
                b
                c
            +27a^2
                d
        }
        {
            27a^3
        }
    \\
    &A
    :=
        2\sqrt{
           \dfrac{-p}
           {3}
        }
    &B
    :=
        \dfrac{
            -b
        }
        {3a}
    \\
    &\phi
    :=
        \arccos\left(
            \dfrac{
                3q
            }
            {
                A
                p
            }
        \right).  
\end{align*}
By Vi\`{e}te's Formula, the roots of $k_z(\varepsilon_1,\varepsilon_2;x)$ are the suggestively defined quantities: 
\begin{align*}
    &\mu
    :=
        A
        \cos\left(
            \dfrac{
                \phi
            }
            {3}
        \right)
        +B
    &\nu
    :=
        A
        \cos\left(
            \dfrac{
                \phi
                +2\pi
            }
            {3}
        \right)
        +B
    \\
    \delta
    &:=
        A
        \cos\left(
            \dfrac{
                \phi
                +4\pi
            }
            {3}
        \right)
        +B
    .  
\end{align*}
First, We prove the following claim.  
\\
\\
{\bf Claim A: }
    If $(\varepsilon_1,\varepsilon_2,z)$ is sufficiently close to $(0,0,0)$, then 
    \begin{align}\label{eq: spread trig formula}
    S_z(\varepsilon_1,\varepsilon_2)
    &=
        A_z(\varepsilon_1,\varepsilon_2)\sqrt{3}\,
        \cdot\cos\left(
            \dfrac{2\phi_z(\varepsilon_1,\varepsilon_2)-\pi}{6}
        \right)
        .  
\end{align}
Indeed, suppose $z>0$ and $z\to 0$.  
Then for all $(\varepsilon_1,\varepsilon_2)\in (-3z,3z)$, $\varepsilon_1,\varepsilon_2\to 0$.  
With the help of a computer algebra system, we substitute in $z=0$ and $\varepsilon_1,\varepsilon_2=0$ to find the limits: 
\begin{align*}
    (a,b,c,d)
    &\to 
        \left(
            1,
            \dfrac{-2}{3},
            \dfrac{-2}{9},
            0
        \right)
    \\
    (p,q)
    &\to 
        \left(
            \dfrac{-10}{27},
            \dfrac{-52}{729}
        \right)
    \\
    (A,B,\phi)
    &\to 
        \left(
            \dfrac{2\sqrt{10}}{9},
            \dfrac{2}{9},
            \arccos\left(
                \dfrac{13\sqrt{10}}{50}
            \right)
        \right).  
\end{align*}
Using a computer algebra system, these substitutions imply that 
\begin{align*}
    (\mu,\nu,\delta)
    \to 
        \left(
            0.9107\dots, 
            -0.2440\dots, 
            0.  
        \right)
\end{align*}
So for all $z$ sufficiently small, $S = \mu-\nu$.  
After some trigonometric simplification, 
\begin{align*}
    \mu - \nu
    &=
        A\cdot\left(
            \cos\left(
                \dfrac{\phi}{3}
            \right)
            -\cos\left(
                \dfrac{\phi+2\phi}{3}
            \right)
        \right)
    = 
        A\sqrt{3}\,
        \cdot\cos\left(
            \dfrac{2\phi-\pi}{6}
        \right)
\end{align*}
and Equation \eqref{eq: spread trig formula}.  
This completes the proof of Claim A.  
\\

Now we prove the following claim.  
\\
\\
{\bf Claim B: }
    There exists a constants $C_0'>0$ such that the following holds.  
    If $|z|$ is sufficiently small, then $S_z$ is concave-down on $[-C_0,C_0]^2$ and strictly decreasing on $[-C_0,C_0]^2\setminus [-C_0z, C_0z]^2$.  
\\
\\
First, we define 
\begin{align*}
    D_z(\varepsilon_1,\varepsilon_2) 
    := \restr{\left(
        \dfrac{\partial^2 S_z}{\partial\varepsilon_1^2}
        \cdot \dfrac{\partial^2 S_z}{\partial\varepsilon_2^2}
        - \left(\dfrac{\partial^2 S_z}{\partial\varepsilon_1\partial\varepsilon_2}\right)^2
    \right)}
    {(\varepsilon_1,\varepsilon_2,z)}
    .  
\end{align*}
As a function of $(\varepsilon_1,\varepsilon_2)$, $D_z$ is the determinant of the Hessian matrix of $S_z$.  
Using a computer algebra system, we note that 
\begin{align*}
        D_0(0,0)
    &=
        22.5\dots, 
        \quad \text{ and }
    \\
    \restr{\left(
        \dfrac{
            \partial^2S
        }
        {
            \partial \varepsilon_1^2
        }, 
         \dfrac{
             \partial^2S
         }
         {
             \partial \varepsilon_1
             \partial \varepsilon_2
         }, 
        \dfrac{
            \partial^2S
        }
        {
            \partial \varepsilon_2^2
        }
    \right)
    }
    {(0,0,0)}
    &= 
    \left(
        -8.66\dots, 
        -8.66\dots, 
        -11.26\dots
    \right).  
\end{align*}
Since $S$ is analytic to $(0,0,0)$, there exist constants $C_1, C_2>0$ such that the following holds.  
For all $z\in [-C_1,C_1]$, $S_z$ is concave-down on $[-C_1,C_1]^2$.  
This completes the proof of the first claim.  
Moreover for all $z\in [-C_1,C_1]$ and for all $(\varepsilon_1,\varepsilon_2)\in [-C_1,C_1]^2$, 
\begin{align*}
        \restr{\max\left\{
            \dfrac{\partial^2S_z}
            {\partial\varepsilon_1^2},
            \dfrac{\partial^2S_z}
            {\partial\varepsilon_1\partial\varepsilon_2},
            \dfrac{\partial^2S_z}
            {\partial\varepsilon_2^2}
        \right\}}
        {(\varepsilon_1,\varepsilon_2,z)}
    &\leq -C_2.  
\end{align*}
to complete the proof of the second claim, note also that since $S$ is analytic at $(0,0,0)$, there exist constants $C_3,C_4>0$ such that for all $z\in [-C_3,C_3]$ and all $(\varepsilon_1,\varepsilon_2)\in [-C_3, C_3]^2$, 
\begin{align*}
    \dfrac{\partial^2 S}{\partial z\partial\varepsilon_i} \leq C_4.  
\end{align*}
Since $(0,0)$ is a local maximum of $S_0$, 
\begin{align*}
    \restr{
        \dfrac{\partial S}
        {\partial \varepsilon_i}
    }
    {(\varepsilon_1,\varepsilon_2,z)}
    &= 
        \restr{
        \dfrac{\partial S}
        {\partial \varepsilon_i}
    }
    {(0,0,0)}
        +\int_{w=0}^{z}
            \restr{
            \dfrac{\partial^2 S}
            {\partial z\partial \varepsilon_i}
            }
            {(0,0,w)}
        dw
    \\
    &\quad + \int_{ {\bf u} = (0,0)}^{ (\varepsilon_1,\varepsilon_2) }
        \restr{\dfrac{
            \partial^2 S
        }
        {
            \partial {\bf u}\partial \varepsilon_i
        }}
        {({\bf u}, z)}
        d{\bf u}
    \\
    &\leq 
        C_4\cdot z
        - C_2 \cdot \|(\varepsilon_1,\varepsilon_2)\|_2.  
\end{align*}
Since $C_2, C_4>0$, this completes the proof of Claim B.  
\\

Next, we prove the following claim.  \\
\\
{\bf Claim C: }
    If $z$ is sufficiently small, then $\pp_{z,C_0}$ is solved by a unique point $(\varepsilon_1^*,\varepsilon_2^*) = (\varepsilon_1^*(z),\varepsilon_2^*(z))$.  
    Moreover as $z\to0$, 
    \begin{align}\label{eq: optimal epsilon approximation}
        \left(
            \varepsilon_1^*, \varepsilon_2^*
        \right)
        &=
        \left(
            (1+o(z))\, 
            \dfrac{7z}{30}, 
            (1+o(z))\, 
            \dfrac{-z}{3}
        \right).  
    \end{align}
Indeed, the existence of a unique maximum $(\varepsilon_1^*, \varepsilon_2^*)$ on $[-C_0,C_0]^2$ follows from the fact that $S_z$ is strictly concave-down and bounded on $[-C_0, C_0]^2$ for all $z$ sufficiently small.   
Since $S_z$ is strictly decreasing on $[-C_0,C_0]^2\setminus (-C_0z, C_0z)^2$, it follows that $(\varepsilon_1^*, \varepsilon_2^*)\in (-C_0z, C_0z)$.  
For the second claim, note that since $S$ is analytic at $(0,0,0)$, 
\begin{align*}
    0 &= 
    \restr{
        \dfrac{\partial S}
        {\partial \varepsilon_i}
    }
    {(\varepsilon_1^*, \varepsilon_2^*, z)}
    =
    \sqrt{3}\cdot \left(
        \dfrac{\partial A}{\partial\varepsilon_i}\cdot
        \cos\left(
            \dfrac{2\phi-\pi}{6}
        \right)
        - \dfrac{A}{3}\cdot
        \dfrac{\partial\phi}{\partial\varepsilon_i}\cdot 
        \sin\left(
            \dfrac{2\phi-\pi}{6}
        \right)
    \right)
\end{align*}
for both $i = 1$ and $i = 2$.  
Let
\begin{align*}
    \tau_i := 
        \dfrac{
            3\cdot \dfrac{\partial A}{\partial\varepsilon_i}
        }
        {
            A\cdot \dfrac{\partial \phi}{\partial \varepsilon_i}
        }
\end{align*}
for both $i = 1$ and $i = 2$
Then by Equation \eqref{eq: spread trig formula}, 
\begin{align*}
    \restr{
        \arctan(\tau_i)
    }
    {(\varepsilon_1^*,\varepsilon_2^*,z)}
    &=
        \restr{
            \dfrac{2\phi-\pi}{6}
        }
        {(\varepsilon_1^*,\varepsilon_2^*,z)}
\end{align*}
for both $i=1$ and $i=2$.  
We first consider linear approximation of the above quantities under the limit  $(\varepsilon_1,\varepsilon_2,z)\to (0,0,0)$.  
Here, we write $f(\varepsilon_1,\varepsilon_2,z) \sim g(\varepsilon_1,\varepsilon_2,z)$ to mean that 
\begin{align*}
    f(\varepsilon_1,\varepsilon_2,z) 
    &=
        \left(
            1
            +o\left(\max\left\{
                |\varepsilon_1|,
                |\varepsilon_2|,
                |z|
            \right\}\right)
        \right)
        \cdot g(\varepsilon_1,\varepsilon_2,z).  
\end{align*}
With the help of a computer algebra system, we note that 
\begin{align*}
    \arctan\left(
        \tau_1
    \right)
    &\sim
        \dfrac{
            -78\varepsilon_1
            -96\varepsilon_2
            -3z
            -40\arctan\left(
                \dfrac{1}{3}
            \right)
        }
        {40}
    \\
    \arctan\left(
        \tau_2
    \right)
    &\sim
        \dfrac{
            -64\varepsilon_1
            -103\varepsilon_2
            -14z
            -20\arctan\left(
                \dfrac{1}{3}
            \right)
        }
        {20}
    \\
    \dfrac{2\phi-\pi}
    {6}
    &\sim 
       \dfrac{
            108\varepsilon_1
            +81\varepsilon_2
            +18z
            +20\arccos\left(
                \dfrac{13\sqrt{10}}{50}
            \right)
            -10\pi
        }
        {60}
    .  
\end{align*}
By inspection, the constant terms match due to the identity 
\begin{align*}
    -\arctan\left(
        \dfrac{1}{3}
    \right)
    &=
    \dfrac{1}{3}
    \arccos\left(
        \dfrac{13\sqrt{10}}{50}
    \right)
    -\dfrac{\pi}{6}.  
\end{align*}
Since $\max\left\{|\varepsilon_1^*|, |\varepsilon_2^*|\right\}\leq C_0z$, replacing $(\varepsilon_1,\varepsilon_2)$ with $(\varepsilon_1^*,\varepsilon_2^*)$ implies that 
\begin{align*}
    \dfrac{-78\varepsilon_1^*-96\varepsilon_2^*-3z}{2}
    &= (1+o(z))\cdot 
        (36\varepsilon_1^*+27\varepsilon_2^*+6z), \quad \text{ and }
    \\
        -64\varepsilon_1^*-103\varepsilon_2^*-14z
    &=
        (1+o(z))\cdot
        (36\varepsilon_1^*+27\varepsilon_2^*+6z)
\end{align*}
as $z\to 0$.  
After applying Gaussian Elimination to this $3$-variable system of $2$ equations, it follows that 
\begin{align*}
    (\varepsilon_1^*, \varepsilon_2^*) 
    &= 
        \left(
            (1+o(z))\cdot \dfrac{7z}{30}, 
            (1+o(z))\cdot \dfrac{-z}{3}
        \right).  
\end{align*}
This completes the proof of Claim C.  
\\
\\

For the next step, we prove the following claim.  
First, let $\qq_{n}$ denote the program formed from $\pp_{n^{-1}, C_0}$ subject to the added constraint that $\textstyle n\cdot( \frac{2}{3}-\varepsilon_1 ),n\cdot( \frac{1}{3}-\varepsilon_2 )\in\mathbb{Z}$.  
\\
\\
{\bf Claim D: }
    For all $n$ sufficiently large, $\qq_{n}$ is solved by a unique point $(n_1^*, n_3^*)$ which satisfies $n_1^* + n_3^* = n$.  
\\
\\

Note by Lemma \ref{lem: few exceptional vertices} that for all $n$ sufficiently large, 
\begin{align*}
    \max\left\{
        \left|\dfrac{n_1}{n}-\dfrac{2}{3}\right|,
        \left|\dfrac{n_3}{n}-\dfrac{1}{3}\right|
    \right\}
    &\leq C_0.  
\end{align*}
Moreover, by Claim C, $\pp_{n^{-1}}$ is solved uniquely by 
\begin{align*}
    (\varepsilon_1^*, \varepsilon_2^*) 
    &= 
        \left(
            (1+o(z))\cdot \dfrac{7}{30n}, 
            (1+o(z))\cdot \dfrac{-1}{3n}
        \right).  
\end{align*}
Since 
\begin{align*}
    \dfrac{2n}{3}-n\cdot \varepsilon_1^*
    &= 
        \dfrac{2n}{3}-(1+o(1))\cdot \dfrac{7}{30} 
\end{align*}
and $7/30 < 1/3$, it follows for $n$ sufficiently large that $2n/3-n\cdot\varepsilon_1^*\in I_1$ where 
\begin{align*}
    I_1 &:= 
    \left\{\begin{array}{rl}
        \left(
            \dfrac{2n}{3}
            -1, 
            \dfrac{2n}{3}
        \right), 
        & 
            3 \mid n 
        \\
        \left(
            \left\lfloor 
                \dfrac{2n}{3}
            \right\rfloor, 
            \left\lceil 
                \dfrac{2n}{3}
            \right\rceil
        \right), 
        & 
            3 \nmid n 
    \end{array}\right.
    .  
\end{align*}
Similarly since 
\begin{align*}
    n\cdot(\varepsilon_1^* + \varepsilon_2^*)
    &=
        (1+o(1))
        \cdot \left(
            \dfrac{7}{30}-\dfrac{1}{3}
        \right)
    = 
        (1+o(1))
        \cdot 
            \dfrac{-1}{10}
\end{align*}
and $1/10 < 1/3$, it follows that $n\cdot (\varepsilon_1^*+\varepsilon_2^*)\in (-1,0)$.  
Altogether, 
\begin{align*}
    \left(
        \dfrac{2n}{3}-n\cdot \varepsilon_1, 
        n\cdot(\varepsilon_1^*+\varepsilon_2^*)
    \right)
    &\in I_1\times (-1, 0).  
\end{align*}

Note that to solve $\qq_{n}$, it is sufficient to maximize $S_{n^{-1}}$ on the set $[-C_0, C_0]^2 \cap \{ (n_1/n, n_3/n) \}_{u,v\in\mathbb{N}}$.  
Since $S_{n^{-1}}$ is concave-down on $I_1\times (-1, 0)$, $(n_1^*, n-n_1^*-n_3^*)$ is a corner of the square $I_1\times (-1,0)$.  
So $n_1^*+n_2^* = n$, which implies Claim D.  
This completes the proof of the main result.

\section{A Computer-Assisted Proof of Lemma \ref{lem: 2 feasible sets}}\label{sec: appendix}

In this appendix, we derive a number of formulas that a stepgraphon corresponding to some set $S \subseteq \{1,2,3,4,5,6,7\}$ in Lemma \ref{lem: 19 cases} satisfies, and detail how these formulas are used to provide a computer-assisted proof of Lemma \ref{lem: 2 feasible sets}.

\subsection{Formulas}\label{sub-sec: formulas}

In this subsection, we derive the formulas used in our computer-assisted proof, from the equations described in Section \ref{appendix 17 cases}. 
First, we define a number of functions which will ease the notational burden in the results that follow. Let
\begin{align*}
    F_1(x) &:= (\mu+\nu)x + 2 \mu \nu, 
    \\
    F_2(x) &:= 2 ( \mu \nu + (\mu+\nu) x)^2 + (\mu+\nu) x^3, 
    \\
    F_3(x) &:= 
        4\mu^2\nu^2\cdot (\mu\nu + (\mu+\nu)x)^2 
        \\
        &\quad - 2(\mu+\nu)x^3\cdot ((\mu+\nu)x + \mu\nu)((\mu+\nu)x+3\mu\nu)
        \\
        &\quad - (\mu+\nu)x^5\cdot (2\mu\nu + (\mu+\nu)x), 
    \\
    F_4(x) &:=
        4\mu^2\nu^2x\cdot ((3(\mu+\nu)x + \mu\nu)\cdot (2(\mu+\nu)x+\mu\nu)-\mu\nu(\mu+\nu)x)
        \\
        &\quad +4(\mu+\nu)x^4\cdot (((\mu+\nu)x+\mu\nu)^2 + (\mu+\nu)^2\cdot ((\mu+\nu)x+4\mu\nu))
        \\
        &\quad + (\mu+\nu)^2x^7.  
\end{align*}
Letting  $S := \{i\in\{1,\dots,7\} : \alpha_i>0\}$, we prove the following six formulas.

\begin{proposition}\label{prop: fg23 assume 23}
	Let $i\in \{1,2,5\} \cap S$ and $j\in \{3,4,6,7\} \cap S$ be such that $N_i \cap S = (N_j\cap S)\dot{\cup}\{j\}$. Then
	\begin{align*}
		f_j^2 
			&= \dfrac{(\alpha_j+2\nu)\mu}{F_1(\alpha_j)}, 
		\quad \quad g_j^2 
			= \dfrac{(\alpha_j+2\mu)\nu}{F_1(\alpha_j)}
	\end{align*}
	and
	\begin{align*}
	    f_i &= \left( 1 + \frac{\alpha_j}{\mu} \right) f_j,
		 \quad \quad g_i = \left( 1 + \frac{\alpha_j}{\nu} \right) g_j.
	\end{align*}
	Moreover, $F_1(\alpha_j)$ and $\alpha_j+2\nu$ are negative.  
\end{proposition}
\begin{proof}
	By Lemma \ref{lem: local eigenfunction equation}, 
	\begin{align*}
		\mu f_i^2 - \nu g_i^2 &= \mu-\nu\\
		\mu f_j^2 - \nu g_j^2 &= \mu-\nu.  
	\end{align*}
By taking the difference of the eigenvector equations for $f_i$ and $f_j$ (and also $g_i$ and $g_j$), we obtain
	\begin{align*}
		\alpha_j f_j &= \mu(f_i - f_j)\\
		\alpha_j g_j &= \nu(g_i - g_j),
	\end{align*}
or, equivalently,
\begin{align*}
		f_i &= \left( 1 + \frac{\alpha_j}{\mu} \right) f_j\\
		g_i &= \left( 1 + \frac{\alpha_j}{\nu} \right) g_j.
	\end{align*}
	This leads to the system of equations 
	\begin{align*}
		\left[\begin{array}{cc}
			\mu 
				& -\nu\\
			\mu\cdot \left(1+\dfrac{\alpha_j}{\mu}\right)^2 
				& -\nu\cdot \left(1+\dfrac{\alpha_j}{\nu}\right)^2
		\end{array}\right]
		\cdot
		\left[
			\begin{array}{c}
				f_j^2\\
				g_j^2
			\end{array}
		\right]
		=
		\left[\begin{array}{c}
			\mu-\nu\\
			\mu-\nu
		\end{array}\right].  
	\end{align*}
	If the corresponding matrix is invertible, then after  substituting the claimed formulas for $f_j^2,g_j^2$ and simplifying, it follows that they are the unique solutions.  
	To verify that $F_1(\alpha_j)$ and $\alpha_j+2\nu$ are negative, it is sufficient to inspect the formulas for $f_j$ and $g_j$, noting that $\nu$ is negative and both $\mu$ and $\alpha_j+2\mu$ are positive.  
	\\
	
	Suppose the matrix is not invertible. By assumption $\mu, \nu \ne 0$, and so 
	$$\bigg(	1+\frac{\alpha_j}{\mu}\bigg)^2 = \bigg(	1+\frac{\alpha_j}{\nu}\bigg)^2.$$ 
	But, since $i\in \{1,2,5\}$ and $j\in \{3,4,6,7\}$, 
	\begin{align*}
		1 > f_j^2g_i^2 
		= 
			f_i^2g_i^2\cdot \left(
				1+\dfrac{\alpha_j}{\mu}
			\right)^2 
		= 
			f_i^2g_i^2\cdot \left(
				1+\dfrac{\alpha_j}{\nu}
			\right)^2
		= 
			f_i^2g_j^2 > 1, 
	\end{align*}
	a contradiction.  
\end{proof}

\begin{proposition}\label{prop: fg24 assume 2N4}
	Let $i\in \{1,2,5\} \cap S$ and $j\in \{3,4,6,7\} \cap S$ be such that $N_i \cap S = (N_j\cap S)\dot{\cup}\{i\}$. Then
		\begin{align*}
		f_i^2 
			&= \dfrac{(\alpha_i-2\nu)\mu}
			{-F_1(-\alpha_i)},
		\quad \quad 
		g_i^2 
			= \dfrac{(\alpha_i-2\mu)\nu}
			{-F_1(-\alpha_i)},
	\end{align*}
	and
			\begin{align*}
		f_j 
			&= \left( 1- \frac{\alpha_i}{\mu} \right) f_i,
		\quad \quad 
		g_j 
			= \left( 1- \frac{\alpha_i}{\nu} \right) g_i.
	\end{align*}
	Moreover, $-F_1(-\alpha_i)$ is positive and $\alpha_i-2\mu$ is negative.  
\end{proposition}

\begin{proof}
    The proof of Proposition \ref{prop: fg23 assume 23}, slightly modified, gives the desired result.
\end{proof}

\begin{proposition}\label{prop: a2 assume 234}
	Suppose $i,j,k\in S$ where $(i,j,k)$ is either $(2,3,4)$ or $(5,6,7)$. Then 
	\begin{align*}
        f_k &= \dfrac{ \mu f_j-\alpha_if_i }{\mu}, \quad \quad 
        g_k = \dfrac{ \nu g_j-\alpha_ig_i }{\nu},  
    \end{align*}
    and
	$$\alpha_i = \frac{2 \mu^2 \nu^2 \alpha_j}{ F_2(\alpha_j)}.$$
\end{proposition}
\begin{proof}
	%Taking the difference of the eigenvector equations for $j$ and $k$, we obtain $\mu(f_j-f_k) = \alpha_i f_i$ and $\nu(g_j-g_k) = \alpha_i g_i$. Solving these equations for $f_k$ and $g_k$, and substituting these solutions into the equation $\mu f_k^2 - \nu g_k^2 = \mu - \nu $, leads to the equation
    
    Using the eigenfunction equations for $f_j,f_k$ and for $g_j,g_k$, it follows that 
    \begin{align*}
        f_k &= \dfrac{ \mu f_j-\alpha_if_i }{\mu}, \quad \quad 
        g_k = \dfrac{ \nu g_j-\alpha_ig_i }{\nu}.  
    \end{align*}
    Combined with Lemma \ref{lem: local eigenfunction equation}, it follows that 
    \begin{align*}
        0 &= \mu f_k^2 - \nu g_k^2 - (\mu-\nu) \\
        &=
            \mu\left(
               \dfrac{ \mu f_j-\alpha_if_i }{\mu}
            \right)^2
            -\nu\left(
                 \dfrac{ \nu g_j-\alpha_ig_i }{\nu}
            \right)^2
            - (\mu-\nu).  
    \end{align*}
    After expanding, we note that the right-hand side can be expressed purely in terms of $\mu, \nu, \alpha_i, f_i^2, f_if_j, f_j^2, g_i^2, g_ig_j,$ and $g_j^2$.  
    Note that Proposition \ref{prop: fg23 assume 23} gives explicit formulas for $f_i^2, f_if_j$, and $f_j^2$, as well as $g_i^2, g_ig_j$, and $g_j^2$, purely in terms of $\mu, \nu$, and $\alpha_j$.  
    With the help of a computer algebra system, we make these substitutions and factor the right-hand side as: 
    \begin{align*}
        0 &= 
            (\mu-\nu)\cdot \alpha_i
            \cdot \dfrac{
                2\mu^2\nu^2\cdot \alpha_j
                - F_2(\alpha_j)\cdot \alpha_i
            }
            {
                \mu^2\nu^2 \cdot F_1(\alpha_i)
            }
        .  
    \end{align*}
    Since $\alpha_i, (\mu-\nu) \neq 0$, the desired claim holds.  
\end{proof}

\begin{proposition}\label{prop: a4 assume 1234}
	Suppose $1,i,j,k \in S$ where $(i,j,k)$ is either $(2,3,4)$ or $(5,6,7)$.  
	Then 
	\begin{align*}
        f_1 &= \dfrac{\mu f_i + \alpha_k f_k}{\mu}, \quad\quad g_1 = \dfrac{\nu g_i + \alpha_k g_k}{\nu}, 
    \end{align*}
    and
	\begin{align*}
	    \alpha_k 
	    &= 
	        \dfrac{
	            \alpha_j\cdot F_2(\alpha_j)^2
	        }
	        {
	            F_3(\alpha_j)
	        }.
	\end{align*}
	% old stuff: 
    % 	Then $\alpha_k = N_{\ref{prop: a4 assume 1234}}(\alpha_j,\mu\nu,\mu+\nu)/D_{\ref{prop: a4 assume 1234}}(\alpha_j,\mu\nu,\mu+\nu)$ where 
    % 	\begin{align*}
    % 		N_{\ref{prop: a4 assume 1234}}(x,P,V)
    % 		&= 
    % 			x\cdot (2\cdot (P + xV)^2 + x^3V)^2\\
    % 		D_{\ref{prop: a4 assume 1234}}(x,P,V) 
    % 		&= 
    % 			4P^2\cdot(P + xV)^2 \\
    % 		&\quad 
    % 			- 2Vx^3\cdot (xV + P)(xV + 3P) \\
    % 		& \quad 
    % 			- Vx^5\cdot ( 2P + xV ).  
    % 	\end{align*}
\end{proposition}
\begin{proof}
    Using the eigenfunction equations for $f_1, f_i, f_j, f_k$ and for $g_1, g_i, g_j, g_k$, it follows that 
    \begin{align*}
        f_1 &= \dfrac{\mu f_i + \alpha_k f_k}{\mu}, \quad\quad g_1 = \dfrac{\nu g_i + \alpha_k g_k}{\nu}, 
    \end{align*}
    and 
    \begin{align*}
        f_k &= \dfrac{ \mu f_j-\alpha_if_i }{\mu}, \quad \quad 
        g_k = \dfrac{ \nu g_j-\alpha_ig_i }{\nu}.  
    \end{align*}
    Altogether, 
    \begin{align*}
        f_1 &= 
            \dfrac{\mu^2 f_i + \alpha_k( \mu f_j - \alpha_if_i )}{\mu^2}, 
        \quad\quad 
        g_1 &= 
            \dfrac{\nu^2 g_i + \alpha_k( \nu g_j - \alpha_ig_i )}{\nu^2}
    \end{align*}
    Combined with Lemma 3.6, it follows that 
    \begin{align*}
        0 &= \mu f_1^2-\nu g_1^2 - (\mu-\nu)
        \\
        &=
            \mu\left( \dfrac{\mu^2 f_i + \alpha_k( \mu f_j - \alpha_if_i )}{\mu^2} \right)^2
            -\nu\left( \dfrac{\nu^2 g_i + \alpha_k( \nu g_j - \alpha_ig_i )}{\nu^2} \right)^2
            -(\mu-\nu).  
    \end{align*}
    After expanding, we note that the right-hand side can be expressed purely in terms of $\mu,\nu, f_i^2, f_if_j, f_j^2, g_i^2, g_ig_j,$ and $\alpha_i$.  
    Note that Proposition \ref{prop: fg23 assume 23} gives explicit formulas for $f_i^2, f_if_j, f_j^2, g_i^2, g_ig_j,$ and $g_j^2$ purely in terms of $\mu, \nu$, and $\alpha_j$.  
	With the help of a computer algebra system, we make these substitutions and factor the right-hand side as: 
	\begin{align*}
	    0 
	    &= 
	        2\alpha_k\cdot (\mu-\nu)\cdot
	        \dfrac{
	            \alpha_j\cdot F_2(\alpha_j)^2 - \alpha_k\cdot F_3(\alpha_j)
	        }
	        {
	            F_1(\alpha_j)\cdot F_2(\alpha_j)^2
	        }
        .
 	\end{align*}
 	So the desired claim holds.  
% 	\red{After all the substitutions, we see that 
% 	\begin{align*}
% 		0 
% 		&= 
% 			\mu f_1^2-\nu g_1^2-(\mu-\nu)\\
% 		&= 
% 			\dfrac{
% 				2(\mu-\nu)\cdot\alpha_k
% 				\cdot(N_{\ref{prop: a4 assume 1234}}(\alpha_j,\mu\nu,\mu+\nu)-D_{\ref{prop: a4 assume 1234}}(\alpha_j,\mu\nu,\mu+\nu)\cdot\alpha_k)
% 			}
% 			{D_{\ref{prop: fg23 assume 23}}(\alpha_j, \mu\nu, \mu+\nu)\cdot D_{\ref{prop: a2 assume 234}}(\alpha_j,\mu\nu,\mu+\nu)^2}.  
% 	\end{align*}}
\end{proof}

\begin{proposition}\label{prop: a4 assume 12N4}
	Suppose $1,i,k\in S$ and $j\notin S$ where $(i,j,k)$ is either $(2,3,4)$ or $(5,6,7)$.  
	Then, 
	\begin{align*}
        f_1 &= \dfrac{\mu f_i + \alpha_k f_k}{\mu}, \quad\quad g_1 = \dfrac{\nu g_i + \alpha_k g_k}{\nu}, 
    \end{align*}
    and
	\begin{align*}
		\alpha_k
			&= 
			\dfrac{2\alpha_i \mu^2\nu^2}
			{
				F_2(-\alpha_i)}
	\end{align*}
\end{proposition}
\begin{proof}
	Using the eigenfunction equations for $f_1, f_i, f_j, f_k$ and for $g_1, g_i, g_j, g_k$, it follows that 
    \begin{align*}
        f_1 &= \dfrac{\mu f_i + \alpha_k f_k}{\mu}, \quad\quad g_1 = \dfrac{\nu g_i + \alpha_k g_k}{\nu}, 
    \end{align*}
    and 
    \begin{align*}
        f_k &= \dfrac{ \mu f_i-\alpha_if_i }{\mu}, \quad \quad 
        g_k = \dfrac{ \nu g_i-\alpha_ig_i }{\nu}.  
    \end{align*}
    Altogether, 
    \begin{align*}
        f_1 &= 
            \dfrac{\mu^2 f_i + \alpha_k( \mu f_i - \alpha_if_i )}{\mu^2}, 
        \quad\quad 
        g_1 &= 
            \dfrac{\nu^2 g_i + \alpha_k( \nu g_i - \alpha_ig_i )}{\nu^2}
    \end{align*}
    Combined with Lemma 3.6, it follows that 
    \begin{align*}
        0 &= \mu f_1^2-\nu g_1^2 - (\mu-\nu)
        \\
        &=
            \mu\left( \dfrac{\mu^2 f_i + \alpha_k( \mu f_i - \alpha_if_i )}{\mu^2} \right)^2
            -\nu\left( \dfrac{\nu^2 g_i + \alpha_k( \nu g_i - \alpha_if_i )}{\nu^2} \right)^2
            -(\mu-\nu).  
    \end{align*}
    After expanding, we note that the right-hand side can be expressed purely in terms of $\mu,\nu, f_i^2, f_if_j, f_j^2, g_i^2, g_ig_j,$ and $\alpha_i$.  
    Note that Proposition \ref{prop: fg23 assume 23} gives explicit formulas for $f_i^2, f_if_j, f_j^2, g_i^2, g_ig_j,$ and $g_j^2$ purely in terms of $\mu, \nu$, and $\alpha_j$.  
	With the help of a computer algebra system, we make these substitutions and factor the right-hand side as: 
	\begin{align*}
	    0 
	    &= 
	        2\alpha_k\cdot (\mu-\nu)\cdot
	        \dfrac{
	            \alpha_j\cdot F_2(\alpha_j)^2 - \alpha_k\cdot F_3(\alpha_j)
	        }
	        {
	            F_1(\alpha_j)\cdot F_2(\alpha_j)^2
	        }
        .
 	\end{align*}
 	So the desired claim holds.  
\end{proof}

\begin{proposition}\label{prop: a4 assume N2347}
	Suppose $1\notin S$ and $i,j,k,\ell\in S$ where $(i,j,k,\ell)$ is either $(2,3,4,7)$ or $(5,6,7,4)$.  
	Then 
	\begin{align*}
	    \alpha_k &= \dfrac{F_4(x)}{F_3(x)}.  
	\end{align*}
% 	\red{
%     	Then $\alpha_k = N_{\ref{prop: a4 assume N2347}}(\alpha_j, \mu\nu, \mu+\nu) / D_{\ref{prop: a4 assume N2347}}(\alpha_j, \mu\nu, \mu+\nu)$ where 
%     	\begin{align*}
%     		N(x, P, V)
%     			&= 
%     				4xP^2\cdot ((3xV + P)\cdot(2xV + P) - xPV)\\
%     			&\quad 
%     				+4x^4V\cdot ((P + xV)^2 + V^2\cdot (4P + xV))\\
%     			&\quad 
%     				+x^7V^2\\
%     		D(x, P, V)
%     		&=
%     			4P^2\cdot(P + xV)^2 \\
%     		&\quad 
%     			- 2Vx^3\cdot (xV + P)(xV + 3P)\\
%     		&\quad 
%     			- Vx^5\cdot ( 2P + xV ).  	
%     	\end{align*}
    %}
\end{proposition}
\begin{proof}
    Using the eigenfunction equations for $f_\ell, f_i, f_j, f_k$ and for $g_\ell, g_i, g_j, g_k$, it follows that 
    \begin{align*}
        f_\ell 
        &= \dfrac{\alpha_if_i + \alpha_jf_j + \alpha_k f_k}{\mu}, 
        \quad\quad 
        g_1 = \dfrac{\alpha_ig_i + \alpha_jg_j + \alpha_k g_k}{\nu}, 
    \end{align*}
    and 
    \begin{align*}
        f_k &= \dfrac{ \mu f_j-\alpha_if_i }{\mu}, \quad \quad 
        g_k = \dfrac{ \nu g_j-\alpha_ig_i }{\nu}.  
    \end{align*}
    Altogether, 
    \begin{align*}
        f_\ell 
        &= 
            \dfrac{ \mu \alpha_if_i + \alpha_jf_j + \alpha_k (\mu f_j-\alpha_if_i) }{\mu^2}, 
        \quad\quad 
        g_\ell 
        &= 
            \dfrac{ \nu \alpha_ig_i + \alpha_jg_j + \alpha_k (\nu g_j-\alpha_ig_i) }{\nu^2}
    \end{align*}
    Combined with Lemma \ref{lem: local eigenfunction equation}, it follows that 
    \begin{align*}
        0 &= \mu f_\ell^2 - \nu g_\ell^2 - (\mu-\nu)\\
        &= 
            \mu\left(
                \dfrac{
                    \mu \alpha_if_i + \alpha_jf_j + \alpha_k (\mu f_j-\alpha_if_i)
                }{\mu^2}
            \right)^2-\nu\left(
                \dfrac{
                    \nu \alpha_ig_i + \alpha_jg_j + \alpha_k (\nu g_j-\alpha_ig_i)
                }{\nu^2}
            \right)^2\\
            &\quad -(\mu-\nu)
    \end{align*}
    
    After expanding, we note that the right-hand side can be expressed purely in terms of $\mu,\nu, f_i^2, f_if_j, f_j^2, g_i^2, g_ig_j,$ and $\alpha_i$.  
    Note that Proposition \ref{prop: fg23 assume 23} gives explicit formulas for $f_i^2, f_if_j, f_j^2, g_i^2, g_ig_j,g_j^2,\alpha_i,\alpha_j,\alpha_k$ purely in terms of $\mu, \nu$, and $\alpha_j$.  
	With the help of a computer algebra system, we make these substitutions and factor the right-hand side as: 
    
    \begin{align*}
        0
        &= 
            2(\mu-\nu)\cdot\alpha_k\cdot 
            \dfrac{
                F_4(\alpha_j)-\alpha_k\cdot F_3(\alpha_j)
            }
            {F_1(\alpha_j)\cdot F_2(\alpha_j)^2}
    \end{align*}

    % \red{
    % 	After all the substitutions, we see that 
    % 	\begin{align*}
    % 		0 
    % 		&= \mu f_\ell^2-\nu g_\ell^2 - (\mu-\nu)\\
    % 		&=
    % 			\dfrac{
    % 				2(\mu-\nu)\cdot \alpha_k\cdot 
    % 				(N_{\ref{prop: a4 assume N2347}}(\alpha_j,\mu\nu,\mu+\nu)-\alpha_j\cdot D_{\ref{prop: a4 assume N2347}}(\alpha_j,\mu\nu,\mu+\nu)
    % 			}
    % 			{
    % 				D_{\ref{prop: fg23 assume 23}}(\alpha_j,\mu\nu,\mu+\nu)\cdot D_{\ref{prop: a2 assume 234}}(\alpha_j,\mu\nu,\mu+\nu)^2
    % 			}
    % 	\end{align*}
    % 	}
\end{proof}

\begin{proposition}\label{prop: a47 assume 2457}
	Suppose $2,4,5,7\in S$ and let $\alpha_{\ne 4,7}:=\sum_{\substack{i\in S,\\ i \ne 4,7}} \alpha_i$. Then 
	\begin{align*}
	    \alpha_4 &= \frac{ (1-\alpha_{\ne 4,7}) f_7 - \mu (f_2 - f_7)}{f_4 +f_7}, \\
	    \alpha_7 &= \frac{ (1-\alpha_{\ne 4,7}) f_4 - \mu (f_5 - f_2)}{f_4 +f_7},
	\end{align*}
	and
		\begin{align*}
	    \alpha_4 &= \frac{ ((1-\alpha_{\ne 4,7}) g_7 - \nu (g_2 - g_7)}{g_4 +g_7}, \\
	    \alpha_7 &= \frac{ (1-\alpha_{\ne 4,7}) g_4 - \nu (g_5 - g_2)}{g_4 +g_7}.
	\end{align*}
\end{proposition}

\begin{proof}
Taking the difference of the eigenvector equations for $f_2$ and $f_5$, and for $g_2$ and $g_5$, we have
$$\alpha_7 f_7 -\alpha_4 f_4 = \mu(f_2 - f_5), \qquad \alpha_7 g_7 -\alpha_4 g_4 = \nu(g_2 - g_5).$$
Combining these equalities with the equation $\alpha_4 + \alpha_7 = 1-\alpha_{\ne 4,7}$ completes the proof.
\end{proof}

\subsection{Algorithm}

In this subsection, we briefly detail how the computer-assisted proof of Lemma \ref{lem: 2 feasible sets} works. This proof is via interval arithmetic, and, at a high level, consists largely of iteratively decomposing the domain of feasible choices of $(\alpha_3,\alpha_6,\mu,\nu)$ for a given $S$ into smaller subregions (boxes) until all subregions violate some required equality or inequality. We provide two similar, but slightly different computer assisted proofs of this fact, and both of which can be found at the spread\_numeric GitHub repository \cite{2021riasanovsky-spread}. The first, found in folder interval$1$, is a shorter and simpler version, containing slightly fewer formulas, albeit at the cost of overall computation and run time. The second, found in the folder interval$2$, contains slightly more formulas and makes a greater attempt to optimize computation and run time. Below, we further detail the exact output and run time of both versions (exact output can be found in \cite{2021riasanovsky-spread}), but for now, we focus on the main aspects of both proofs, and consider both together, saving a more detailed discussion of the differences for later.

These algorithms are implemented in Python using the PyInterval package. The algorithms consists of two parts: a main file containing useful formulas and subroutines and $17$ different files used to rule out each of the $17$ cases for $S$. The main file, casework\_helper, contains functions with the formulas of Appendix Subsection \ref{sub-sec: formulas} (suitably modified to limit error growth), and functions used to check that certain equalities and inequalities are satisfied. In particular,  casework\_helper contains formulas for 
\begin{itemize}
	\item $\alpha_2$, assuming $\{2,3,4\} \subset S$ (using Proposition \ref{prop: a2 assume 234})
	\item $\alpha_4$, assuming $\{1,2,3,4\} \subset S$ (using Proposition \ref{prop: a4 assume 1234})
	\item $\alpha_4$, assuming $\{2,3,4,7\} \subset S$, $1 \not \in S$ (using Proposition \ref{prop: a4 assume N2347})
	\item $\alpha_4$, assuming $\{1,2,4\}\subset S$, $3 \not \in S$ (using Proposition \ref{prop: a4 assume 12N4})
	\item $f_3$ and $g_3$, assuming $\{2,3\}\subset S$ (using Proposition \ref{prop: fg23 assume 23})
	\item $f_2$ and $g_2$, assuming $\{2,3\} \subset S$ (using Proposition \ref{prop: fg23 assume 23})
	\item $f_4$ and $g_4$, assuming $\{2,3,4\} \subset S$ (using Proposition \ref{prop: a2 assume 234})
	\item $f_1$ and $g_1$, assuming $\{1,2,4\} \subset S$ (using Propositions \ref{prop: a4 assume 1234} and \ref{prop: a4 assume 12N4})
	\item $f_2$ and $g_2$, assuming $\{2,4\} \subset S$, $3 \not\in S$ (using Proposition \ref{prop: fg24 assume 2N4})
	\item $f_4$ and $g_4$, assuming $\{2,4\} \subset S$, $3 \not\in S$ (using Proposition \ref{prop: fg24 assume 2N4})
\end{itemize}
as a function of $\alpha_3$, $\mu$, and $\nu$ (and $\alpha_2$ and $\alpha_4$, which can be computed as functions of $\alpha_3$, $\mu$, and $\nu$). Some of the formulas are slightly modified compared to their counterparts in this Appendix, for the purpose of minimizing accumulated error. Each formula is performed using interval arithmetic, while restricting the resulting interval solution to the correct range. In addition, we recall that we have the inequalities
\begin{itemize}
    \item $\alpha_i \in [0,1]$, for $i \in S$
    \item $|g_2|,|f_3|\le 1$, $|f_2|,|g_3| \ge 1$, for $\{2,3\}\subset S$
    \item $|f_4|\le 1$, $|g_4|\ge 1$, for $4 \in S$
    \item $|f_1|\ge 1$, $|g_1|\le 1$, for $\{1,2,4\}\in S$
    \item $|f_4|,|g_2| \le 1$, $|f_2|,|g_4|\ge 1$, for $\{2,4\} \in S$, $3 \not \in S$
    \item $\alpha_3 + 2 \nu \le 0$, for $\{2,3\} \in S$ (using Proposition \ref{prop: fg23 assume 23})
    \item $\alpha_2 - 2 \mu \le 0$, for $\{2,4\} \in S$, $3 \not \in S$ (using Proposition \ref{prop: fg24 assume 2N4}).
\end{itemize}
These inequalities are also used at various points in the algorithms. This completes a brief overview of the casework\_helper file. Next, we consider the different files used to test feasibility for a specific choice of $S \subset \{1,...,7\}$, each denoted by case$\{\text{elements of S}\}$, i.e., for $S = \{1,4,5,7\}$, the associated file is case$1457$. For each specific case, there are a number of different properties which can be checked, including eigenvector equations, bounds on edge density, norm equations for the eigenvectors, and the ellipse equations. Each of these properties has an associated function which returns FALSE, if the property cannot be satisfied, given the intervals for each variable, and returns TRUE otherwise. The implementation of each of these properties is rather intuitive, and we refer the reader to the programs themselves (which contain comments) for exact details \cite{2021riasanovsky-spread}. Each feasibility file consists of two parts. The first part is a function is\_feasible(mu,nu,a3,a6) that, given bounding intervals for $\mu$, $\nu$, $\alpha_3$, $\alpha_6$, computes intervals for all other variables (using interval arithmetic) and checks feasibility using the functions in the casework\_helper file. If any checked equation or inequality in the file is proven to be unsatisfiable (i.e., see Example \ref{ex: infeasible box}), then this function outputs `FALSE', otherwise the function outputs `TRUE' by default. The second part is a divide and conquer algorithm that breaks the hypercube
$$ (\mu, \nu, \alpha_3,\alpha_6) \in [.65,1] \times [-.5,-.15] \times [0,1] \times [0,1]$$
into sub-boxes of size $1/20$ by $1/20$ by $1/10$ by $1/10$, checks feasibility in each box using is\_feasible, and subdivides any box that does not rule out feasibility (i.e., subdivides any box that returns `TRUE'). This subdivision breaks a single box into two boxes of equal size, by subdividing along one of the four variables. The variable used for this subdivision is chosen iteratively, in the order $\alpha_3,\alpha_6,\mu, \nu, \alpha_3,...$. The entire divide and conquer algorithm terminates after all sub-boxes, and therefore, the entire domain
$$ (\mu, \nu, \alpha_3,\alpha_6) \in [.65,1] \times [-.5,-.15] \times [0,1] \times [0,1],$$
has been shown to be infeasible, at which point the algorithm prints `infeasible'. Alternatively, if the number of subdivisions reaches some threshold, then the algorithm terminates and outputs `feasible'.

Next, we briefly detail the output of the algorithms casework\_helper/intervals$1$ and casework\_helper/intervals$2$. Both algorithms ruled out 15 of the 17 choices for $S$ using a maximum depth of 26, and failed to rule out cases $S = \{4,5,7\}$ and $S = \{1,7\}$ up to depth 51. For the remaining 15 cases, intervals$1$ considered a total of 5.5 million boxes, was run serially on a personal computer, and terminated in slightly over twelve hours. For these same 15 cases, intervals$2$ considered a total of 1.3 million boxes, was run in parallel using the Penn State math department's `mathcalc' computer, and terminated in under 140 minutes. The exact output for both versions of the spread\_numeric algorithm can be found at \cite{2021riasanovsky-spread}.

\end{document}